\documentclass[10pt]{amsart}
\usepackage[T1]{fontenc}
\usepackage[utf8]{inputenc}

\usepackage{amsfonts, amsmath, amssymb, amsthm, array, bm, braket,
  caption, changepage, dsfont, enumerate, esint, graphicx, mathrsfs, mathtools, textcomp}

\usepackage{color}
\usepackage[colorlinks=true]{hyperref}

\usepackage{cleveref}
\usepackage{mathscinet}

\newcommand{\eqdef}{\overset{\text{\tiny def}}{=}}
\newcommand{\eqdist}{\overset{{\text{\tiny dist}}}{=}}

\newcommand{\one}{\mathbf{1}}
\newcommand{\supp}{\text{supp}\,}
\newcommand{\lpara}{\prec}
\newcommand{\rpara}{\succ}
\newcommand{\res}{\circ}

\newcommand{\Ss}{\mathscr{S}}
\newcommand{\T}{\mathbb{T}}
\newcommand{\X}{Y} 
\newcommand{\XX}{\mathbf{X}}
\newcommand{\Y}{X}
\newcommand{\RY}{R}
\newcommand{\RX}{S}
\newcommand{\Z}{\mathbb{Z}}
\newcommand{\R}{\mathbb{R}}

\newcommand{\Nb}{\mathbb{N}}

\newcommand{\Pc}{\mathcal{P}}
\newcommand{\Pb}{\mathbb{P}}
\newcommand{\EE}{\mathbb{E}}
\newcommand{\Cc}{\mathcal{C}}

\newcommand{\Ll}{\mathscr{L}}

\newcommand{\ccdot}{\;\cdot\;}

\newcommand{\mm}{^-}
\newcommand{\spc}[1]{\,#1\,}
\newcommand{\simQ}{\approx}
\newcommand{\eqOrder}{\eqsim}

\newcommand{\Law}{\mathrm{Law}}
\newcommand{\Cov}{\Xi}
\newcommand{\tQ}{\widetilde{Q}}
\newcommand{\tW}{\widetilde{W}}

\newcommand{\tZ}{\widetilde{Z}}

\newtheorem{proposition}{Proposition}[section]
\newtheorem{lemma}[proposition]{Lemma}
\newtheorem{corollary}[proposition]{Corollary}
\newtheorem{remark}[proposition]{Remark}
\newtheorem{theorem}[proposition]{Theorem}

\newtheorem{definition}[proposition]{Definition}

\numberwithin{equation}{section}

\crefname{algorithm}{Algorithm}{Algorithms}
\crefname{assumption}{Assumption}{Assumptions}
\crefname{lemma}{Lemma}{Lemmas}
\crefname{theorem}{Theorem}{Theorems}
\crefname{remark}{Remark}{Remarks}
\crefname{corollary}{Corollary}{Corollaries}
\crefname{figure}{Fig.}{Figures}
\crefname{section}{Section}{Sections}
\crefname{proposition}{Proposition}{Propositions}
\crefname{definition}{Definition}{Definitions}

\newcommand{\death}{\raisebox{-0.17em}{\includegraphics[width=3.3mm]{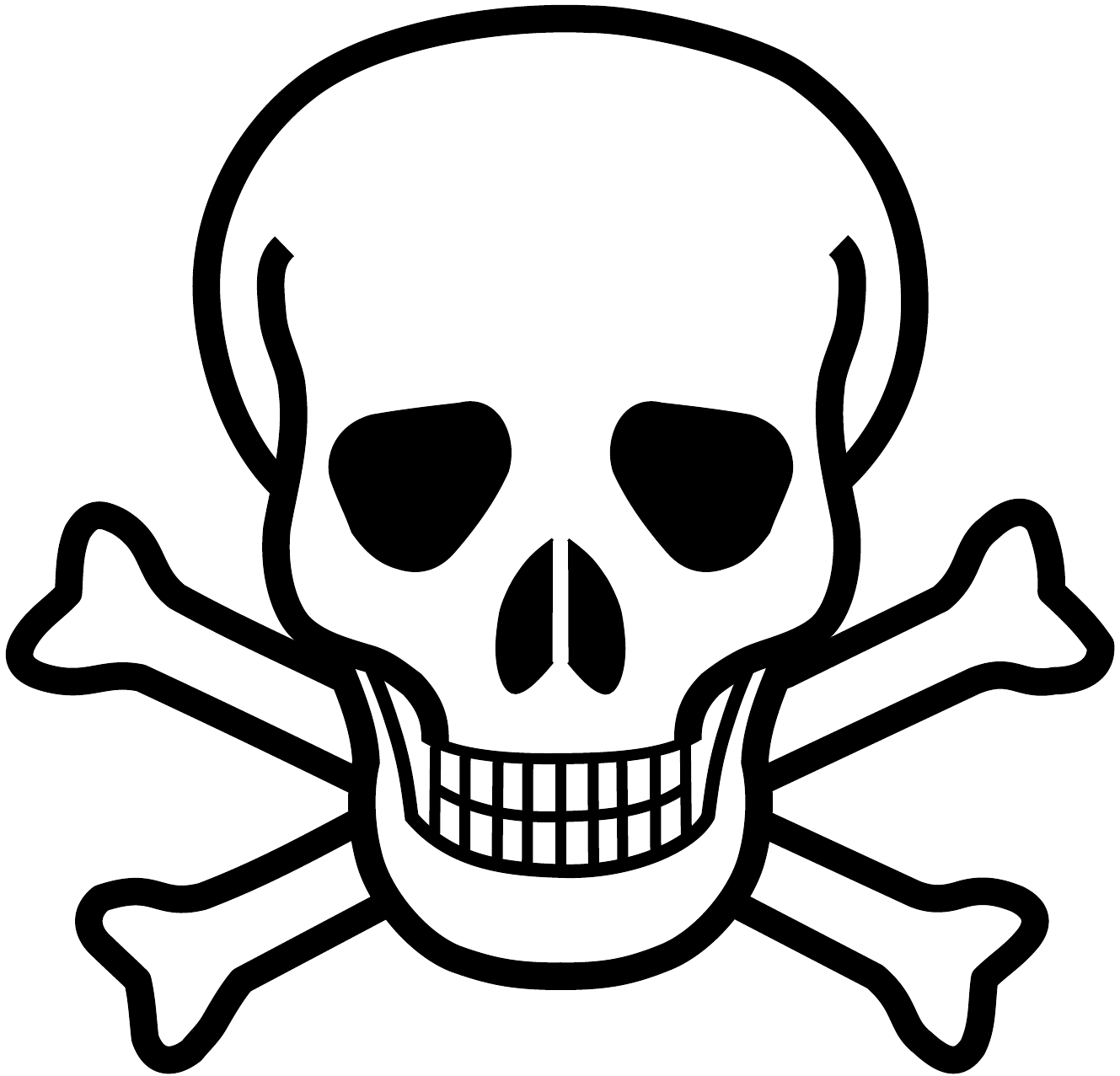}}}

\begin{document}

\title[Gaussian Structure of Singular Stochastic Burgers]{The Gaussian Structure of the Singular Stochastic Burgers Equation}
\author{Jonathan C. Mattingly}
\address{Department of Mathematics and Department of Statistical 
  Science, Duke University, Durham, NC USA}
\author[M. Romito]{Marco Romito}
  \address{Dipartimento di Matematica, Universit\`a di Pisa, Largo Bruno Pontecorvo 5, I--56127 Pisa, Italia }
 \author{Langxuan Su}
\address{Department of Mathematics, Duke University, Durham, NC USA}
\begin{abstract}
We consider the stochastically forced Burgers equation with an emphasis
on spatially rough driving noise. We show that the law of the process at a
fixed time $t$, conditioned on no explosions, is absolutely continuous
with respect to the stochastic heat equation obtained by removing the
nonlinearity from the equation. This establishes a form of ellipticity
in this infinite dimensional setting. The results follow from a
recasting of the Girsanov Theorem to handle less spatially regular
solutions while only proving absolute continuity at a fixed time and not
on path-space. The results are proven by decomposing the solution into
the sum of auxiliary processes which are then shown to be absolutely
continuous in law
to a stochastic heat equation. The number of levels in this
decomposition diverges to infinite as we move to the stochastically
forced Burgers equation associated to the KPZ equation, which we
conjecture is just beyond the validity of our results (and certainly
the current proof). The analysis provides insights into the structure
of the solution as we approach the regularity of KPZ. A number of
techniques from singular SPDEs are employed as we are beyond the
regime of classical solutions for much of the paper.
\end{abstract}
\maketitle

\begin{center}
\textit{LS dedicates this to his mother Deli Xie for everything
she has done for the family.}
\end{center}
\vspace{0.1cm}

\section{Introduction} 

The lack of one, distinguished \textit{standard} Borel topology, with its
associated Lebesgue measure, is the source of many differences between stochastic dynamics in finite and infinite
dimensions.  It is typical for the law of a stochastic ordinary
differential equation to have a transition law which is absolutely
continuous with respect to Lebesgue measure. In finite dimensions, the 
equivalence of transition densities is the norm while in infinite
dimensions it is the exception.
Of course, this fact is at the core of the difference between ordinary
and partial differential equations. In the stochastic setting, it
produced additional difficulties as many of the classical ideas such as
ellipticity, smoothing and transition densities are tied to the
existence of a Lebesgue measure.

Here, we provide an analysis showing when there is a preferred  topology
whose associated Gaussian measure plays the role of 
Lebesgue measure in infinite dimensions. We study the stochastically
forced Burger's equation in a singular regime and show
that the distribution of the dynamics at time $t$ is mutually
absolutely continuous with the  Gaussian
measure associated with linear dynamics where the nonlinear term
has been removed.

In the infinite-dimensional setting, if one only considers
finite-dimensional functionals of the solution 
(such as the evaluation in a space-time point), existence of densities
with respect to the natural reference measure -- again the Lebesgue
measure -- has a large literature, mostly related to
Malliavin calculus. Here we point out for instance to the
monograph \cite{San2005}, or the papers \cite{MatPar2006,HairerMattingly2011,gerasimovivcs2019hormander}. In particular,
the setting in \cite{MatPar2006,HairerMattingly2011,gerasimovivcs2019hormander} is orthogonal 
to ours, as the authors there consider equations driven 
by finite-dimensional Wiener processes, 
while our equation is very singular with a stochastic forcing 
that is non-degenerate in all directions. Through that lens, these
papers are dealing with the hypoelliptic setting but only answering
finite-dimensional questions about any transition densities, while this
paper considers what might be called the truly elliptic setting where
the structure of the stochastic forcing sets the  relevant topology,
and hence the reference measure, for
the full infinite-dimensional setting  (see \cite{Mattingly2002a} for
a broader, all though, dated discussion of this).

A much more substantial literature is devoted to the same problem
(in a smoother regime) at the level of path measures, thanks to the Girsanov Theorem.
We point out for instance to the monograph \cite{DapZab1992}. There is
strong evidence that that approach is not directly applicable to our setting.

The first works we are aware of that consider the problem we are
interested are \cite{DapDeb2004,Mattingly_Suidan_2005}. In \cite{DapDeb2004}  equivalence is proved
for invariant measures and, via the strong Feller property,
 the solution at fixed times. This work takes a different tact,
leveraging the time-shifted Girsanov method
contained in \cite{Mattingly_Suidan_2005,Mattingly_Suidan_2008},
and in a more consistent presentation in \cite{Watkins2010}. Those
works are the starting point for this investigation, but we will see
that significant work is required to extend to the singular setting. 

In the case of rough but sufficiently smooth forcing when all of the objects are 
classically defined, the time-shifted Girsanov method
contained in \cite{Mattingly_Suidan_2005,Mattingly_Suidan_2008,
  Watkins2010} can be applied to our setting. 
As the roughness increases, we decompose the equation into
an increasing number of levels of equations, and 
some stochastic objects in some levels require
renormalizations in the sense of
\cite{hairer2013solving,gubinelli2015paracontrolled,gubinelli2017kpz,mourrat2015construction,catellier2018paracontrolled}. The
additional levels of decomposition are driven by our need to prove
absolute continuity and not by the need for renormalizations in the
sense of singular SPDEs. The analysis further illuminates the structure
of the equations by underlining structural changes that occur as the
roughness increases.  In particular, the KPZ equation (in Burgers equation
form) presents itself as a boundary case just beyond the analysis of
this paper. There is strong evidence that this relates to a
fundamental change in the structure of the equation in the KPZ setting. 

If the KPZ is the boundary case, it is still open whether our
results extend to that case. When the forcing is precisely the spatial
derivative  of the space-time white noise, the invariant measure is
Gaussian. However, it is unclear if any Gaussian structure persists if
the structure of the noise is perturbed. Since the semigroup in that setting is known to be Strong
Feller in the KPZ case \cite{HaiMat2018} even with more general
forcing, we know that the failure of our results to generalize will
not be because of the appearance of a rough, random shift outside of
the needed Cameron-Martin space of admissible shifts as in \cite{BarGub2020}. (See
\cite{da2006introduction} and Sections~\ref{sec:cm-TSG-overview} and~\ref{sec:Full-CM}.) In \cite{BarGub2020}, the authors prove
the singularity of the $\Phi^4_3$ measure with respect to the
Gaussian free field and absolute continuity with respect to
a random shift of the
Gaussian free field.

Additionally, we believe that establishing absolute continuity
of the dynamics 
with respect to a Gaussian  reference measure will open additional
perspectives and approaches to analyzing these rough SPDEs.
Finally, we mention that in \cite{BesFer2016} the authors prove
a connection between a nonlinear problem and a linear
problem.

\vspace{1em}
\noindent \textbf{Outline of Paper: }
The paper is organized as follows. \cref{sec:main}
contains our main result, and \cref{sec:centralideas}
contains the main tools we use to prove it:
decompositions of the solution and the time-shifted
Girsanov method. In \cref{sec:Preliminary},
we give the basic definitions and estimates,
and study the regularity of the solution
and of the terms appearing in the equation.
In \cref{sec:equivalence}, we prove
our general statements on absolute
continuity and equivalence,
which are used in \cref{sec:decomposition}
and \cref{sec:Xdecomposition} to prove
absolute continuity of the decompositions.
All together, these results prove the main
theorem. %
In the Appendix, we recall some details on
Besov spaces and paraproducts (Appendix \ref{Besov}),
we define the Gaussian objects that appear in the
decompositions and prove their regularity
(Appendix \ref{sec:Gaussian}), and we give a result
of existence and uniqueness for
the needed equations (Appendix \ref{sec:ExistenceRegularity}).

\section{Main Result}\label{sec:main}
Consider the stochastic Burgers equation on $\T = [0, 2\pi]$ 
with periodic boundary condition: 
\begin{equation}\label{eq:burgers}
\Ll u_t = B(u_t)dt + Q dW_t,
\end{equation}
where $A = - \partial_{xx}$, $\mathscr{L} = \partial_t + A$, $B(u, v) = \partial_x (uv)$, 
and we write $B(u) := B(u, u)$. 
Also, $W$ is a cylindrical Brownian motion on $L^2(\T)$.
Since $A$ is a positive, symmetric operator on functions
in $L^2(\T)$ with mean zero in space, we can
define $A^\delta$ for any $\delta \in \R$ by its spectral
decomposition.
Assume that $Q \simQ A^{\alpha/2}$ for some $\alpha \in \R$, 
where we write $Q \simQ A^{\beta/2}$ for some $\beta \in \R$
when $A$ and $Q$ have a common eigenbasis and 
$A^{-\beta} QQ^{*}$ is bounded with bounded inverse. 

We denote by $e^{-t A}$ the semigroup
generated by $-A$. 
The use of the notation $\Ll u_t$ on the left-hand side of \eqref{eq:burgers} is meant
to be both compact and evocative of the fact that we will
consider the mild or integral formulation of the equation. Namely, if
$u_0$ is the initial condition, then $u$ solves
\begin{align}\label{eq:burgersMild}
  u_t - e^{-t A}u_0= \int_0^t   e^{-(t-s) A}B(u_s)ds +  \int_0^t   e^{-(t-s) A} Q dW_s \,.
\end{align}
Based on the assumption on $Q$ and the structure of the equation,
if $u_0$ has spatial mean zero,
all terms in the equation will have mean zero, which
is in consistency with the domains of $A^\delta$ and $e^{-tA}$.
We will consider the setting when $Q\simQ A^{\frac{\alpha}2}$ for
$\alpha < 1$, with particular interest in the case of $\alpha$ close
to $1$. 

We will see that when $\alpha <1$, local solutions exist in the
H\"older space $\Cc^{(\frac12-\alpha)^-}$ up to a
stopping time $\tau_\infty$ that is almost surely positive for initial
conditions in $\Cc^\gamma$ for $\gamma > -1$.\footnote{A
  formal definition of the function spaces used is given at the start
  of Section~\ref{sec:Preliminary}.}When
$\alpha < \frac12$,
 standard energy estimates guarantee the
existence of a unique global solution (that is  $\tau_\infty = \infty$
almost surely). When $\alpha \geq \tfrac12$, the solution is no longer a
function and extra care needs to be taken as it is \textit{a priori}
possible that $\tau_\infty < \infty$ with positive probability. 

Because in the settings of primary interest, global solutions are not
assured, we will extend our state space to include an isolated
``death'' state, denoted by  $\death$, and define $u_t = \death$ when $t
\geq \tau_\infty$. This also has the advantage of underscoring the
applicability of these ideas to equations with explosive
solutions. With this in mind, we will extend state space to include the death
state by defining $\bar \Cc^\beta =  \Cc^\beta \cup \{\death\}$. We
extend the dynamics by setting $u_t =\death$ for all $t>0$ if
$u_0=\death$. To state our main results we
define the Markov transition semigroup $\Pc_t$ by
\begin{align*}
  (\Pc_t \phi)(u_0) = \EE_{u_0} \phi(u_t),
\end{align*}
where $\phi\colon \bar\Cc^{(\frac12-\alpha)^-}\rightarrow \R$ is a bounded
measurable function. This extends in a natural way to a transitions measure
$\Pc_t(u_0,K)=(\Pc_t \one_K)(u_0)= \Pb_{u_0}(u_t \in K)$ for measurable subsets $K$ of
$\bar \Cc^{(\frac12-\alpha)^-}$ and to the left action of probability measures
$\mu$ on $\bar \Cc^{(\frac12-\alpha)^-}$ by
\begin{align*}
  \mu \Pc_t\phi = \int (\Pc_t \phi)(u) \mu(du).
\end{align*}

Our main result will show that, at a fixed time $t$, the law of the random
variable $u_t$ on the event $\{\tau_\infty > t\}$ is absolutely
continuous to the law of the  Ornstein-Uhlenbeck
process obtained by removing the non-linearity from
\eqref{eq:burgers}.  In other words, if we define
$\mathcal{Q}_t(z_0,K) =\Pb_{z_0}( z_t \in K)$ where
\begin{equation}\label{eq:OU}
\Ll z_t =  Q dW_t,
\end{equation}
then we have the following result which will follow from more detailed
results proved in later sections.
\begin{theorem}\label{thm:SimplyMain}
  For any $\alpha<1$, $t>0$, and any $u_0, z_0 \in\Cc^\gamma$ with $\gamma>-1$
 and zero spatial mean,
  $\Pc_t(u_0,\ccdot) \ll
  p\mathcal{Q}_t(z_0,\ccdot)+(1-p)
  \one_{\{\death\}}(\ccdot)$ \footnote{Here $\mu \sim \nu$ signifies
    that the two measure are mutually   absolutely continuous with
    respect to each other. That is to say, $\mu \ll \nu$ and $\nu \ll
    \mu$, where $\nu\ll \mu$ means $\nu$ is  absolutely continuous with respect to   $\mu$.}, where $p=\Pb_{u_0}(\tau_\infty >
  t)$.
  In other words, the law of $u_t$,
  conditioned on non-explosion by time $t$, is equivalent as a measure
  to the law of $z_t$ when $u_t$ and $z_t$ start from $u_0$ and $z_0$
  respectively. 
\end{theorem}

\begin{remark} The absolute continuity given in  \cref{thm:SimplyMain}
  implies that any almost sure property of the Gaussian measure
  $\mathcal{Q}_t(z_0,\ccdot)$ is shared by $\Pc_t(
  u_0,\ccdot)$. For example, the spatial H\"older regularity  or the
  Hausdorff dimension of spatial level-sets or solutions of
  \eqref{eq:burgers} are the same as those of \eqref{eq:OU}.
  See \cite{BagRom2014} for results along these lines.
\end{remark}

Unfortunately, our methods are not (yet!) powerful
enough to cover the case $\alpha=1$. We remark though
that with a little effort, and the help of \cite{HaiMat2018},
one can prove that, at least when the diffusion operator
in equation~\eqref{eq:burgers} is $Q=\partial_x \approx A^{\frac12}$,
the law of the solutions of \cref{eq:burgers}
and of \cref{eq:OU} at each fixed time
are both equivalent to the law of white noise.
Indeed, by \cite{HaiMat2018}, the transition
semigroup of \cref{eq:burgers}
is strong Feller. If one assumes that
the transition semigroup is irreducible,
then by a theorem of Khasminskii
(see for instance \cite{DapZab1996}),
transition probabilities are equivalent.
The final part of the argument is,
again by \cite{HaiMat2018},
that white noise is invariant for
the semigroup.
We do not know if equivalence holds
beyond the case $Q=\partial_x$.

\section{Central ideas in Theorem~\ref{thm:SimplyMain}}\label{sec:centralideas}

We will now give the central arc of 
three different (but related) arguments which can prove 
\cref{thm:SimplyMain}. Though there is overlap in the arguments, we 
feel each highlights a particular connection and helps to give a more 
complete picture.

\subsection{First Decomposition}
\label{sec:FirstDecompPreview}

The core idea used to prove Theorem~\ref{thm:SimplyMain} is the
decomposition of the solution of~\eqref{eq:burgers} into to the sum of
different processes of increasing regularity. In ~\eqref{eq:burgers},
the smoothness of solutions is dictated solely by the stochastic
convolution term, namely the last term on the right-hand side
of~\eqref{eq:burgersMild}.

We begin by taking the stochastic convolution as the first level in
our decomposition. This first level will be fed through the integrated
nonlinearity, namely the first term on the right-hand side
of~\eqref{eq:burgersMild}. We will then keep only the roughest
component and use it to force the next level in our
hierarchy.  At each level, we will include a stochastic forcing term
which, though smoother than the forcing at the previous level, will be
sufficiently rough to generate a stochastic convolution which is
less smooth than the forcing generated by the previous level through
the nonlinearity. Eventually, we will reach a level where the terms in
the equation can be handled by classical methods, and the expansion terminated.

More concretely, fixing number of levels $n$, $n \in \Nb$,  we begin by positing the
existence of process $\Y^{(0)}, \dots, \Y^{(n)}$ and remainder
term $\RY^{(n)}$ so that
\begin{equation}\label{eq:decomp2}
 u_t \eqdist \sum_{i=0}^n \Y^{(i)}_t + \RY^{(n)}_t,
\end{equation}
where $\eqdist$ denotes equality in law.
We define the  terms in this expansion by  
\begin{equation}
  \label{eq:system2}\left\{
  \begin{aligned}
    \Ll \Y^{(0)}_t &= Q_0 dW^{(0)}_t, \\
    \Ll\Y^{(1)}_t &= B(\Y^{(0)}_t)dt + Q_1 dW_t^{(1)}, \\
    \Ll \Y^{(2)}_t &= \big(B(\Y^{(0,1)}_t) - B(\Y^{(0,0)}_t)\big)dt + Q_2 dW_t^{(2)}, \\
    \Ll \Y^{(3)}_t &= \big(B(\Y^{(0,2)}_t) - B(\Y^{(0,1)}_t)\big)dt + Q_3 dW_t^{(3)}, \\
    \vdots\quad&\qquad \qquad\vdots\quad \qquad \qquad\vdots \qquad\qquad \qquad\vdots \\
    \Ll \Y^{(n)}_t &= \big(B(\Y^{(0,n-1)}_t) - B(\Y^{(0,n-2)}_t)\big)dt + Q_n dW_t^{(n)}, \\
    \Ll \RY^{(n)}_t &= \big(B(\Y^{(0,n)}_t+ \RY^{(n)}_t) - B(\Y^{(0,n-1)}_t)\big)dt
    + \tQ_n d\tW _t^{(n)},
  \end{aligned}\right.
\end{equation}
where the $Q_1,\dots,Q_n, \tQ_n$ are a collection of linear operators,
the $W_t^{(0)}, \dots, W_t^{(n)}, \tW _t^{(n)}$ a collection of
standard, independent, cylindrical Wiener processes, and
\begin{equation*}
  \Y^{(m,n)}_t = \sum_{i=m}^n \Y^{(i)}_t, 
\end{equation*}
with $\Y^{(0, -1)}_t = \Y^{(0, -2)}_t = 0$. If we choose
$\Y_0^{(0)}=\cdots=\Y_0^{(n)}=0$ and $\RY_0^{(n)}=u_0$
and  require
\begin{equation}
  \label{eq:Qcondition}
  Q_0 Q_0^* + Q_1 Q_1^* + \cdots + Q_n Q_n^* + \tQ_n\tQ_n^* = QQ^*,
\end{equation}
then
\begin{align}\label{eq:QWexpansion}
 Q W_t \eqdist \tQ_n \tW _t^{(n)} + \sum_{k=0}^n Q_k W_t^{(k)},
\end{align}
and, at least formally, the condition given
in~\eqref{eq:decomp2} holds. To make the argument complete, we need to
demonstrate that each of the equations in~\eqref{eq:system2} is well
defined and have at least local solutions.
The number of levels $n$ will be chosen as a function of $\alpha$.
The closer $\alpha$ is to one, the more levels are required.

Notice that because of \eqref{eq:QWexpansion}, which followed from
\eqref{eq:Qcondition}, the stochastic convolution from
\eqref{eq:burgersMild} satisfies
\begin{align}
  \label{eq:StochasticConvZs}
  z_t = e^{-t A} z_0+ \int_0^t   e^{-(t-s) A} Q dW_s \eqdist \tZ_t^{(n)} +\sum_{k=0}^n Z_t^{(k)}\,,
\end{align}
where 
\begin{align}\label{eq:Z}
   \Ll  Z_t^{(k)} = Q_{k} dW_t^{(k)}\quad\text{and}\quad  \Ll \tZ_t^{(n)} = \tQ_{n} d\tW_t^{(n)}\,
\end{align}
with initial conditions $Z_0^{(0)}= \cdots =Z_0^{(n)}=0$ and $\tZ_0^{(n)}=z_0$.
It is worth noting that $ Z_t^{(0)}=\Y_t^{(0)}$ and that all
equations, but $\RY_t^{(n)}$, are ``feed-forward'' in the sense that the forcing
drift $B(\Y^{(0,k-1)}_t) - B(\Y^{(0,k-2)}_t)$ in the $k$-th level is adapted to
the filtration $\mathcal{F}_t^{(k-1)} = \sigma( W_s^{(j)}: j
\leq k-1, s \leq t)$. In this sense, conditioned on
$\mathcal{F}_t^{(k-1)}$, $\Y^{(k)}_t$ is a forced linear equation with
both stochastic and (conditionally) deterministic forcing. This in turn
implies that  conditioned on
$\mathcal{F}_t^{(k-1)}$,  $\Y^{(k)}_t$ is a Gaussian random variable.

We will prove Theorem~\ref{thm:SimplyMain} by showing, for any
fixed $t>0$ and all $k=1,\dots,n$, \footnote{Given a random
      variable $X$ and a $\sigma$-algebra $\mathcal{F}$, we write
      $\Law(X\mid \mathcal{F})$ for the random measure defined by
      $\Law(X\mid \mathcal{F})(A) = \EE ( \,\one_A(X) \mid 
      \mathcal{F}\,)$. When $\mathcal{F}$ is not the trivial
      $\sigma$-algebra, randomness will remain, and one typically
      requires that properties of  $\Law(X\mid \mathcal{F})$ hold
      only almost surely. }
\begin{equation}
  \label{eq:YabsCont}
  \begin{aligned}
    \Law(\Y_t^{(k)}\mid  \mathcal{F}_t^{(k-1)}) &\sim
    \Law(Z_t^{(k)})\quad\text{a.s.,\qquad and}\\
    \Law(\RY^{(n)}_t  \mid  \tau_\infty>t, \mathcal{F}_t^{(n)}) &\ll
    \Law(\tZ_t^{(n)})\quad\text{a.s.}\,.
\end{aligned}
\end{equation}
We will see in  Section~\ref{sec:existenceTime} that the random existence time of $\RY^{(n)}$ is almost surely equal to
that of $u$ and hence we will use $\tau_\infty$ in both settings. 
We will show in Section~\ref{sec:absContiunityViaDecomp} how this
sequence of statements about the conditional laws combined with the
structure of \eqref{eq:system2}  will imply Theorem~\ref{thm:SimplyMain}.

\begin{remark}
  We have chosen to structure the initial conditions in
  \eqref{eq:system2} with  $\Y_0^{(0)}=\cdots=\Y_0^{(n)}=0$ and
  $\RY_0^{(n)}=u_0$. This is solely for convenience, as a number of the
  estimates for the $\Y_t^{(k)}$ and $Z_t$ are simpler to develop
  without the mild complication of initial conditions. We could have just
  as easily taken $\Y_0^{(0)}=u_0$ and the rest zero or
  $\RY_0^{(n)}=0$ and  $\Y_0^{(0)}=\cdots=\Y_0^{(n)}=\frac{1}{n}u_0$.
\end{remark}

\subsection{Second Decomposition}

 Accepting the result of \eqref{eq:YabsCont} from the previous section, 
 it might seem reasonable to
 replace the instances of $\Y_t^{(j)}$ in \eqref{eq:decomp2} with
 $Z_t^{(j)}$.  This would have a number of advantages. One is that the
 $Z_t^{(j)}$ are explicit Gaussian processes which will simplify the
rigorous definition of some of the more singular terms in the
decomposition. Additionally, it will emphasize the relationship between
the   $\Y_t^{(j)}$ construction in \eqref{eq:decomp2} and the tree
constructions developed in 
analysis of singular SPDEs (such as 
\cite{hairer2013solving,mourrat2015construction,gubinelli2017kpz}), 
which is driven
by isolating the singular objects in the solution. 

Motivated by this discussion, we now consider the expansion
  \begin{equation}\label{eq:decomp}
 u_t \eqdist \sum_{i=0}^n \X^{(i)}_t + \RX^{(n)}_t ,
\end{equation}
where
\begin{equation}\label{eq:system}
\left\{	\begin{aligned}
	\Ll \X^{(0)}_t &= Q_0 dW^{(0)}_t, \\
	\Ll \X^{(1)}_t &= B(Z^{(0)}_t)dt + Q_1 dW_t^{(1)}, \\
	\Ll \X^{(2)}_t &= (B(Z^{(0,1)}_t) - B(Z^{(0,0)}_t))dt + Q_2 dW_t^{(2)}, \\
	\Ll \X^{(3)}_t &= (B(Z^{(0,2)}_t) - B(Z^{(0,1)}_t))dt + Q_3 dW_t^{(3)}, \\
 \vdots\quad&\qquad \qquad\vdots\quad \qquad \qquad\vdots \qquad\qquad \qquad\vdots \\
	\Ll \X^{(n)}_t &= (B(Z^{(0,n-1)}_t) - B(Z^{(0,n-2)}_t))dt + Q_n dW_t^{(n)}, \\
	\Ll \RX^{(n)}_t &= (B(\X^{(0,n)}_t+ \RX^{(n)}_t) - B(Z^{(0,n-1)}_t))dt 
	 + \tQ_n d\tW _t^{(n)}, 
	\end{aligned}\right.
\end{equation}
and the $Q$'s are again as in~\eqref{eq:Qcondition} and
\begin{align*}
   Z^{(m,n)}_t = \sum_{i=m}^n Z_t^{(i)},
\end{align*}
with 
$Z_t^{(1)}, \dots, Z_t^{(n)}$ again defined by~\eqref{eq:Z}. Again we
take $\X_0^{(0)}=\cdots=\X_0^{(n)}=0$ and $\RX_0^{(n)}=u_0$ where $u_0$ was the initial
condition of \eqref{eq:burgers}.

Though in many ways we find the  $\Y$ expansion in \eqref{eq:system}
more intuitive and better motivated, we will find it easier to prove
Theorem~\ref{thm:SimplyMain} for the $\X$ expansion in
\eqref{eq:system} first. We will then use it to deduce
Theorem~\ref{thm:SimplyMain} for the $\Y$ expansion in
\eqref{eq:system2}.

More concretely, we will begin by proving that for each $k \leq n$, 
\begin{equation}\label{eq:XabsCont}
  \begin{aligned}
  \Law(\X_t^{(k)}\mid \mathcal{F}_t^{(k-1)}) &\sim \Law(Z_t^{(k)})\quad\text{a.s.,\qquad and}\\
  \Law(\RX^{(n)}_t\mid t < \tau_\infty,\mathcal{F}_t^{(n)} ) &\ll \Law(\tZ_t^{(n)}) \quad\text{a.s.}\,.
\end{aligned}
\end{equation}
Then we will deduce that for $k=1,\dots,n$,
\begin{equation}\label{eq:XYabsCont}
\begin{aligned}
  \Law(\Y_t^{(k)}\mid\mathcal{F}_t^{(k-1)} )
    &\sim \Law(\X_t^{(k)}\mid\mathcal{F}_t^{(k-1)})\quad\text{a.s.,\qquad and}\\
   \Law(\RY^{(n)}_t\mid  \tau_\infty>t, \mathcal{F}_t^{(n)}) &\ll \Law(\tZ_t^{(n)} ) \quad\text{a.s.}\,.
\end{aligned}
\end{equation}
By combining \eqref{eq:XabsCont} with \eqref{eq:XYabsCont}, we can deduce
that \eqref{eq:YabsCont} holds.

There is no fundamental obstruction to proving \eqref{eq:YabsCont} directly, as
it essentially requires the same calculation as proving
\eqref{eq:XYabsCont}. Similarly, we could have directly proven
\eqref{eq:YabsCont} before \eqref{eq:XabsCont}; however, along the way
we would have collected most of the estimates needed to prove
\eqref{eq:XabsCont}. We hope that proving all three statements,
namely \eqref{eq:YabsCont},  \eqref{eq:XYabsCont}, and
\eqref{eq:XabsCont}, will
help show the relationship between different ideas around singular SPDEs.

\begin{remark}
  The careful reader has likely noticed that the 
  absolute continuity statements in \eqref{eq:YabsCont},
  \eqref{eq:XYabsCont}, and  \eqref{eq:XabsCont} are stated only at a
  fixed time $t$ and not on the space of trajectories from $0$ to $t$. Hence one is not
  free to prove the result for the $\X$ expansion by simply replacing the $\Y$ with
  $Z$ in \eqref{eq:system2} by a change of measure on path-space to obtain \eqref{eq:system}.  
  There is strong evidence that $Z$ is not absolutely
  continuous with respect to $\Y$ on path-space, see \cref{rem:optimal}.
  Nonetheless, we will show that $\X$ and $\Y$ satisfy a
  modified version of absolute continuity on path-space which will imply \eqref{eq:XYabsCont}.
\end{remark}

\subsection{Cameron-Martin Theorem and  Time-Shifted Girsanov Method}
\label{sec:cm-TSG-overview}

The Cameron-Martin Theorem and the closely related Girsanov's Theorem
are the classical tools for proving two stochastic processes are
absolutely continuous. Both describe when a ``shift'' in the drift  can be absorbed into
a stochastic forcing term while keeping the law of the resulting random variable
or stochastic process absolutely continuous with respect to the
original law.

More concretely, to prove \eqref{eq:XabsCont}, we will absorb the drift terms on the
right-hand side of the equations for $\X_t^{(1)}, \dots, \X_t^{(n)},
\RX_t^{(n)}$ into the stochastic forcing term on the right-hand side of the
the same equation. In the case of the Cameron-Martin Theorem, this
absorption is done using the integrated, mild form of the equation, analogous
to \eqref{eq:burgersMild}.
The resulting expressions are identical in form to
the analogous $Z$ expressions implied by~\eqref{eq:Z}. The Girsanov
Theorem proceeds similarly as the Cameron-Martin Theorem, but the
drift is removed instantaneously at the level of the driving equation and not in an
integrated form as in the Cameron-Martin Theorem. We will see that
this leads to both stronger conclusions and a need for stronger
assumptions in order to apply the  Girsanov
Theorem.

In
\cite{Mattingly_Suidan_2005,Mattingly_Suidan_2008,Watkins2010}, Girsanov
Theorem was recast by shifting the infinitesimal perturbation injected
by the drift at one instance of time to a later instance. By shifting
the drift perturbation forward in time by the flow $e^{-tA}$, it is
regularized in space. This regularized, time-shifted drift can
then be compensated by the noise at the later moment of time,
thereby 
extending the applicability of Girsanov's Theorem.
This extended applicability will
be critical to our results. The price of the extended applicability is
that only a
modified form of trajectory-level absolute continuity is proven, which
nonetheless, is sufficient to deduce absolute continuity at the terminal
time of the path.  We have dubbed this approach as the
\textsc{Time-Shifted-Girsanov Method}. A full discussion with all of the
details is provided in Section~\ref{sec:TSG}.

\subsection{Gaussian Regularity}
There is a tension when applying either the Cameron-Martin Theorem or the  Time-Shifted Girsanov Method
between the roughness of the stochastic forcing,
set by   $Q\simQ A^{\frac{\alpha}2}$,  and the roughness of the drift
term on the right-hand side  of the $k$-th equations. The gap between
these regularities cannot be too big. Hence, it is critical to understand
the regularity of the solutions and the drift term involving the
nonlinearity $B$.

When $\alpha < \frac12$, we will see that the
Time-Shifted Girsanov Method can be applied directly to
\eqref{eq:burgers} to obtain the desired result following the general
outline of \cite{Mattingly_Suidan_2008, Watkins2010}.
When $\alpha \in [\frac{1}{2},1)$, the multi-level decomposition
from~\eqref{eq:system2} and~\eqref{eq:system}  will be 
required to make sure the jump in regularity between the stochastic
forcing in an equation and the drift to be removed is not too
large.

The reason for this change at $\alpha=\frac12$ is fundamental to our
discussion. When $\alpha < \frac12$, the product of all of the spatial functions
$f$ and $g$ contained in $B(f,g)$ is well defined classically as the
functions will have positive H\"older regularity. Hence, pointwise
multiplication is well defined. When  $\alpha \geq \frac12$, we must
leverage the Gaussian structure of the specific processes being multiplied and a
renormalization procedure to
make sense of the product of some of the terms in $B(f,g)$. When
$\alpha \in [\frac12,\frac34)$, through these considerations, we will always be able to give
meaning to $B(f,g)$  at each moment of time in all of the needed cases. When $\alpha \in [\frac34,1)$, at times we must consider the time integrated version of the
drift term from the mild formulation (analogous to the second term from
the right in~\eqref{eq:burgersMild}) and leverage time decorrelations
of the Gaussian process to make sense of the nonlinear term in its
integrated form $J(f,g)_t$ defined in \eqref{eq:J}.

\subsection{Relation to Trees and Chaos expansions}
We now explore briefly the relationship between tree representations
of stochastic Gaussian objects from
\cite{hairer2013solving,mourrat2015construction,gubinelli2017kpz} and this work,
in a heuristic way. 
One way to view the trees in those work is to consider the expansion one obtains by formally
substituting the integral representation of $z$ given in
\eqref{eq:Z} back into the first
integral term on the right-hand side of~\eqref{eq:burgersMild}.
Repeated applications of this
is one way to develop an expansion of the solution $u_t$ in terms of
finite trees of $z$ with a remainder.
These tree representations of stochastic objects 
are key to the analysis in those work.
We will later see the drift terms of \eqref{eq:system} can be
decomposed into some of the same trees of $z$.

We push the idea of tree expansions further by grouping the 
trees formed by $Z$ with different regularity, and
adding an extra stochastic forcing at each level.
Here, looking back at \eqref{eq:StochasticConvZs},  
as we have subdivided our noise into $n$
levels ($Z^{(1)},\dots, Z^{(n)}$) and one remainder term
($\tZ^{(n)}$), we have a tree-like expansion mixing the Gaussian
inputs of different levels. 
This work can be viewed as giving a more refined analysis of
the stochastic objects in \eqref{eq:burgers}.

We will see that our eventual assumptions
on the  $Q$'s will imply that the regularity of the $Z^{(k)}$ increases
with $k$. Hence, we can understand the expansions in~\eqref{eq:system2}
and \eqref{eq:system} as two different groupings of the a subset of the
tree objects from this expansion so that the sum of the terms in a level is sufficiently
regular to be absorbed into that level's stochastic forcing term
via the Cameron-Martin Theorem or the Time-Shifted Girsanov Method. 
Clearly this grouping is not unique, but it makes analysis based
on regularity more straightforward. 

\section{Preliminaries}\label{sec:Preliminary}
We now collect a number of estimates and observations, which will be
needed to prove the versions of \cref{thm:SimplyMain} based on the
expansion across noise levels given in \eqref{eq:system2}
and  \eqref{eq:system}. We start by setting the function
analytic setting in which we will work and recalling some basic estimates
on the operator $A$ and the semigroup it generates. We then discuss
the stochastic convolution, the regularity of
solutions~\eqref{eq:burgers} and the equations in~\eqref{eq:system2}
and~\eqref{eq:system}.

\subsection{Function spaces and basic estimates}\label{sec:functionSpaces}
We shall denote by $\Cc^\gamma$, $\gamma\in\R$,
the separable version of the Besov-H\"older space
$B^\gamma_{\infty,\infty}(\T)$ of order $\gamma$, namely the closure
of periodic smooth functions with respect to
the $B^\gamma_{\infty,\infty}$ norm.
See Appendix \ref{Besov} for some details about Besov spaces.
If $f \in \Cc^{\gamma}$, we will say that $f$ has (H\"older) regularity $\gamma$.
We will write $\Cc^{\gamma\mm}$ for the union of all of space
$\Cc^{\beta}$ with $\beta < \gamma$. \footnote{Notice that this
  definition does not coincide with the classical definition of
  H\"older spaces for integer values of the index. See \cref{rm:clasicalHolder}. }

Given a Banach space $\XX$ of functions $f(x)$ on $\T$, we will write
$C_T\XX$ for the space of time dependent functions $f(t,x)$ on $[0,T]
\times \T$ 
 such that for each $t \in [0,T]$, $f(t,\ccdot) \in \XX$ and as $s
\rightarrow t$ we have that $\|f(s,\ccdot)-f(t,\ccdot)\|_{\XX}\rightarrow 0$. We will endow this
space with the norm
\begin{align*}
  \|f\|_{C_T\XX} = \sup_{t \in [0,T]} \|f(t,\ccdot)\|_{\XX}.
\end{align*}
Typical examples we will consider are $C_T\Cc^\beta$ and $C_TL^2$.
If $f \in C_T\Cc^{\gamma}$, we will say that 
$f$ has (H\"older) regularity $\gamma$ (in space).
For convenience, we will write $C_T \Cc^{\gamma\mm}$ for the union of all of the spaces
$C_T \Cc^{\beta}$ with $\beta < \gamma$.

As we are interested in solutions which might have a finite time of
existence, we will introduce the one-point compactification of
$\Cc^{\gamma}$,  $\overline \Cc^{\gamma}=\Cc^{\gamma}\cup \{
\death\}$. $\overline \Cc^{\gamma}$ is a topological space where the open
neighborhoods of $\death$ are given by $\{ u \in \Cc^{\gamma}: \|u\|_{\Cc^\gamma} >R\}$ for
$R>0$. With a light abuse of notation, we will write $C_T \overline \Cc^{\gamma}$ to mean the space of
all continuous functions on $[0,T]$ taking values in
$\overline\Cc^{\gamma}$. We do not place a norm on $C_T \overline
\Cc^{\gamma}$ and view it only as a topological space. Observe that
if $u \in C_T \overline
\Cc^{\gamma}$ and $\tau \in (0,\infty]$ such that for $t \in [0,T]$, $u_t =
\death$ if $t \geq \tau$ and $u_t \in \Cc^{\gamma}$ if $t < \tau$,
then for all $t \in [0,T] \cap [0,\tau)$, $u \in C_t \Cc^{\gamma}$ because $u$ is
continuous in  $\overline
\Cc^{\gamma}$.  We feel this justifies the notation $ C_T \overline
\Cc^{\gamma}$ for continuous functions on $\Cc^{\gamma}$ even though the
space is not endowed with the supremum norm. Additionally, because of
our choice of open neighborhoods of $\death$, if $\tau<
\infty$ then $\|u_t\|_{\Cc^\gamma}  \rightarrow \infty$ as $t
\rightarrow \tau$. Again, we will write $C_T \Cc^{\gamma\mm}$ 
for the union of all of the spaces
$C_T \overline\Cc^{\beta}$ with $\beta < \gamma$.

We collect a few useful properties of $A$ and its semigroup in the
following proposition. Here and through out the text, we write $a
\lesssim b$ to mean there exists a positive constant $c$ so that $a
\leq c b$. When the constant $c$ depends on some parameters we will
denote them by subscripts on $\lesssim$. We will write $\gtrsim$ when the
reverse inequality holds for some constant and $\eqOrder$  when both
$\lesssim$ and $\gtrsim$ hold (for possibly different constants).
\begin{proposition}\label{prop:besicEstimates} For $\gamma\in\R$,
    $\partial_x:\Cc^\gamma\to\Cc^{\gamma-1}$ and   $A:\Cc^\gamma\to\Cc^{\gamma-2}$ 
    are  bounded linear operators. Additionally, if $\delta\in\R$ with
    $\gamma\leq\delta$ then
    \begin{align}
      \label{eq:basicASemigroupEst}
        \|e^{-t A}f\|_{\Cc^\delta}
        \lesssim
      t^{\frac12(\gamma-\delta)}\|f\|_{\Cc^\gamma}\quad\text{and}\quad
      \|A^{\frac\delta2} e^{-t A}f\|_{L^2}
        \lesssim t^{\frac12(\gamma-\delta)}\|A^{\frac\gamma2} f\|_{L^2}.
    \end{align}
    In particular, for any $\epsilon >0$ and $\delta,t >0$,
    \begin{align}
      \label{eq:L2ASemigroupEst}
         \|A^{\frac\delta2}e^{-t A}f\|_{L^2}
        \lesssim \|A^{\frac\delta2}e^{-t A}f\|_{\Cc^\epsilon}  
        \lesssim_\epsilon t^{\frac12(\gamma-\delta-\epsilon)}\|f\|_{\Cc^\gamma}\,.
    \end{align}
\end{proposition}
Using these estimates, one can obtain the regularity of the
solution $z_t$ of \eqref{eq:OU}. 

\begin{remark}[Regularity of Stochastic Convolution]
  \label{rem:StochasticConvolutionReg} Using the
  results \eqref{prop:besicEstimates} and some classical embedding theorems, we
  have that if
 \begin{align}\label{eq:stochConv}
  z_t= \int_0^t e^{-(t-s) A} \Cov\,dW_s,
   \end{align}
 with  $\Cov \simQ A^{\frac{\delta}2}$ (and hence is a mild solution of
 $\Ll z_t = \Cov\,dW_t$), then
 $\|z_t\|_{\Cc^\gamma} < \infty$ uniformly on finite time intervals for
 all $\gamma <\frac12-\delta$. More compactly,
 $z \in
 C_t\Cc^{(\frac12-\delta)^-}$  for any $t>0$ almost surely.
\end{remark}

Given the structure of the nonlinearity $B$,  we are particularly
interested in the properties of the pointwise product of two
functions. We now summarize the results in the classical setting and
recall the results in the Gaussian setting.

\begin{remark}[Canonical Regularity of Gaussian Products]\label{rem:prod}
It is a classical result that if $f \in \Cc^\delta$ and $g \in
\Cc^\gamma$, then their product is well defined if $\delta+\gamma>0$
with $fg \in \Cc^{r}$ where $r=\gamma\wedge\delta \wedge (\gamma + \delta)$. This can be
summarized more completely in     the statement  that the pointwise
product $(g,f) \mapsto gf$
is a continuous bilinear operator between $\Cc^\gamma\times\Cc^\delta$
to $\Cc^{r}$ if
$\gamma+\delta>0$.

When $f \in \Cc^\delta$ and $g \in \Cc^\gamma$ with $\gamma+\delta \le 0$,
there is no canonical way to define the product. A critical
observation for this work, and most of the recent progress in singular
PDEs \cite{hairer2013solving,gubinelli2015paracontrolled,mourrat2015construction,gubinelli2017kpz,catellier2018paracontrolled}, is that  even when $\gamma+\delta \le 0$ one can often
define the product  in $B(f,g)$ via a renormalization procedure
to have the canonical regularity by leveraging the particular Gaussian
structure of $f$ and $g$. We will see that this is not possible in the
needed cases when $\gamma+\delta \le -\frac12$. However, 
we still can make sense of $B(f,g)$
convolved in time with the   heat semigroup, by leveraging the specific
structure of the time correlations of the  specific $f$ and $g$ of interest.

\end{remark}
With this fact about Gaussians in mind, we make the following
definition to simplify discussions.
\begin{definition}\label{def:canonicalProduct}
  Given  $f \in 
\Cc^\delta$ and $g \in \Cc^\gamma$, we will say that the product $fg$ has the \textsc{canonical regularity}
if it is well defined, possibly after a renormalization procedure, with  $fg \in 
\Cc^{r}$ where $r=\gamma\wedge\delta \wedge (\gamma + \delta)$.
\end{definition}

\subsection{Regularity of the Mild form of the Nonlinearity.} \label{sec:Jreg}

Looking back at \eqref{eq:burgersMild}, we see that the nonlinearity
$B$ integrated in time against the heat semigroup (namely the first
term on the right-hand side of this equation) will be a principal
object of interest. We now pause to study the main properties of this
object, while postponing some more technical considerations to the Appendix.

In the sequel, it will be notationally convenient to define the bilinear operator $J(f,g)_t$ by
\begin{align}
  \label{eq:J}
   J(f, g)_t : = \int_0^t e^{-(t-s)A}B(f_s, g_s)\,ds,
\end{align}
and $J(f)_t=J(f,f)_t$.

\begin{remark}[Canonical Regularity of  $J$]\label{rem:JClassicalProdReg}
   If $f \in C_t\Cc^\gamma$ and $g \in C_t\Cc^\delta$ with $\gamma+
   \delta>0$, then
$B(f, g) \in C_t\Cc^{r -1}$
where $r=\gamma\wedge \delta \wedge(\gamma+\delta)$, which
implies that
\begin{align*}
\| J(f,g)_t\|_{\Cc^\beta} \lesssim \int_0^t
  \|e^{-(t-s)A} B(f_s,g_s)\|_{\Cc^\beta}\, ds  \lesssim \|f g  \|_{C_t\Cc^{r}}\int_0^t
  \frac{1}{(t-s)^{\frac12(\beta-r+1)}}\, ds\,.
\end{align*}
Here we have used the estimate from~\cref{prop:besicEstimates}. Since this last integral is finite when $\frac12(\beta-r+1)<1$,
we deduce that $\beta < r+1$, implying that $J(f,g)_t \in
\Cc^{(r+1)^-}$ where  $r=\gamma\wedge \delta \wedge(\gamma+\delta)$.
\end{remark}

As mentioned in Remark~\ref{rem:prod} (and proved in Appendix \ref{sec:Gaussian}), we
can prove that product $fg$ is well defined with canonical regularity
in  the specific examples needed in this work, when $f \in C_t\Cc^\gamma$ and
$g \in C_t\Cc^\beta$ with $\gamma+ \delta >
-\frac12$. However, we will show in Appendix \ref{sec:Gaussian} that when  $\gamma+\delta >
-\frac32$, 
$J(f,g)$ is well defined with  $J(f,g) \in
C_t\Cc^{(r+1)^-}$ where $r=\gamma\wedge \delta
\wedge(\gamma+\delta)$, even though the product $fg$ might not
be well defined with its canonical regularity.

\begin{definition}\label{def:canonicalJ}
  Given  $f \in 
C_t\Cc^\delta$ and $g \in C_t\Cc^\gamma$, we will say that $J(f,g)$ has the \textsc{canonical regularity}
if $J(f,g)$ is well defined (possibly via a renormalization procedure) with   $J(f,g) \in
C_t \Cc^{(r+1)^-}$ where $r=\gamma\wedge \delta \wedge(\gamma+\delta)$.
\end{definition} 

\begin{remark}[$J$ regularizes in our setting]\label{rem:Jregs}
  Looking at  \eqref{eq:burgers}, it is relevant to understand when the
  map $f \mapsto  J(f)_t$ produces an image process which is more
  regular than the input process $f$. Assume that we are in the setting
  where $J(f)_t$ has the canonical regularity. Then $J(f)_t$ will be
  smoother if $f \in \Cc^\gamma$ with $\gamma< \gamma+1$
  if $\gamma >0$ and  $\gamma< 2\gamma+1$ if $\gamma < 0$. Thus, $J$
  is always regularizing when applied to functions of positive
  regularity and it will be regularizing in the canonical setting
  when for a distribution of negative H\"older regularity  greater than
  $-1$. We will always find ourselves in one of these two settings.

  Building from the above, it is also relevant to understand how the
  regularizing effect of $J$ interact with products.
  More specifically, later we need to compute the regularity of
  sums of terms that are 
  essentially like $B(z, z')$, $B(J(z), z)$, $B(J(z))$,
  and etc, where $z'$ is another Ornstein-Unlenbeck process with positive 
  H\"older regularity. 
  In our setting,
  we will see that a term like $B(z, z')$ is the least regular term, which dictates 
  the canonical regularity of the sum,
  while all other terms with more $J$'s involved are more regular.
\end{remark}

The regularizing nature of $J$ highlighted in \cref{rem:Jregs} is
closely related to the use of fixed point methods to prove existence and
uniqueness of local in time solutions with the needed regularity. This
is explored further in Appendix~\ref{sec:ExistenceRegularity}.

\subsection{Regularity of solutions}\label{sec:reg}

We now turn to the regularity of the Burgers equation
\eqref{eq:burgers} and those in our decompositions \eqref{eq:system2}
and \eqref{eq:system}. Most of the equations are forced linear
equations except for the remainder equations $\RY_t$ and
$\RX_t$ and the original Burgers equation \eqref{eq:burgers}. While the
 following discussion will be illuminating in these later cases, it
 is most directly applicable in the setting of forced linear equations. A
complete treatment in the nonlinear setting (namely $\RY_t$, 
$\RX_t$ and \eqref{eq:burgers}) involves a fixed point argument which we postpone to
Appendix~\ref{sec:ExistenceRegularity}. Nonetheless, the
discussion in this section will still be illuminating to these cases
while focusing  on the forced linear equation setting.

We begin by studying a more general equation which can subsume most of
the equations
in our decompositions \eqref{eq:system2}
and \eqref{eq:system}. Since all of the forcing drift terms on the
right-hand side of the equations are a finite sum of terms of the form
$B(f,g)$ for some $f$ and $g$, it is enough to consider the more
general equation 
\begin{equation}
  \Ll v_t =  B(f_t,g_t) dt + \Cov\,dW_t\,,
  \label{eq:generalReg}%
\end{equation}
for some given $f \in C_t\Cc^\gamma$ and   $g \in C_t\Cc^\delta$ and $\Cov \simQ A^{\beta/2}$ 
for some $\beta, \gamma \in \R$. All of the forced linear equations of interest are
a finite sum of equations of this form. 

The solution to \eqref{eq:generalReg} with
initial condition $v_0$ is given by
\begin{align*}
v_t= e^{-t A}v_0 + J(f,g)_t +  z_t,
\end{align*}
where now $z_t$ is the stochastic convolution solving \eqref{eq:stochConv}
and $J$ is again defined by \eqref{eq:J}. We will assume that $f$ and $g$
are such that $J(f,g)_t$ has the canonical regularity in the sense of \cref{def:canonicalJ}.

 For  any $t>0$ and any reasonable $v_0$,  $e^{-t A}v_0
 \in \Cc^b$ for all $b \in \R$. Hence, the first term will not be the
 term which fixes the regularity of the equation, and either $J$ or $z$
 will determine the maximal regularity of the system.
 
 By Remark~\ref{rem:StochasticConvolutionReg}, the stochastic convolution $z
\in C_t\Cc^{(\frac12-\beta)^-}$. By \cref{def:canonicalJ}, we have that $J(f,g) \in
C_t \Cc^{(r +1)^-}$  where $r=\gamma\wedge \delta \wedge
(\gamma+\delta)$. Hence the
regularity of the solution will be set by the stochastic convolution
$z$ to be  $C_t\Cc^{(\frac12-\beta)^-}$ if $r > \frac12 - \beta$.
We will arrange our choice of
parameters so that the equations \eqref{eq:system2}
and \eqref{eq:system} will always satisfy this condition. Hence, the
equations will always have their regularity set by the
stochastic convolution term in the equation. With this motivation, we
make the following definition.

\begin{definition}\label{def:CanonicalRegularity}We say that an equation of the general form 
  \eqref{eq:generalReg} has \textsc{canonical regularity} if $v_t$
  has same H\"older regularity as the associated stochastic convolution $z_t$
  uniformly on finite time intervals. 
\end{definition}

The above considerations are also relevant to assessing the regularity of
the remaining equations in \eqref{eq:system2}
and \eqref{eq:system} as well as the original Burgers equation.
Observe that the solution to 
\eqref{eq:burgers} can be written as 
\begin{align*}
  u_t= e^{-tA} u_0 + J(u)_t + z_t,
\end{align*}
where $z_t$ solves \eqref{eq:OU}. Since $Q \simQ A^{\alpha/2}$, we know 
from Remark~\ref{rem:StochasticConvolutionReg} that $z \in C_t\Cc^{(\frac12-\alpha)^-}$. Since we 
are interested in $\alpha < 1$, we have that $\frac12-\alpha >
-\frac12$. In light of \cref{rem:Jregs}, we see that $u_t$ has the same 
regularity in space as $z_t$ then $J(u)_t$ will be more regular in 
space (assuming we can show that $J(u)_t$ has the canonical 
regularity dictated by $u$). Hence, it is expected that in our setting 
the regularity of \eqref{eq:burgers}  will be set by the regularity of 
the stochastic convolution term so $u \in
C_t\Cc^{(\frac12-\alpha)^-}$. For more details, see the discussion in
Appendix~\ref{sec:ExistenceRegularity}.

\section{Absolute Continuity  of Measures}\label{sec:equivalence}
We now turn to the main tools used to establish the absolute
continuity statements required to prove \cref{thm:SimplyMain} as
outlined in \cref{sec:cm-TSG-overview}.

Whether at the level of the Burgers equation \eqref{eq:burgers} or
when considering one of the levels in the
expansions in \eqref{eq:decomp2} or \eqref{eq:decomp}, we are left
considering when the law of $v_t$  is equivalent with
respect to the law of $z_t$ where 
\begin{equation}
  \begin{aligned}
  \Ll v_t &= F_t \,dt + \Cov\,  dW_t\,,\\
   \Ll z_t &= \Cov\,  dW_t\,.
\end{aligned}
  \label{eq:general}%
\end{equation}
Here $\Cov \simQ A^{\beta/2}$ 
for some $\beta \in \R$, and $F_t$ is a continuous (in time)
stochastic process with the space regularity
to be specified presently. We always assume that $F_t$ is adapted to
some filtration to which $W_t$ is also adapted. In some instances, it
is possible that  $F$ is independent of $W$.

When all of the terms are well defined and $z_0=v_0$, observe that
$v_t = z_t + h_t$ where
\begin{align}
  \label{eq:h}
  h_t=\int_0^t e^{-(t-s)A}F_s \,ds.
\end{align}
\subsection{The Cameron-Martin Theorem}\label{sec:Full-CM}
The Cameron-Martin Theorem gives if and only if conditions describing
when the $\Law(z_t)$ is equivalent to $\Law(v_t)$, with $v_t=z_t+h_t$, for a fixed time
$t$ and a deterministic shift $h_t$. If $\Cov \simQ A^{\beta/2}$, then
the covariance operator of the Gaussian random variable
$z_t$ is (up to a compact operator)
$A^{\beta-1}$. Then Cameron-Martin Theorem requires that
$\|A^{\frac{1-\beta}{2}} h_t\|_{L^2} < \infty$ (see \cite[Theorem 2.8]{da2006introduction}).
If $F_t$ is random but
independent of the stochastic forcing $W_t$, then we can still apply
the  Cameron-Martin theorem by first conditioning on the trajectory
of $F$. This produces the following sufficient condition for
absolute continuity, which is a version of the classical Cameron-Martin
Theorem adapted to our setting.
\begin{theorem}[A version of Cameron-Martin]\label{lem:CM}
  In the setting of \eqref{eq:general} with $z_0=v_0$, let $\mathcal{G}_t$ be a
  filtration independent of
  the Brownian forcing $W_t$. Let $h_t$ be as  in \eqref{eq:h} and 
  adapted to  $\mathcal{G}_t$. If for some $ t>0$,
  $\|A^{\frac{1-\beta}{2}} h_t\|_{L^2} < \infty$ almost surely, then
  $\Law(z_t+h_t\mid  \mathcal{G}_t )$ is equivalent as a measure
  to  $\Law(z_t)$ almost surely. In particular, it is sufficient  that $h_t \in
  \Cc^\gamma$ almost surely for $\gamma+\beta >1$.
 \end{theorem}
\begin{remark}\label{rem:CMApplicablity}
If $F=B(f,g)$ for some $f \in
  C_t\Cc^\gamma$, and $g\in
  C_t\Cc^\delta$ such that $J(f,g)_t$ has the canonical regularity 
  then $h_t=J(f,g)_t \in \Cc^{(r+1)^-}$ where  $r=\gamma\wedge \delta
  \wedge(\gamma+\delta)$.  Hence, the condition
  in \cref{lem:CM} becomes $r+\beta=\gamma\wedge \delta
  \wedge(\gamma+\delta)+\beta>0$.
  
  Notice that this remark extends to the setting where
  \begin{align}\label{generalizedF}
    F=\sum_{i=0}^m c_i 
    B(f^{(i)},g^{(i)})
  \end{align}
  for some $ c_i \in \R$,  $f^{(i)} \in
  C_t\Cc^{\gamma_i}$ and  $g^{(i)} \in
  C_t\Cc^{\delta_i}$ where $r_i=\gamma_i\wedge \delta_i
  \wedge(\gamma_i+\delta_i)$ and the condition on the indexes becomes
  $r+\beta>0$ with $r=\min r_i$.
\end{remark}

\begin{remark}
  \cref{lem:CM} immediately extends to the setting where
  $v_t=z_t+h_t+k_t$ and both $h_t$ and $k_t$ are adapted to
  $\mathcal{G}_t$ with $h_t$ satisfying the assumptions of
  \cref{lem:CM}. Then  the $\Law(z_t+h_t+k_t \mid  \mathcal{G}_t )$
  is equivalent as a measure to the
  $\Law(z_t+k_t \mid   \mathcal{G}_t)$ almost surely.
\end{remark}

\begin{remark}
  The condition that $z_0=v_0$ is only for simplicity and not needed. 
\end{remark}

\begin{proof}[Proof of \cref{lem:CM}] Without loss of generality, we
  can take $z_0=v_0=0$ since the effect of the initial condition
  cancels out when looking at the difference between $z_t$ and $v_t$. Since $h_t$ is adapted to
  $\mathcal{G}_t$, we can apply the classical Cameron-Martin Theorem with $h_t$
  considered as deterministic by conditioning. 
  A direct application of the It\^o
  isometry to \eqref{eq:stochConv} shows that $z_t$ from
  \eqref{eq:general} has covariance operator $C_t=\int_0^t e^{s A}\Cov
 \Cov^*  e^{s A} ds$. Because $\Cov \simQ A^{\beta/2}$, we have that
 $C_t \simQ \widetilde C_t= \int_0^t e^{s A}A^\beta e^{s A} ds$. Hence,
 the classical condition from the Cameron-Martin Theorem that
 $\|C_t^{\frac12} h_t\|_{L^2} < \infty$ is equivalent to
 $\|A^{\frac{1-\beta}{2}} h_t\|_{L^2} < \infty$. Since this condition
 is assumed to hold almost surely, the classical Cameron-Martin
 Theorem implies that $\Law(z_t+h_t\mid  \sigma( h_s : s\leq t)\,) $
 is equivalent as a measure  to $\Law(z_t)$ almost surely. Since $\sigma( W_s : s \leq t)$ is
  independent of $\mathcal{G}_t$, we have that the complement 
  of $\sigma( h_s : s\leq t)$ in $\mathcal{G}_t$ is independent of
  the random measure
  $\Law(z_t+h_t\mid  \sigma( h_s : s\leq t)\,)$, which implies  $\Law(z_t+h_t\mid \mathcal{G}_t) $ is equivalent as a measure
  to the $\Law(z_t)$. To verify the last remark, observe that since $h_t \in
  \Cc^\gamma$ almost surely we have  $A^{\frac{1-\beta}{2}} h_t \in
  \Cc^{\gamma+\beta -1}$ almost surely. Now since
  $\|A^{\frac{1-\beta}{2}} h_t\|_{L^2} \lesssim \|A^{\frac{1-\beta}{2}}
  h_t\|_{\Cc^\epsilon}$, we see that if  $\beta + \gamma - 1 > \epsilon$
  for some $\epsilon>0$ the last remark holds. This is possible
  because we have assumed $\beta+\gamma>1$.
\end{proof}

\subsection{The Standard Girsanov Theorem}
\label{sec:Girsanov}

The Girsanov Theorem is essentially the specialization of the
Cameron-Martin theorem to the path-space of a stochastic differential
equation, while relaxing the assumptions to allow random shifts in the drift
which are adapted to the Brownian Motion forcing the SDE.

We again consider the setting of \eqref{eq:general}. Since we will
be discussing path-space measures, we will write $v_{[0,t]}$ and $z_{[0,t]}$ for
the random variable denoting the entire path of $v$ and $z$
respectively on the time
interval $[0,t]$.
We now give a version of the Girsanov
Theorem adapted to our setting.

\begin{theorem}[A version of Girsanov]\label{thm:basicG}
  In the
  setting of  \eqref{eq:general} with $z_0=v_0$,
  let $\mathcal{F}_t$ be a filtration to which $W$ is an adapted
  Brownian Motion, and let $\mathcal{G}_t$ be a filtration independent
  of $\mathcal{F}_t$. Let $\tau$ be a stopping time adapted to
  $\mathcal{H}_t=\sigma(\mathcal{G}_t,\mathcal{F}_t)$ with  $\Pb(\tau
  >0) > 0$ such that $v_t$ and 
  $F_t$ are stochastic processes adapted to $\mathcal{H}_{t\wedge \tau}$ so
  that  $v_t$ 
  solves \eqref{eq:general} for $t < \tau$, and 
  \begin{equation}
  \label{eq:basicGirsanov}
  \int_0^t \|A^{-\beta/2} F_s\|_{L^2}^2 \, ds  < \infty,
\end{equation}
almost surely 
for all $t < \tau$.
Then 
$\Law( v_{[0,t]} \mid  t < \tau, \mathcal{G}_t) \ll \Law( z_{[0,t]})$ almost surely.

In particular, it is sufficient
that $F \in C_t\Cc^\sigma$ for $\sigma+\beta >0$ and any $t < \tau$
for \eqref{eq:basicGirsanov} to hold almost surely.
\end{theorem}

\begin{remark}\label{rem:BasicG}
  In the setting of \cref{thm:basicG}, we assume that there exists
  stochastic processes $f_t$ and $g_t$ so that $F_t=B(f_t,g_t)$  for
  all $t < \tau$ with $f \in C_t\Cc^\gamma$,  $g
  \in C_t\Cc^\delta$ and such that $F_t=B(f_t,g_t)$ has the canonical
  regularity, namely $F \in C_t\Cc^{r-1}$ for $r=\gamma\wedge \delta
  \wedge(\gamma+\delta)$ and  all $t < \tau$. Then  \cref{thm:basicG}
  applies and the regularity assumption in
  \eqref{eq:basicGirsanov} is implied by $r-1+\beta=\gamma\wedge \delta
  \wedge(\gamma+\delta)+\beta-1>0$ or rather $r+\beta=\gamma\wedge \delta
  \wedge(\gamma+\delta)+\beta>1$ .
\end{remark}
 The condition $r+\beta=\gamma\wedge \delta
  \wedge(\gamma+\delta)+\beta>1$ from \cref{rem:BasicG}
  should be contrasted with the condition $r+\beta=\gamma\wedge \delta
  \wedge(\gamma+\delta)+\beta>0$ from
  \cref{rem:CMApplicablity}. Relative to the Cameron-Martin 
  \cref{lem:CM}, the basic Girsanov \cref{thm:basicG} does
  have the advantage that $F_t$ can be adapted to the forcing Brownian
  motion and not independent. Also, the results are not
  comparable, as \cref{thm:basicG} proves pathwise equivalence
  while \cref{lem:CM} only proves equivalence at a fixed time $t$.

  \begin{proof}[Proof of \cref{thm:basicG}]
    We begin by defining
    \begin{align*}
    \tau_M=\tau \wedge \inf\Big\{ t >0 : \int_0^t \|A^{-\beta/2} F_s\|^2_{L^2} ds > M\Big\}
    \end{align*}
    and 
  \begin{equation*}\label{eq:vHat}
  \Ll \, v_t^M = \one_{\{t < \tau_M\}}F_t \,dt+ \Cov\, dW_t.
\end{equation*}
Observe that $v_t^M$ is well defined on $[0,t]$ for any $t > 0$ due to
the stopping time and that $v_t=  v_t^M$ on the event $\{t <
\tau^M\}$. Then,
\begin{align*}
  \exp\Big( \int_0^t \|A^{-\beta/2}\one_{\{s < \tau_M\}}F_s\|^2_{L^2} ds\Big)
  < C(1+ e^M),
\end{align*}
almost surely for some fixed constant $C$. Thus, the classical
Kazamaki criterion (see for instance \cite{Krylov_2009}) ensures
that the local-martingale in the Girsanov Theorem is an integrable
martingale.  Let $v_{[0,t]}^{M}$ and $z_{[0,t]}$ be the path-valued 
random variables over the time interval $[0,t]$. We have that 
$\Law( v_{[0,t]}^{M} )$ is equivalent as a measure to
$\Law( z_{[0,t]})$.  Since $\mathcal{G}_t$ is independent of the
Brownian motion $W$, we have that the
$\Law( v_{[0,t]}^{M} \mid  \mathcal{G}_t)$  is equivalent as a
measure to the $\Law(z_{[0,t]})$
almost surely. 

We now show we can remove the truncation level $M$.  Now let $E$ be a
measurable subset of paths of length $T$ such that
$\Pb( z_{[0,T]} \in E ) =0$. To prove the absolute continuity claim,
we need to show that  $\Pb\big(v_{[0,T]}\in E,\tau>T
\mid \mathcal{G}_T\big)=0$ almost surely. If $\Pb(\tau>T
\mid \mathcal{G}_T)=0$ almost surely, we are done. Hence we proceed assuming $\Pb(\tau>T
\mid \mathcal{G}_T)>0$ almost surely.

Because, conditioned on $\mathcal{G}_T$,
the law of $ v^M_{[0,T]}$ is equivalent to the law of $z_{[0,T]}$
almost surely, we know that
$\Pb( v^M_{[0,T]}\in E \mid \mathcal{G}_T ) =0$ for all $M>0$ almost
surely. 
We also know from the construction of $ v^M$ that
$\Pb( v^M_{[0,T]}\in E, \tau_M > T \mid \mathcal{G}_T ) =\Pb(
v_{[0,T]}\in E, \tau_M > T \mid  \mathcal{G}_T )$. Now since

\begin{equation*}
  \Big\{ v_{[0,T]}\in E,\tau>T\Big\}
  \subset\bigcup_{M>0}\Big\{ v_{[0,T]}\in E,\tau_M>T\Big\},
\end{equation*}
we have that
\begin{align*}
  \Pb\big(v_{[0,T]}\in E,\tau>T \mid \mathcal{G}_T\big)
  & \leq\sup_{M>0}\Pb\big(v_{[0,T]}\in E,\tau_M > T \mid \mathcal{G}_T\big)\\
  & \leq\sup_{M>0}\Pb\big( v^M_{[0,T]} \in E,\tau_M > T
    \mid \mathcal{G}_T\big)\\
  &\leq\sup_{M>0}\Pb\big( v^M_{[0,T]}\in E \mid \mathcal{G}_T\big)
    = 0,
\end{align*}
as already noted since $\Pb\big(z_{[0,T]}\in E \big)=0$.
\end{proof}

\begin{remark}\label{rem:optimal}
  We believe that in the context of diffusions,
  namely when the $F$ in \eqref{eq:general} is
  a non-anticipative function of $v$,
  condition \eqref{eq:basicGirsanov}
  of \cref{thm:basicG} should be optimal,
  in the sense that
  equivalence holds if and only if
  \eqref{eq:basicGirsanov} holds.
  This statement is true in finite dimension,
  see \cite[Theorem 7.5]{LipShi2001}.
  Moreover, by adding a condition
  similar to \eqref{eq:basicGirsanov}
  for $z$ it should be possible to
  prove equivalence of the laws.
  The extension of these results
  in the framework of
  \cref{thm:basicG} goes beyond the
  scope of this paper and
  will be addressed elsewhere.

  Alternatively, if one has control of some moments of the solution
  sufficient to implied global existence (namely $\tau_\infty=\infty$),
  one can typically prove the equivalence between the laws in
  \cref{rem:BasicG}. For example, this can be accomplished using the
  relative entropy calculations given in Lemma~C.1 of \cite{Mattingly_Suidan_2005}.
\end{remark}

\subsection{The Time-Shifted Girsanov Method}\label{sec:TSG}

We now present the  \textsc{Time-Shifted Girsanov Method} which was
developed in \cite{Mattingly_Suidan_2005,Mattingly_Suidan_2008, Watkins2010}. It will
provide essentially the same regularity conditions in our setting as in
\cref{rem:CMApplicablity}  while allowing adapted shifts as in the
standard Girsanov Theorem. Interestingly, we will see that the
classical  Cameron-Martin Theorem still holds some advantages when
dealing with extremely rough Gaussian objects.

Considering the mild-integral
formulation of \eqref{eq:general}
\begin{equation}
  \label{eq:generalMild}
  v_t - e^{-t A} v_0 = \int_0^t  e^{-(t-s) A} F_s\, ds + \int_0^t
  e^{-(t-s) A} \Cov\, dW_s\,,
\end{equation}
we can understand the first term as the
shift of the Gaussian measure which is the second term. We will now
recast the drift term in \eqref{eq:generalMild} to extend the
applicability of the Girsanov Theorem.
 
We begin with the observation that for fixed $T>0$
\begin{equation}
  \label{eq:timeShiftDrift}
  \begin{aligned}
    \int_0^T   e^{-(T-s) A}F_sds &=  \int_0^T   e^{-\frac12(T-s) 
                                   A}e^{-\frac12(T-s) A}F_s\,ds \\
  &= \int_{\frac{T}2}^T   e^{-(T-s)  A}e^{-(T-s) A}F_{2s-T}\,ds=
    \int_{0}^T   e^{-(T-s)  A} \widetilde F_s \,ds,
  \end{aligned}
\end{equation}
where 
\begin{align}
  \label{eq:Fhat}
\widetilde F_s=2\one_{[\frac{T}2,T]}(s) e^{-(T-s) A}F_{2s-T}\,. 
\end{align}
Since 
$2s-T \leq s$ for all $s \in [\frac{T}2,T]$, $\widetilde F_s$ is adapted to the 
filtration of $\sigma$-algebras generated by the forcing Wiener 
process $W$ when $F_s$ is also adapted to the same filtration. Hence,
we can define the auxiliary It\^o stochastic differential equation 
\begin{equation}
  \label{eq:timeShifted}
  \Ll \,\widetilde v_t = \widetilde F_t\, dt + \Cov\, dW_t,
\end{equation}
which is driven by the same stochastic forcing as used to construct 
$v_t$. Choosing the initial 
data to coincide with $u_0$, the 
mild/integral formulation of this equation is 
\begin{align}
  \label{eq:timeShiftedMild}
  \widetilde v_t - e^{-t A}u_0 = \int_0^t e^{-(t-s)A}  \widetilde F_s\, ds +   \int_0^t   e^{-(t-s) A} \Cov\, dW_s\,. 
\end{align}

By comparing~\eqref{eq:generalMild}~and~\eqref{eq:timeShiftedMild}, we 
see that $\widetilde v_T=v_T$ while $\widetilde v_t$ need not equal $v_t$ for $t\neq 
T$. Hence, if we use the standard Girsanov's theorem to show that the 
law of $\widetilde v_{[0,T]}$ on path-space
is absolutely continuous with respect to the law of $z_{[0,T]}$ (the solution 
to \eqref{eq:general}), then we can conclude that the law of 
$\widetilde v_T$ (at the specific time $T$) is absolutely continuous with respect 
to the law of $z_T$ (again at the specific time $T$). Finally, since 
$v_T= \widetilde v_T$, we conclude that the law of 
$v_T$ is absolutely continuous with respect 
to the law of $z_T$, both at the specific time $T$. 

The power of this reformulation is seen when we write down the
condition needed to apply the Girsanov Theorem to remove the drift
from~\eqref{eq:timeShifted}. We now are required to have control over
moments of
\begin{align}
   \label{eq:timeShiftedGirsanov}
\int_0^T \| A^{-\beta/2}\widetilde F_s\|_{L^2}^2\, ds  
=\int_0^T \| A^{-\beta/2} e^{-(T-s)A} F_s\|_{L^2}^2\, ds  
\end{align}
to apply the standard Girsanov Theorem to transform the path-space
law of $\widetilde v_{[0,T]}$ to that of $z_{[0,T]}$. Comparing
\eqref{eq:basicGirsanov} with \eqref{eq:timeShiftedGirsanov}, we see
that the additional semigroup $e^{-(T-s)A}$ in the integrand improves
its regularity significantly.

Similarly, if we want to compare the distribution at time $T>0$ of two
equations starting from different initial conditions $v_0,z_0 \in
\Cc^b$, for $\beta \in \R$, then we can observe that
\begin{align*}
  e^{-T A} v_0 =  e^{-T A} z_0 +  e^{-T A} (v_0-z_0)  &=e^{-T A} z_0  +  \int_{\frac{T}2}^T e^{-(T-s) A}
                 \frac{2}{T}e^{-s A} (v_0 -z_0)\,ds\\
  &=  e^{-T A} z_0  +  \int_{\frac{T}2}^T e^{-(T-s) A} \widetilde F_s^{(0)}\,ds,
\end{align*}
where $\widetilde  F_s^{(0)}= \one_{[\frac{T}2,T]}(s)\frac{2}{T}e^{-s A} (v_0
-z_0)$. This is the observation at the core of the
 Bismut-Elworthy-Li formula \cite{Bismut84,ElworthyLi94}.  Observe that $\widetilde  F^{(0)}
\in C_T \Cc^b$ for any $b \in \R$, regardless of the initial conditions, so
$A^{-\frac{\beta}2} \widetilde  F^{(0)} \in C_T L^2$, and we will always
be able to use the Girsanov Theorem to remove this term.

It will be convenient to consider a slightly generalized setting where
$v_t$, $\widetilde v_t$ and $\zeta_t$ respectively solve mild forms of
the following equations,
\begin{equation}
  \label{eq:moreGen}
  \begin{aligned}
  \Ll \,v_t &= F_t\,dt +\,G_t\, dt+ \Cov\, dW_t, \\
  \Ll \,\widetilde v_t &= (\widetilde F_t+\widetilde F_t^{(0)}) \,dt + G_t\, dt + \Cov\, dW_t,\\
    \Ll \,\zeta_t &=  G_t\,dt+  \Cov\, dW_t,
\end{aligned}
\end{equation}
with initial conditions $v_0$ and $z_0$ where $\zeta_0=\widetilde v_0 = z_0$. Here $\widetilde
F_t$ is defined as in \eqref{eq:Fhat}, $\widetilde F_t^{(0)}$ as just
above, and $G_t$ and
$F_t$  are some stochastic processes.

\begin{theorem}[Time-Shifted Girsanov Method]\label{thm:tsG}
 In the setting of  \eqref{eq:moreGen},
  let $\mathcal{F}_t$ be a filtration to which $W$ is an adapted
  Brownian Motion, and let $\mathcal{G}_t$ be a filtration independent
  of $\mathcal{F}_t$. Fix initial conditions $v_0$ and $z_0$, and let $\tau$ be a stopping time adapted to
  $\mathcal{H}_t=\sigma(\mathcal{F}_t,\mathcal{G}_t)$ such that $\Pb(\tau
  >0) > 0$. Let $G_t$ be
  stochastic processes adapted to  $\mathcal{G}_t$ and defined for all
  $t \geq 0$. Let $v_t$ and $F_t$ be  stochastic process adapted to
  $\mathcal{H}_{t\wedge \tau}$, such that  $v_t$ 
  solves \eqref{eq:moreGen} for $t < \tau$ and
  \begin{equation}
  \label{eq:TSGirsanov}
  \int_0^t \|A^{-\frac{\beta}{2}}  e^{-(t-s)A} F_s\|_{L^2}^2\, ds   < \infty
\end{equation}
almost surely for all $t < \tau$, and $v_t$ and $\zeta_t$ have initial conditions
$v_0$ and $z_0$ respectively. Then $\Law( v_{t} \mid  t < \tau,  \mathcal{G}_t
  ) \ll \Law(
  \zeta_{t} \mid  \mathcal{G}_t)$ almost surely. Additionally,  there exists a solution $\widetilde
  v_t$ which  solves
  \eqref{eq:moreGen} for $t < \tau$ and with  $\Law(\widetilde
  v_{[0,t]} \mid  t < \tau,  \mathcal{G}_t ) \ll \Law(
  \zeta_{[0,t]} \mid  \mathcal{G}_t)$ almost surely.  In particular, it is sufficient  that $F \in
  C_t\Cc^\sigma$ almost surely  for $\sigma+\beta +1>0$ and for any $t < \tau$ to
  ensure \eqref{eq:TSGirsanov} holds.
 \end{theorem}

  \begin{remark}\label{rem:TSG}
  In the setting of \cref{thm:tsG}, we assume that there exists
  stochastic processes $f_t$ and $g_t$ so that $F_t=B(f_t,g_t)$ for
  all $t < \tau$ with $f \in C_t\Cc^\gamma$,  $g
  \in C_t\Cc^\delta$ and such that $F_t=B(f_t,g_t)$ has the canonical
  regularity, namely $F \in C_t \Cc^{r-1}$  for $r=\gamma\wedge
  \delta \wedge(\gamma+\delta)$ and all $t < \tau$. Then the regularity assumption of 
  \eqref{eq:basicGirsanov} is satisfied provided that
  $r+\beta=\gamma\wedge
  \delta \wedge(\gamma+\delta)+\beta>0$.

  As in \cref{rem:CMApplicablity}, this remark extends to the setting where
  \begin{align*}
    F=\sum_{i=0}^m c_i 
    B(f^{(i)},g^{(i)}),
  \end{align*}
  for some $ c_i \in \R$,  $f^{(i)} \in
  C_t\Cc^{\gamma_i}$ and  $g^{(i)} \in
  C_t\Cc^{\delta_i}$ where $r_i=\gamma_i\wedge \delta_i
  \wedge(\gamma_i+\delta_i)$ and the condition on the indexes becomes
  $r+\beta>0$ with $r=\min r_i$.
\end{remark}

\begin{remark}[Comparing Theorems~\ref{lem:CM}, \ref{thm:basicG}, and \ref{thm:tsG}]\label{rem:compare}
  Comparing the Cameron-Martin Theorem, the Standard Girsanov Theorem,
and the Time-Shifted Girsanov Method in the setting of $F_t=B(f_t,g_t)$, we see that the  Cameron-Martin
Theorem and the Time-Shifted Girsanov Method impose identical
regularity conditions on $f_t$ and $g_t$. The Time-Shifted Girsanov
Method has the added advantage of allowing one to consider $f$ and $g$
which are only adapted to the Brownian Motion $W$ and not independent as
the Cameron-Martin theorem requires. This added flexibility will be
critical to proving the needed absolute continuity for the remainder
variables $\RY$ and $\RX$.

Both only prove equivalence at a
fixed time which is an advantage as we only need this for our applications. 
However, we will see that the requirement that $B(f_t,g_t)$
has the canonical regularity of the Time-Shifted Girsanov
Method will be more restrictive than the requirement that $J(f,g)_t$
has the canonical regularity of the Cameron-Martin Theorem.
\end{remark}

\begin{proof}[Proof of \cref{thm:tsG}] Fixing a time $T$, we define
  \begin{align*}
   \int_0^T \|A^{-\frac{\beta}{2}} 
    \widetilde F_s\|^2_{L^2}\, ds=  \int_0^T \|A^{-\frac{\beta}{2}} 
    e^{-(T-s)A} F_s\|^2_{L^2}\, ds <\infty,
  \end{align*}
  almost surely on the event $\{T <
    \tau\}$. Fix a positive  integer $M$. The following stopping time 
  \begin{align*}
    \tau_M=\tau \wedge \inf\Big\{ t >0 : \int_0^t \|A^{-\frac{\beta}{2}} 
   \widetilde F_s\|^2_{L^2}\, ds > M\Big\}
  \end{align*}
is well defined  and that $\tau_M \rightarrow \tau$
monotonically as $M
\rightarrow \infty$. Let $\widetilde v_t$ be the solution to
\eqref{eq:timeShifted} and observe that it is well defined on $[0,T]$
on the event $\{T < \tau\}$ with $v_T=\widetilde v_T$ on the same event.
Now consider
\begin{equation*}\label{eq:vHat+ic}
  \Ll \,\widetilde v_t^M = \one_{\{t < \tau_M\}}\widetilde F_t\, dt + \widetilde F_t^{(0)} dt +G_t\,dt+ \Cov\, dW_t, 
\end{equation*}
with $\widetilde v_0=z_0$.
Clearly, $\widetilde v_t^M= \widetilde v_t$ for $t < \tau_M$. Furthermore,
because of the definition of $\tau_M$ and the fact that $\widetilde
F^{(0)} \in C_T\Cc^b$ for any $b \in \R$, the classical Girsanov Theorem
implies that the law of the trajectories of $\widetilde v^M$ on $[0,T]$,
conditioned on $\mathcal{G}_T$,
are equivalent (i.e. mutually absolutely continuous) to the law of
$\zeta$, conditioned on $\mathcal{G}_T$, 
on $[0,T]$. In the sequel, we will write $\zeta_{[0,T]}$ for the
random variable on paths of lengths $T$ induced by the law of $\zeta$.

By the same argument as in Theorem~\ref{thm:basicG}, we remove the
localization by $\tau^M$ to obtain $  \Law(\widetilde v_{[0,T]}\mid 
\tau>T) \ll \Law(\zeta_{[0,T]} \mid \mathcal{G}_T\big)$ almost surely.

    To conclude the proof, we let $D$ be any measurable subset such
    that $\Pb( \zeta_T \in D\mid \mathcal{G}_T)=0$ where $z_T$ is the distribution of $z$ at the
    fixed time $T$. Let $D_{[0,T]}$ be the subset of path-space of
    trajectories which are in $D$ at time $T$. Then
    \begin{equation}
      \label{eq:pathABStofixedTime}
      \begin{multlined}[.9\linewidth]
      0=\Pb\big( \zeta_T \in D  \mid \mathcal{G}_T)\\
       = \Pb( \zeta_{[0,T]} \in D_{[0,T]} \mid \mathcal{G}_T\big)
       =  \Pb\big(\widetilde v_{[0,T]} \in D_{[0,T]} , \tau> T \mid \mathcal{G}_T\big) \\
       = \Pb\big( \widetilde v_T\in D,\tau> T \mid \mathcal{G}_T\big)
       =\Pb\big( v_T\in D  , \tau> T \mid \mathcal{G}_T\big),
       \end{multlined}
    \end{equation}
where the last equality follows from the fact that $ \widetilde
v_T=v_T$ on the event $\{\tau> T\}$. 
    The chain of implications in \eqref{eq:pathABStofixedTime}
    shows that the law of $v_T$ restricted to the event
    $\{\tau> T\}$ is absolutely continuous with respect to the
    law of $\zeta_T$ with both conditioned on $\mathcal{G}_T$. 
\end{proof}

\subsection{Range of Applicability of Methods}
\label{sec:rangeOfApp}

We now consider for what regimes the Cameron-Martin Theorem and the Time-Shifted Girsanov Method
  can be applied directly to~\eqref{eq:burgers} to
  prove~\cref{thm:SimplyMain}. We will proceed formally with the
  understanding that some neglected  factors will lead to additional
  complications which will require more nuanced arguments.

  For the moment, we assume that \eqref{eq:burgers} has the
  canonical regularity, namely the regularity dictated by the stochastic convolution
  term. Thus, $u \in C_t\Cc^{(\frac12-\alpha)^-}$ where recall that
  $\alpha$ is the exponent which sets the spatial regularity of the
  forcing.

  When $\alpha<\frac12$, the
  solution $u_t$ to~\eqref{eq:burgers} has positive H\"older regularity
  with $u_t \in \Cc^{(\frac12-\alpha)\mm}$. This implies that $B(u_t)$
  has the canonical regularity with $B(u_t)
  \in \Cc^{(-\frac12-\alpha)\mm}$. Thus, the regularity condition to
  apply the  Cameron-Martin Theorem or the Time-Shifted Girsanov
  Method to \eqref{eq:burgersMild} becomes $\frac12-\alpha + \alpha =\frac12 >0$, which is always
  true. See \cref{rem:CMApplicablity}, \cref{rem:TSG}, and
  \cref{rem:compare}. 

  When $\alpha\geq \frac12$, the
  solution $u_t$ to~\eqref{eq:burgers} has negative H\"older regularity
  since $u_t \in \Cc^{(\frac12-\alpha)\mm}$ still. If we proceed as if
  the relevant terms have their canonical regularity ($B(u_t)$ in
  the case of the  Time-Shifted Girsanov Method and $J(u)_t$ in the
  case of the Cameron-Martin Theorem), then the regularity condition
  becomes $2(\frac12-\alpha)+\alpha=1-\alpha>0$, which restricts us to
  the setting of $\alpha<1$. One cannot directly apply neither the
  Time-Shifted Girsanov Method nor the Cameron-Martin Theorem directly
  to \eqref{eq:burgers}  when $\alpha\geq \frac12$. We will see that
  we need the  multi-level decomposition to incrementally improve the
  regularity of the solution to the point where we can apply the
  Time-Shifted Girsanov Method to the last level, namely $\RY^{(n)}$
  or $\RX^{(n)}$ depending on the decomposition. Along the way, we
  typically use the Cameron-Martin Theorem to prove
  equivalence of the levels in the decomposition to the appropriate
  Gaussian processes. This is possible since the fed-forward structure
  of the decomposition means that each level is conditionally
  independent from the previous.  When $\alpha \in [\frac34,1)$, there
  is an added complication that the terms to be removed by the change
  of measure only can be defined when integrated against the
  heat semigroup. This necessitates the use of the  Cameron-Martin
  Theorem  rather than the Time-Shifted Girsanov Method.
 
\subsection{Interpreting the Time-Shifted Girsanov Method}
It is tempting to dismiss the manipulations
in~\eqref{eq:timeShiftDrift} as a trick of algebraic manipulation. We
encourage you not to do so.

The standard Girsanov Theorem
compares the two equations of the form in~\eqref{eq:general} and asks when we can
shift the noise realization to another ``allowed'' noise realization
to absorb any differences in the drift terms, namely the $F_t$ in our
setting. Here ``allowed'' means in a way that across all realizations, 
the resulting noise term's distribution stays equivalent  to the
original noise term's distribution. To keep the path measures
equivalent on $[0,T]$, one needs to do this instantaneously  at every
moment of time.

Since the equation for $z_t$ in~\eqref{eq:general} is a
forced linear equation, the linear superposition
principle (a.k.a Duhamel's principle or the variation of constants
formula) applies. It states that we move an impulse injected into the
system at time $s$ to another time $t$ by mapping it under
the linear flow from the tangent space at time $s$ to the tangent space  at time
$t$.
Through this lens, we can interoperate~\eqref{eq:timeShiftDrift} as
a reordering of the impulses injected into the system by $F_s$ over
the interval $s \in [0,T]$. The
impulse injected at time $s$ is moved to time $t=\frac12(T+s)$ via
$e^{-(t-s)A}=e^{-\frac12(T-s)A}$. $\widetilde F_t$ is the resulting
effective impulse at the time $t=\frac12(T+s)$. The  Time-Shifted
Girsanov Method compensates for the time shifted impulse $\widetilde
F_t$ using the forcing via a change of measure. The resulting
$\widetilde F_t$ is more regular than $F_t$ for $t<T$. The
regularizing effect of the semigroup $e^{tA}$ vanishes as we approach
$T$. The requirement $\beta+\gamma +1 >0$ ensures that the
singularity at $t=T$  is sufficiently 
integrable to apply the classical Girsanov Theorem to the resulting
process with its  forcing impulses rearranged, which was denoted by
$\widetilde v$ in the Time-Shifted Girsanov discussion in Section~\ref{sec:TSG}.

\section{A Decomposition of Noise and Smoothness}\label{sec:decomposition}

The idea of decomposing the solution into the sum of terms of 
different regularity is a staple of SPDE analysis dating back at least 
to the pioneering work of Da Prato, Zabczyk, Flandoli,
Debussche, etc. See for instance \cite{DapZab1996,FlaGat1995,DapDeb2002}.
The decomposition 
of the solution $u_t$ into $z_t+v_t$ where $z_t$
solves~\eqref{eq:OU} is the starting point of many arguments. The 
advantage of this decomposition is that $z$, the rougher one of 
the two equations, is very explicit and has all of the direct 
stochastic forcing. In contrast,  typically the equation for $v$ is 
not a stochastic equation (as it contains no It\^o 
integrals). Rather, it is a random equation, and $z$ is contained in some 
of its terms. This leads to $v$ typically being more regular than $z$. 

We will build on these ideas with some important distinctions. The 
most basic will be that because we intend to use the 
Cameron-Martin Theorem/Time-Shifted 
Girsanov Method on each level, we will leave noise in each 
equation. Moreover, we  see that our explorations expose 
additional structures in the equation. In particular, to reach to 
$\alpha$ arbitrarily close to one, we will be required to divide our 
solution into an ever increasing number of pieces as we approach one. 

As mentioned in the introduction, 
there are three key ingredients in our result. The first is that 
products of Gaussian objects can be defined via renormalizations with 
their canonical regularity. 
The second is the Cameron-Martin Theorem/Time-Shifted Girsanov Method. 
These two elements were discussed in the proceeding two 
sections. We now introduce the third component, a noise decomposition 
across scales. With all three central ideas on the table, we can 
sketch the main proofs of this note.

\subsection{Regularity and Existence Times of Solutions}\label{sec:existenceTime}
We begin with a simple lemma which relates the maximal time of
existence of $u_t$ with those of 
 $\Y^{(0)}, \Y^{(1)}, \ldots$, $\Y^{(n)}$, $\RY^{(n)}$ satisfying \eqref{eq:decomp2} and \eqref{eq:system2}, and $\X^{(0)}, \X^{(1)}, \ldots$, $\X^{(n)}$, $\RX^{(n)}$
satisfying \eqref{eq:decomp} and \eqref{eq:system}.
\begin{lemma}\label{lem:localExt} Let $\tau_\infty$ be the maximal existence time
  of $u_t$.
  \begin{enumerate}
 \item If $(\Y^{(0)}, \Y^{(1)}, \ldots, \Y^{(n)}, \RY^{(n)})$ solves
      \eqref{eq:system2}, then $\Y^{(0)}_t, \Y^{(1)}_t, \ldots,
      \Y^{(n)}_t$ exist for all time $t$. Additionally, if
      \eqref{eq:decomp2} holds (or equivalently
      \eqref{eq:Qcondition} holds), then the  maximal time of existence for $\RY^{(n)}$ is
	the same as $\tau_\infty$ almost surely.
    \item If $(\X^{(0)}, \X^{(1)}, \ldots, \X^{(n)}, \RX^{(n)})$ solves
      \eqref{eq:system}, then $\X^{(0)}_t, \X^{(1)}_t, \ldots,
      \X^{(n)}_t$ exist for all time $t$. Additionally, if
      \eqref{eq:decomp} holds (or equivalently
      \eqref{eq:Qcondition} holds), then the  maximal time of existence for $\RX^{(n)}$ is
	the same as $\tau_\infty$ almost surely.
  \end{enumerate}
\end{lemma}
\begin{proof} The argument is the same in both cases. We detail the
  first case.
	$\Y^{(0)}_t$, $\Y^{(1)}_t$, $\ldots$, $\Y^{(n)}_t$ exist for all time $t$, 
	because they are linear equations, and the drifts are well-defined for all time.
        If  \eqref{eq:decomp}  holds, we have
        \begin{equation*}
           u_t - \sum_{i=0}^n \Y^{(i)}_t  \eqdist \RY^{(n)}_t,
        \end{equation*}
	so the maximal time of existence for $\RY^{(n)}$ is the same as 
	that of $u$ almost surely. If  \eqref{eq:Qcondition} holds, 
  then \eqref{eq:system2} combined with \eqref{eq:Qcondition}
  implies \eqref{eq:decomp2}.
\end{proof}

\begin{remark}
  Moving forward we will take $u_t$ to be constructed by the
  decomposition in either \eqref{eq:decomp2} or
  \eqref{eq:decomp}. Hence, in light of Lemma~\ref{lem:localExt} , the
  existence time $\tau_\infty$ will be almost surely that of
  $\RY^{(n)}$ and $\RX^{(n)}$. Recalling the definition of $C_T\overline
  \Cc^\delta$ from  Section~\ref{sec:functionSpaces}, we will see in
  \cref{prop:regR} and  \cref{prop:regRX} that $\RY^{(n)}, \RX^{(n)} \in
  C_T\Cc^{\delta\mm}$ for some $\delta >0$, 
  by setting $\RY^{(n)}_t = \death$ for $t \ge \tau_\infty$ and the same for 
  $\RX^{(n)}$.
   Both of these results
  follow from the rather classical existence and uniqueness theory in
  Appendix~\ref{sec:ExistenceRegularity}, once all
  the more singular terms have been properly renormalized to give them meaning. 

\end{remark}

\subsection{Absolute Continuity via Decomposition}
\label{sec:absContiunityViaDecomp}

As already indicated, we will prove  \cref{thm:SimplyMain} using the
decomposition in either \eqref{eq:decomp2}  or \eqref{eq:decomp}. In
the first case, we will prove \eqref{eq:YabsCont} and in the second
case  \eqref{eq:XabsCont}. In both cases, \cref{thm:SimplyMain} will follow from inductively applying the
following lemma.
\begin{lemma}\label{lem:absSum} Let $U$, $U'$, $Z$, and $Z'$ be random
  variables and $\mathcal{G}$ be a $\sigma$-algebra such that $U$ and
  $Z$ are $\mathcal{G}$-measurable. If   $\Law(U) \ll \Law(Z)$ and
  $\Law(U' \mid  \mathcal{G}) \ll \Law(Z' \mid  \mathcal{G})$ almost surely, then
  $\Law(U' + U ) \ll  \Law(Z' + Z )$.
\end{lemma}
\begin{proof}[Proof of \cref{lem:absSum}]
  Let $D$ be any measurable set with $\Pb(Z'+Z \in D)=0$.
  Then since
  \[
    0=\Pb(Z'+Z \in D) = \EE(\Pb(Z' \in D - Z \mid  Z ))
    = \EE(\EE ( \Pb(Z' \in D - Z \mid  \mathcal{G} )\mid  Z)),
  \]
  we conclude that there exists a set $E$ with $\Pb(Z \in E)=1$ such that
  \[
    \EE(\Pb(Z' \in D - Z \mid  \mathcal{G} )\mid  Z = x)
    = \EE(\Pb(Z' \in D - U \mid  \mathcal{G} )\mid  U = x)
    = 0,
  \]
  for all $x \in E$. Since $\Law(U) \ll \Law(Z)$ and $\Law(U' \mid 
  \mathcal{G} ) \ll \Law(Z' \mid  \mathcal{G})$ almost surely,
  we have, respectively, 
  that $\Pb(U \in E)=1$ and  
  \[
    \EE ( \Pb(U' \in D - U \mid  \mathcal{G} )\mid  U = x)
    = 0,
  \]
  for all $x \in E.$ Combining these facts, we see that
  \begin{align*}
    \Pb(U'+U \in D) = \EE(\EE (\Pb(U' \in D - U \mid  \mathcal{G} )\mid  U),
    U \in E) = 0,
  \end{align*}
  which completes the proof.
\end{proof}

\begin{corollary}\label{cor:absXY}
  Assume that, for some $n$, the system of equations~\eqref{eq:system2}
  (respectively \eqref{eq:system}) is well defined and satisfies the
  absolute continuity conditions given in \eqref{eq:YabsCont}
  (respectively  \eqref{eq:XabsCont}). Then in the first case
  \begin{align*}
    \Law\left(\RY_t^{(n)} +\sum_{k=0}^n \Y_t^{(k)} \spc{\Big|} \tau_\infty >t \right)
      \ll   \Law\left(\tZ_t^{(n)} +\sum_{k=0}^n Z_t^{(k)}\right)
  \end{align*}
 holds, and in the second
   \begin{align*}
    \Law\left(\RX_t^{(n)} +\sum_{k=0}^n \X_t^{(k)}\spc{\Big|} \tau_\infty >t \right)
      \ll   \Law\left(\tZ_t^{(n)} +\sum_{k=0}^n Z_t^{(k)}\right)
   \end{align*}
   holds.
 \end{corollary}
 \begin{proof}[Proof of~\cref{cor:absXY}]
   The proof in the two cases is the same. We give the first.
   Since $\Law(\Y_t^{(0)})=\Law(Z_t^{(0)})$ and
   $\Law(\Y_t^{(1)}\mid \mathcal{F}_t^{(0)}) \sim
   \Law(Z_t^{(1)})=\Law(Z_t^{(1)}\mid \mathcal{F}_t^{(0)} )$ almost surely
   where
   $Y_t^{(0)}$ and $Z_t^{(0)}$ are adapted to $\mathcal{F}_t^{(0)}$,
   \cref{lem:absSum} implies that $\Law(\Y_t^{(0)}+\Y_t^{(1)}) \sim
   \Law(Z_t^{(0)}+Z_t^{(1)})$. We proceed inductively. If we have shown
   that $\Law(\sum_{k=0}^m \Y_t^{(k)}) \sim \Law(\sum_{k=0}^m
   Z_t^{(k)})$, then because $\sum_{k=0}^m \Y_t^{(k)}$ and
   $\sum_{k=0}^m Z_t^{(k)}$ are adapted to $\mathcal{F}_t^{(m)}$, the
   fact that  
   $$\Law(\Y_t^{(m)}\mid \mathcal{F}_t^{(m-1)}) \sim
   \Law(Z_t^{(m)})=\Law(Z_t^{(m)}\mid \mathcal{F}_t^{(m-1)} )$$ 
   almost surely implies the
   next step in the induction again using  \cref{lem:absSum} . The
   final step in the proof, namely the step from $ \Law(\sum_{k=0}^n \Y_t^{(k)}) \sim
   \Law(\sum_{k=0}^n Z_t^{(k)})$ 
      uses the condition that either
   to the result quoted in the
   corollary, proceeds analogously to the previous induction steps,
   except that one 
   $\Law(\RY_t^{(n)}\mid  \tau_\infty>t,\mathcal{F}_t^{(n-1)})$ or
   $\Law(\RX_t^{(n)}\mid \tau_\infty>t,\mathcal{F}_t^{(n-1)})$ is absolutely continuous
   with respect to $\Law(\tZ_t^{(m)})$. Here we have applied
   \cref{lem:absSum} to the measures conditioned on $\{\tau_\infty>t\}$.
 \end{proof}

 The proof of the following corollary is completely analogous to that
 of \cref{cor:absXY}.
 \begin{corollary}
   Assume that, for some $n$, the system of equations~\eqref{eq:system2}
and \eqref{eq:system} is well defined with \eqref{eq:XYabsCont}
holding. Then
\begin{align*}
  \Law\left(\sum_{k=0}^n \Y_t^{(k)}\right)
    \sim \Law\left(\sum_{k=0}^n \X_t^{(k)}\right).
\end{align*}
 \end{corollary}

\subsection{Some informal computations}\label{sec:motive}
With tools above, we can informally describe how to 
choose the number of levels and 
the $Q_i$'s in the decomposition \eqref{eq:system2} and \eqref{eq:system}.
We focus on \eqref{eq:system2} because the equation structure 
of $\Y^{(i)}$ and the remainder $\RY^{(n)}$ is aligned, which
makes discussions more intuitive.
But the intuition is the same for \eqref{eq:system2}, and later 
we will see that the result for \eqref{eq:system2} is actually
more straightforward to prove rigorously. 
For simplicity, we do not distinguish between the Cameron-Martin Theorem
or the Time-Shifted Girsanov Method, since they require
the same condition on the canonical regularity, as discussed in \cref{rem:compare}.

Building on the preliminary discussion in \cref{sec:Preliminary},
the first idea is that we assume every Gaussian term, 
i.e. each of the $X^{(0,i)}$, $B(X^{(0,i)})$
or $J(X^{(0,i)})$, in \eqref{eq:system2} can be well-defined with
its canonical regularity. 
As a reminder, it means that $\Y^{(i)}$
has the same H\"older regularity as $Z^{(i)}$, and the $B$ terms
(and $J$ terms) have the canonical regularity following \cref{rem:prod} and
\cref{rem:JClassicalProdReg}. The second idea is that we want $X^{(i)}$ to
become smoother as $i$ increases, so we can take $Q_i \simQ A^{\alpha_i}$
so that $Z^{(i)} \in C_T \Cc^{(\frac{1}{2} - \alpha_i)\mm} $, 
where $\alpha_i$ decreases as $i$ increases. 
It is straightforward to show (later) that we can choose
$Q_i$, $\tQ_n$ satisfying \eqref{eq:Qcondition} with $\alpha_0 = \alpha$.

When $\alpha < \frac{1}{2}$, as discussed in \cref{sec:rangeOfApp}, 
we can directly apply the Time-Shifted Girsanov Method
to \eqref{eq:burgersMild} and obtain $\Law(u_t) \ll \Law(z_t)$. In fact,
since the solutions can be seen to be almost surely global with finite
control of some moments of the norm, one can show that $\Law(u_t) \sim \Law(z_t)$.

When $\alpha \ge \frac{1}{2}$, 
then $u_t \in \Cc^{(\frac{1}{2} - \alpha)\mm}$ is a distribution 
with its canonical regularity, so
we cannot make sense of $u_t^2$ classically. 
We first consider $u_t = \Y^{(0)}_t + \Y^{(1)}_t + \RY^{(1)}_t$ 
in \eqref{eq:system2} for $t < \tau_\infty$. Clearly, $\Y^{(0)} = Z^{(0)}$.
To apply the Cameron-Martin Theorem or the Time-Shifted Girsanov Method on
$\Y^{(1)}$ to show \eqref{eq:YabsCont}, 
the regularity condition is $2\alpha_0 - \alpha_1 < 1.$
On the other hand, note that the remainder $\RY^{(1)}$ is not a Gaussian object.
We want $\RY^{(1)}_t$ to have positive regularity so that $B(\RY^{(1)}_t)$
is well-defined, so we can take $\tQ_1 \simQ A^{\beta_1}$
for some $\beta_1 < \frac{1}{2}.$ 
For convenience of computing regularity, 
we additionally impose $\alpha_1 < \frac{1}{2}$ so that $\tZ^{(1)}$
has positive regularity.
Assume $\RY^{(1)}$ also has its canonical regularity as $\tZ^{(1)}$.
Then $B(\Y^{(0,1)}_t, \RY^{(1)}_t)$ is well-defined classically
if $1 - \beta_1 - \alpha_0 > 0$, and the roughest term in the drift is 
\[
  B(\Y^{(0,1)}) - B(\Y^{(0)}) = 2B(\Y^{(0)}, \Y^{(1)}) 
  \in C_T \Cc^{(-\frac12 - \alpha_0)\mm},
\]
so the Time-Shifted Girsanov condition for $\RY^{(1)}$
is $\frac{1}{2} - \alpha_0 + \beta_1 > 0.$ 
By collecting the above constraints on $\alpha_0$, $\alpha_1$, and $\beta_1$,
we see that as long as $\alpha = \alpha_0 < \frac{3}{4}$,
we can find $\alpha_1$ and $\beta_1$ such that all constraints are satisfied
and 
\begin{align*}
   \Law(u_t\mid  t < \tau_\infty ) &= \Law(\Y^{(0)}_t + \Y^{(1)}_t +
 \RY^{(1)}_t\mid  t < \tau_\infty)\\
 &
 \ll \Law(Z^{(0)}_t + Z^{(1)}_t + \tZ^{(1)}_t) = \Law(z_t) 
\end{align*}
follows from \cref{cor:absXY}.

\begin{remark}
  The careful reader may notice that for $\Y^{(1)}$ and $\RY^{(1)}$ 
  to have their canonical regularity, additional constraints on
  $\alpha_1$ and $\beta_1$ are needed to ensure that the stochastic
  forcing is rougher than the drifts, 
  but one can check and
  will see later that those constraints are implied by the 
  constraints for the Cameron-Martin Theorem/Time-Shifted Girsanov Method.
\end{remark}

The previous informal computation is based on 
the decomposition \eqref{eq:system2} when $n = 1$.
Next, we consider the case $n = 2$, i.e., 
$u_t = \Y^{(0)}_t + \Y^{(1)}_t + \Y^{(2)}_t + \RY^{(2)}_t$.
Again, based on the same reasoning, we take $Q_i \simQ A^{\alpha_i}$, 
$i = 0, 1$ and $\tQ_2 \simQ A^{\beta_2}$, and we assume that $\RY^{(2)}$
has its canonical regularity as $\tZ^{(2)}$ 
and $\alpha_2 < \frac{1}{2} \le \alpha_1 < \alpha_0$ for convenience.
For the same reason as in the case $n = 1$ above, we need $\beta_2 < \frac{1}{2}$,
$2\alpha_0 - \alpha_1 < 1$, $1 - \beta_2 - \alpha_0 > 0$, 
and $\frac{1}{2} - \alpha_0 + \beta_2 > 0.$
Similarly, to show \eqref{eq:YabsCont} on $\Y^{(2)}$, we need
additionally $\alpha_0 + \alpha_1 - \alpha_2 < 1$. 
However, one can check that the above constraints on $\alpha_i, \beta_2$
give no solutions if $\alpha_0 \ge \frac{3}{4}$, and the bottleneck is the constraint
$1 - \beta_2 - \alpha_0 > 0$ for $B(\Y^{(0,2)}_t, \RY^{(2)}_t)$
to be classically well-defined. To resolve this issue, we
employ the Da Prato-Debussche trick of interpreting 
$\RY^{(2)} = \eta^{(2)} + \rho^{(2)}$, where $\eta^{(2)}$ and $\rho^{(2)}$
satisfy
\[
  \Ll \eta^{(2)}_t = \tQ_n\,dW^{(n)}_t,
\]
\begin{multline*}
  \Ll \rho^{(2)}_t =  B(\Y^{(0,2)}_t) - B(\Y^{(0,1)}_t) 
   + 2 B(\Y^{(0,2)}_t, \eta^{(2)}_t) \\+ 2B(\X^{(0,2)}_t, \rho^{(2)}_t)
     + B(\rho^{(2)}_t) + 2B(\rho^{(2)}_t, \eta^{(2)}_t) + B(\eta^{(2)}_t). 
\end{multline*}
with $\eta^{(2)}_0 = 0$ and $\rho^{(2)}_0 = u_0$.
Then we can interpret
\[
  B(\Y^{(0,2)}_t, \RY^{(2)}_t) = 
  B(\Y^{(0,2)}_t, \eta^{(2)}_t) + B(\Y^{(0,2)}_t, \rho^{(2)}_t),
\]
where $B(\Y^{(0,2)}, \eta^{(2)})$ is a well-defined Gaussian object,
and one can check and will see as long as $\alpha_0 < 1$,
$B(\Y^{(0,2)}, \rho^{(2)})$ is classically well-defined.
Now we do not need the constraint $1 - \beta_2 - \alpha_0 > 0$, and
the remaining constraints can be satisfied 
as long as $\alpha = \alpha_0 < \frac{5}{6}.$

Following the heuristics above, we can increase the number $n$ of
levels of the decomposition \eqref{eq:system2} to
obtain the main result up to $\alpha < 1$. The informal computations
above will be justified in a clean and rigorous way in the next
section.

\subsection{Basic assumptions on factorization of noise into levels}

We now fix additional structure in the  $\Y$ and $\X$ systems
(equations \eqref{eq:system2} and \eqref{eq:system} respectively) to allow us to
better characterize the  regularity of the different levels. We assume
that 	there exists a sequence of real numbers 
\begin{equation}
  \label{eq:alpha}
  \beta_n < \alpha_n < \alpha_{n-1} < \ldots < \alpha_0 = \alpha,
\end{equation}
with
\begin{equation}
  \label{eq:condExtra}
  \alpha_n < \frac{1}{2} \le \alpha_{n-1}
      \end{equation}
such that
\begin{equation}
  \label{eq:Qi}
  Q_i \simQ A^{\alpha_i/2}\quad\text{and}\quad \tQ_n  \simQ A^{\beta_n/2}\,.
\end{equation}
We will see that the effect of the assumption in \eqref{eq:alpha} is
to make the levels in \eqref{eq:system2} and \eqref{eq:system} have
increasing spatial regularity as $k$ increases. Conditions
\eqref{eq:alpha}-\eqref{eq:Qi} will be our standing structural
assumption on the noise. 

\begin{remark}[Importance of Condition \eqref{eq:condExtra}]
At first sight, Condition \eqref{eq:condExtra} may seem unnecessary for our main result.
However, it is critical mainly for two reasons:
\begin{enumerate}
	\item In later arguments, Condition \eqref{eq:condExtra} gives a clean break 
	between terms that are functions (positive H\"older regularity) 
	and those that are distributions (non-positive H\"older regularity). 
	In particular, it makes computations for regularity straightforward. 

	\item It makes sure that the drift in the remainder equation,
      $\RY$ or $\RX$, can be defined without being convolved with the heat
    kernel. This in turn allows us to apply the Time-Shifted Girsanov
    Method. This is critical, as the reminder equations have drift
    terms which depend on the solution of the equation. As such,
    we cannot apply the Cameron-Martin Theorem, and tools based
    on the Girsanov Theorem seem the only option. 
\end{enumerate}
\end{remark}

\begin{remark}[First Note on $\alpha < 1$]
	Throughout this note, we implicitly assume $\alpha_0 = \alpha < 1$,
	because in Appendix \ref{sec:Gaussian}, we only construct
	relevant Gaussian objects up to the case of $\alpha < 1$. 
	However, if the condition $\alpha_0 < 1$ is explicitly stated
	in the assumptions of later results, it highlights another 
	non-trivial dependence on this condition.
\end{remark}

The following lemma shows that one can choose
the $\{\alpha_i : i = 0,\dots,n\}$ and $\beta_n$ so that in addition to
\eqref{eq:alpha}-\eqref{eq:Qi} the condition in 
\eqref{eq:Qcondition} holds, which implies that the sum of the stochastic
forcing in \eqref{eq:system2} or \eqref{eq:system} has the same
distribution as that of the Burgers equation in \eqref{eq:burgers}.

\begin{lemma}\label{lem:param}
  For any sequence of real numbers as in \eqref{eq:alpha} and
  any choice of operator $Q$ from \eqref{eq:burgers}
  with  $Q\simQ A^{-\beta/2}$, there exist operators
  $Q_0, Q_1, \ldots, Q_n, \tQ_n$ satisfying \eqref{eq:Qi} and \eqref{eq:Qcondition}.
\end{lemma}
\begin{proof}[Proof of \cref{lem:param}]
  We can take $Q_i = \frac{1}{\sqrt{n+2}}A^{\alpha_i/2}$ for $1 \le i \le n$ 
  and $\tQ_n = \frac{1}{\sqrt{n+2}}A^{\beta_n/2}$. Note that the operator 
  \begin{equation*}
    QQ^* - \tQ_n\tQ_n^* - Q_nQ^*_n - \cdots - Q_2 Q_2^* -  Q_1 Q_1^*
  \end{equation*}
  is symmetric and positive definite, so it is equal to $Q_0 Q_0^*$
  for some operator $Q_0$. Since
  $Q \simQ A^{\alpha/2}$ and $\alpha = \alpha_0$ is the largest
  among $\alpha_i$'s and $\beta_n$, we have $Q_0 \simQ A^{\alpha_0/2}$.
\end{proof}

\begin{remark}[Sums of $Z^{(i)}$]\label{rmk:OU}
	With the proofs of Lemma \ref{lem:param}, 
	we note that for $0 \le i \le n$, 
	\[
		Z^{(0, i)}_t \eqdist \int_0^t e^{-(t - s)A}Q^{(0,i)}dW_s 
	\]
	for some operator $Q^{(0,i)} \simQ A^{\alpha_0/2}$. 
\end{remark}

\subsection{The  \texorpdfstring{$\X^{(i)}$}{X(i)} Equations} We will begin by establishing
the needed structural and desired absolute continuity results for
$\X^{(i)}$. They will be leveraged to prove the results about the $\Y$ system.
\begin{proposition}[Canonical regularity of drifts]\label{lem:regX}
 Under the standing noise factorization assumptions
 \eqref{eq:alpha}-\eqref{eq:Qi}, one has with probability  one
  \begin{equation}\label{eq:regYdrift}
    \quad J(Z^{(0,i-1)}) - J(Z^{(0,i-2)}) \in C_T\Cc^{(2 - \alpha_0 - \alpha_{i-1})\mm}
  \end{equation}
for $0 \le i \le n$.
  In particular, $J(Z^{(0,i-1)}) - J(Z^{(0,i-2)})$
  has the canonical regularity. In addition, the equations for
  $\{\X^{(i)}: i=0,\dots,n\}$ are well posed with global solutions. 
\end{proposition}
\begin{proof}[Proof of \cref{lem:regX}]
  By assumptions \eqref{eq:alpha}-\eqref{eq:Qi}
  and Proposition \ref{prop:regz}, 
  $Z^{(i)} \in C_T\Cc^{(\frac{1}{2} - \alpha_i)\mm}$. 
  The expression 
  \begin{equation}\label{eq:diffGauss}
    J(Z^{(0,i-1)}) - J(Z^{(0,i-2)}) = 
    J(Z^{(i-1)}) + 2\sum_{j=0}^{i-2} J(Z^{(j)}, Z^{(i-1)}) 
  \end{equation}
  is well defined and 
  belongs to $C_T\Cc^{(2 - \alpha_0 - \alpha_{m-1})\mm}$ almost surely, 
  by Appendix $\ref{sec:Gaussian}$, 
  because $J(Z^{(0)}, Z^{(i-1)}) \in C_T\Cc^{(2 - \alpha_0 - \alpha_{m-1})\mm}$ 
  a.s. is the least regular term.  
\end{proof}

\begin{proposition}[Constraints from canonical regularity of $\X^{(i)}$]\label{prop:regCon}
Under the standing noise factorization assumptions
 \eqref{eq:alpha}-\eqref{eq:Qi},
	if in addition $\alpha_0 + \alpha_{i - 1} - \alpha_i < \frac{3}{2}$ 
	for all $1 \le i \le n$, then all of the
        $\X^{(i)}$ equations have the canonical regularity, 
        $$\X^{(i)} \in C_T\Cc^{(\frac{1}{2} - \alpha_i)\mm},$$
        namely
        that of the stochastic convolution in each equation. 
\end{proposition}
\begin{proof}[Proof of \cref{prop:regCon}]
	For any $1 \le i \le n$ and $t > 0$, we only need to make sure that the drift 
	\[
		J(Z^{(0,i-1)}) - J(Z^{(0,i-2)}) 
		\in C_T \Cc^{(2-\alpha_0 - \alpha_{i-1})\mm}\quad \text{a.s.} 
	\] 
	is smoother than the stochastic convolution 
	\[
		Z^{(i)} \in C_T\Cc^{(\frac{1}{2} - \alpha_i)\mm} \quad \text{a.s.}
	\]
	so that $\X^{(i)}$ has the same regularity as the stochastic convolution. 
	This holds if $2 -\alpha_0 - \alpha_{i-1} >  \frac{1}{2} - \alpha_i.$ 
\end{proof}

We will apply the Cameron-Martin \cref{lem:CM} in our setting, 
to each level by conditioning on previous levels. 
As mentioned in \cref{rem:compare} and \cref{sec:rangeOfApp},
the reason why we cannot apply Time-Shifted Girsanov Method 
to some of the levels is that some of the terms involving $Z^{(i)}$ 
can only be defined when convolving with the heat kernel. 
For example, we cannot define $B(Z^{(0)}_t)$ but can only define $J(Z^{(0)})_t$ 
when $\alpha_0 \ge \frac{3}{4}$ (see Appendix \ref{sec:Gaussian} for more details). 
In this case, the Time-Shifted Girsanov Method, \cref{thm:tsG}, cannot be applied.

\begin{proposition}[Constraints from Cameron-Martin for $\X^{(i)}$]
\label{prop:GirCon}
Under the standing noise factorization assumptions
 \eqref{eq:alpha}-\eqref{eq:Qi}, if 
	$\alpha_0 + \alpha_{i-1} - \alpha_i < 1$ for all $1 \le i \le n$, 
	then the regularity conditions, given in \cref{rem:CMApplicablity}, needed to apply the Cameron-Martin Theorem~\ref{lem:CM} hold. 
	More concretely, it implies that for $1 \le i \le n$, for any $t > 0$,
	it holds almost surely that
	$$ \Law(\X_t^{(i)}\mid \mathcal{F}_t^{(i-1)}) \sim \Law(Z_t^{(i)}), $$
	where we recall 
	$\mathcal{F}_t^{(i-1)} = \sigma( W_s^{(j)}: j \leq i-1, s \leq t)$.
\end{proposition}
\begin{proof}[Proof of \cref{prop:GirCon}]
	For each $1 \le i \le n$, we have 
	\[
		\X^{(i)}_t = J(Z^{(0,i-1)})_t - J(Z^{(0,i-2)})_t + Z^{(i)}_t. 
	\]	
	Since $Q^{(i)} \simQ A^{\alpha_i / 2}$, 
	by \eqref{eq:regYdrift}, 
	the condition from Theorem~\ref{lem:CM} 
	and Remark \ref{rem:CMApplicablity}
 	for the equation of $\X^{(i)}$ is exactly 
	$\alpha_0 + \alpha_{i-1} - \alpha_i < 1$. 
\end{proof}

\begin{remark}[Redundant constraints]\label{rem:redun}
	It is clear that the parameter constraints in Proposition \ref{prop:GirCon} 
	imply those in Proposition \ref{prop:regCon}. 
\end{remark}

With all the constraints so far, we establish the relation between
the range of $\alpha$ and the corresponding number $n$ of levels 
needed in the decomposition (except for the remainder).

\begin{proposition}[Choosing the number of levels $n$ in $\{\X^{(i)}\}$]\label{cor:constntX}
	Fix an $n$  and an $\alpha$ so that $\frac{1}{2} \le \alpha < \frac{2n+1}{2n+2}$. Then there exists 
	a sequence real numbers $\alpha_n < \ldots < \alpha_0 = \alpha$ 
	such that the standing noise factorization assumptions on the
        $\{\alpha_j: j=0,\dots,n\}$ in
 \eqref{eq:alpha}-\eqref{eq:Qi} hold as well as the hypothesis of
 \cref{lem:regX}, \cref{prop:regCon}, and \cref{prop:GirCon}.
\end{proposition}
\begin{proof}[Proof of \cref{cor:constntX}]
	First, we make sure the assumption in \eqref{eq:condExtra} is satisfied. 
	Based on Proposition \ref{prop:regCon}, \ref{prop:GirCon} 
	and Remark \ref{rem:redun}, we only need to make sure 
	$\alpha_0 + \alpha_{i-1} - \alpha_i < 1$ for any $1 \le i \le n$. 
	In particular, we have $\alpha_0 + \alpha_{i-1} - 1 < \alpha_i$, 
	which implies 
	\begin{align*}
		\alpha_1 &> 2\alpha_0 - 1 \\
		\implies \alpha_2 &> \alpha_1 + \alpha_0 - 1 > 3\alpha_0 - 2 \\
		\implies \alpha_3 &> \alpha_2 + \alpha_0 - 1 > 4\alpha_0 - 3 \\
		\vdots \\ 
		\implies \alpha_n &> \alpha_{n-1} + \alpha_0 - 1 > (n+1)\alpha_0 - n.
	\end{align*} 
	Since $\alpha_n < \frac{1}{2}$, 
	we deduce $\frac{1}{2} \le \alpha_0$ 
	and $(n+1)\alpha_0 - n < \frac{1}{2}$, 
	so $\frac{1}{2} \le \alpha = \alpha_0 < \frac{2n+1}{2n+2}$, 
	and starting from this constraint, 
	we may find the possible values of $\alpha_1, \ldots, \alpha_n$. 
\end{proof}

\begin{remark}[Second Note on $\alpha < 1$]
  From \cref{cor:constntX}, we see that when $\alpha < 1$, 
  each level of the decomposition ``gains'' regularity of the gap $1 - \alpha$,
  and this gap is crucial for our method to work. 
\end{remark}

\subsection{Analysis of Remainder \texorpdfstring{$\RX^{(n)}$}{S(n)} and Associated Constraints}
Recall the remainder equation from \eqref{eq:system}: 
\begin{multline}\label{eq:R}
	\Ll \RX^{(n)}_t = \big(B(\X^{(0,n)}_t) - B(Z^{(0,n-1)}_t) \\+ 2 B(\X^{(0,n)}_t, \RX^{(n)}_t) 
		+ B(\RX^{(n)}_t)\big)dt + \tQ_nd\tW ^{(n)}_t. 
\end{multline}
Note that in this equation, the drift depends on the solution $\RX^{(n)}$
itself, so we cannot apply the Cameron-Martin \cref{lem:CM} by conditioning on previous levels. 
However, unlike the previous level, the drift is regular enough so that 
it can be defined without convolving with the heat kernel. 
This is reflected by the fact that $\alpha + \alpha_n < \frac{3}{2}$
and $\alpha + \beta_n < \frac{3}{2}.$
We are in good position to use the Time-Shifted Girsanov Method, \cref{thm:tsG}. 

We first study the well-posedness of $\RX^{(n)}$ and the canonical regularity of the terms.
We start with the term $B(\X^{(0,n)}_t) - B(Z^{(0,n-1)}_t)$.
\begin{proposition}\label{lem:regR}
Under the standing noise factorization assumptions
 \eqref{eq:alpha}-\eqref{eq:Qi}, if the $\X^{(i)}$ equations are well
 posed with all of their terms possessing the canonical regularity (as
 guaranteed for example by \cref{lem:regX} and \cref{prop:regCon}), then
 if additionally $\alpha_0 < 1$, with probability one for any $t > 0$, we have
	\begin{equation}\label{eq:Rdrift}
          B(\X^{(0,n)}) - B(Z^{(0,n-1)}) \in C_T\Cc^{(-\frac{1}{2}-\alpha_0)\mm}. 
      \end{equation}
      That is to say these terms are well defined with their canonical
      regularity.
\end{proposition}
\begin{proof}[Proof of \cref{lem:regR}]
	Note that $$\X^{(0,n)} = J(Z^{(0, n-1)}) + Z^{(0,n)},$$ which yields 
	\[
		\begin{split}
			 B(\X^{(0,n)}) - B(Z^{(0,n-1)}) & = B\big(J(Z^{(0, n-1)}) + Z^{(0,n)}\big) - B(Z^{(0,n-1)}) \\ 
			& = B\big(J(Z^{(0, n-1)})\big) + 2 B\big(J(Z^{(0, n-1)}), Z^{(0,n)}\big)\\ 
				&\qquad+ B( Z^{(0,n)}) - B(Z^{(0,n-1)}), 
		\end{split}
	\]
	and to leverage independence among the $Z^{(i)}$, we note that 
	\[
		B(J(Z^{(0, n-1)}), Z^{(0,n)}) = 
			B(J(Z^{(0, n-1)}), Z^{(0,n-1)}) + B(J(Z^{(0, n-1)}), Z^{(n)}), 
	\]
	\[
		B(Z^{(0,n)}) - B(Z^{(0,n-1)}) = 
    	B(Z^{(n)}) + 2 B(Z^{(0,n-1)}, Z^{(n)}). 
	\]
	By $\alpha_0 < 1, \alpha_n < \frac{1}{2}$, Remark \ref{rmk:OU}, 
	and Appendix \ref{sec:Gaussian}, 
	\[
		\begin{split}
			B(J(Z^{(0, n-1)})) & \in C_T\Cc^{(1 - 2\alpha_0)\mm}, \\  
			B(J(Z^{(0, n-1)}), Z^{(0,n)}) & \in C_T\Cc^{(-\frac{1}{2} - \alpha_0)\mm}, \\
			B(Z^{(0,n)}) - B(Z^{(0,n-1)}) & \in C_T\Cc^{(-\frac{1}{2} - \alpha_0)\mm},
		\end{split}
	\]
	almost surely. 
	Therefore, since $\alpha_0 < 1$, 
	\[
		B(\X^{(0,n)}) - B(Z^{(0,n-1)}) \in C_T\Cc^{(-\frac{1}{2} - \alpha_0)\mm},
	\]
	almost surely. 
\end{proof}

Unlike the system \eqref{eq:system} of $\X^{(i)}$, 
where all terms are Gaussian objects, 
some product terms in \eqref{eq:R} may not be a priori well-defined. 
Assume for the moment that $\RX^{(n)}$ is well-defined 
with its canonical regularity.
From the assumptions of Lemma \ref{lem:regR}, 
we have $\beta_n < \alpha_n < \frac{1}{2}$ and 
so $\RX^{(n)}_t$ is function-valued almost surely, 
which makes $B(\RX^{(n)}_t)$ well-defined and 
belong to $\Cc^{(-\frac{1}{2} - \beta_n)\mm}$ almost surely. 
The only term we need to define appropriately is $B(\X^{(0,n)}_t, \RX^{(n)}_t)$. 

As motivated in Section \ref{sec:motive}, 
we interpret the term $B(\X^{(0,n)}_t, \RX^{(n)}_t)$ as 
\[
	B(\X^{(0,n)}_t, \eta^{(n)}_t) + B(\X^{(0,n)}_t, \rho^{(n)}_t), 
\]
where $\eta^{(n)}$ and $\rho^{(n)}$ 
solve 
\begin{equation*}
	\Ll \eta^{(n)}_t = \tQ_n d\tW ^{(n)}_t, 
\end{equation*}
\begin{multline}\label{eq:rho}
		\Ll \rho^{(n)}_t =  B(\X^{(0,n)}_t) - B(Z^{(0,n-1)}_t) 
	 + 2 B(\X^{(0,n)}_t, \eta^{(n)}_t) \\+ 2B(\X^{(0,n)}_t, \rho^{(n)}_t)
	   + B(\rho^{(n)}_t) + 2B(\rho^{(n)}_t, \eta^{(n)}_t) + B(\eta^{(n)}_t), 
\end{multline}
with $\eta^{(n)}_0 = 0$ and $\rho^{(n)}_0 = u_0$. 
Note that the stochastic forcing term $\eta^{(n)}$ is
 just $\tZ^{(n)}$ but with zero initial condition.
In this case, $B(\X^{(0,n)}, \eta^{(n)})$ can be defined 
with Appendix \ref{sec:Gaussian}, since $\eta^{(n)}$ is an 
Ornstein-Uhlenbeck process with zero initial condition. 
On the other hand, as we remove the stochastic forcing term, 
$\rho^{(n)}$ has better regularity so that 
$B(\X^{(0,n)}, \rho^{(n)})$ can be classically defined
with appropriate choice of parameters.

\begin{lemma}\label{lem:rhoeta}
  Assume the standing noise factorization assumptions
 \eqref{eq:alpha}-\eqref{eq:Qi} hold. If the $\X^{(i)}$ equations are well
 posed with all of their terms possessing the canonical regularity, then with probability one 
	\begin{equation} \label{eq:regeta}
		\eta^{(n)} \in C_T\Cc^{(\frac{1}{2}-\beta_n)\mm}, 
		\quad B(\X^{(0,n)}, \eta^{(n)}) \in C_T\Cc^{(-\frac{1}{2} - \alpha_0)\mm}. 
	\end{equation}
	If $\alpha_0 < 1$, then $\rho^{(n)}$ and 
	$B(\X^{(0,n)}, \rho^{(n)})$ are well-defined locally in time, 
	and for any $T < \tau_\infty$, 
	the maximal existence time, 
	\begin{equation} \label{eq:regrho}
		\rho^{(n)} \in C_T \Cc^{(\frac{3}{2}-\alpha_0)\mm}, 
		\quad B(\X^{(0,n)}, \rho^{(n)}) \in C_T\Cc^{(-\frac{1}{2} - \alpha_0)\mm}. 
	\end{equation}
	In particular, 
	$ \RX^{(n)} = \eta^{(n)} + \rho^{(n)} $ is well-defined locally in time
	 and for $T < \tau_\infty$, 
	\[
		B(\X^{(0,n)}, \RX^{(n)}) =  B(\X^{(0,n)}, \eta^{(n)}) 
		+ B(\X^{(0,n)}, \rho^{(n)}) \in C_T\Cc^{(-\frac{1}{2} - \alpha_0)\mm}. 
	\]
\end{lemma}
\begin{proof}[Proof of \cref{lem:rhoeta}]
	The first statement is proved as before 
	with Remark \ref{rmk:OU} and Appendix \ref{sec:Gaussian}. 
	
	We turn to the second statement. 
	Note that we can rewrite \eqref{eq:rho} as
	\[
		\Ll \rho^{(n)}_t = B(\rho^{(n)}_t) + 2B(\rho^{(n)}_t, \eta_t^{(n)}) + F^{(n)}_t,
	\]
	where $F^{(n)}$ is given by
	\[
		F^{(n)}_t := B(\X^{(0,n)}_t) - B(Z^{(0,n-1)}_t)
			+ 2 B(Y^{(0,n)}_t, \eta^{(n)}_t) + B(\eta^{(n)}_t).
	\]
	By \eqref{eq:regeta}, we know $\eta^{(n)} \in C_T \Cc^{(\frac{1}{2} - \beta_n)\mm}$.
	By Appendix \ref{sec:Gaussian}, 
	since $\alpha_0 + \beta_n < \frac{3}{2}$,
	it is clear that 
	\[
		B(Y^{(0,n)}, \eta^{(n)}) \in C_T \Cc^{(-\frac{1}{2} - \alpha_0)\mm}.
              \]
	Hence, by \eqref{eq:Rdrift}, we have   
	$F^{(n)} \in C_T \Cc^{(-\frac{1}{2} - \alpha_0)\mm}.$
	By a standard fixed point argument in Appendix \ref{sec:ExistenceRegularity},
	we obtain that equation  \eqref{eq:rho} is well posed with
        local in time solutions such that $\rho^{(n)} \in C_T \Cc^{(\frac{3}{2}-\alpha_0)\mm}.$
	In particular, since $\alpha_0 < 1$,  
	$B(\X^{(0,n)}, \rho^{(n)}) \in C_T\Cc^{(-\frac{1}{2} - \alpha_0)\mm} $
	is well-defined classically.
\end{proof}

Now we are in a good place to figure out the needed constraints for
$\RY^{(n)}$ to have its canonical regularity and for 
the application of the Time-Shifted Girsanov Method.

\begin{proposition}[Constraint from the canonical regularity for $\RX^{(n)}$]
	\label{prop:regR}
	Under the standing noise factorization assumptions
 	\eqref{eq:alpha}-\eqref{eq:Qi},
 	if the $\X^{(i)}$ equations are well
 	posed with all of their terms possessing the canonical regularity,
	then if in addition $\alpha_0 < 1$ and $\alpha_0 - \beta_n < 1$, 
	then $\RX^{(n)}$ has the canonical regularity, 
	\[	
		\RX^{(n)} \in C_T\Cc^{(\frac{1}{2}-\beta_n)\mm}
	\]
	for any $T < \tau_\infty$,
	namely 
    that of the stochastic convolution in the equation.  Setting $\RX_t^{(n)}=\death$
  for $t \geq \tau_\infty$, we have that $ \RX^{(n)} \in C_T\overline
  \Cc^{(\frac{1}{2}-\beta_n)\mm}$ for any $T>0$ (see
  Section~\ref{sec:functionSpaces} for the definition of $C_T\overline
  \Cc^\delta$).
\end{proposition} 
\begin{proof}
	By Lemma \ref{lem:rhoeta}, we have $\RX^{(n)} = \eta^{(n)} + \rho^{(n)}$,
	where $\rho^{(n)} \in C_T \Cc^{(\frac{3}{2}-\alpha_0)\mm}$
	and $\eta^{(n)} \in C_T\Cc^{(\frac{1}{2}-\beta_n)\mm}$
	for $T < \tau_\infty$, the maximal existence time.
	Since $\eta^{(n)}$ is exactly the stochastic convolution in the equation,
	it suffices to have $\eta^{(n)}$ less regular than $\rho^{(n)}$,
	which is guaranteed by the condition 
	$\frac{3}{2} - \alpha_0 > \frac{1}{2} - \beta_n$.
\end{proof} 

\begin{proposition}[Constraint from Time-Shifted Girsanov for $\RX^{(n)}$]
	\label{prop:GirR}
	Under the standing noise factorization assumptions
 	\eqref{eq:alpha}-\eqref{eq:Qi},
 	if the $\X^{(i)}$ equations are well
 	posed with all of their terms possessing the canonical regularity,
	and if in addition $\alpha_0 < 1$ and $\alpha_0 - \beta_n < \frac{1}{2}$, 
	then the regularity conditions needed to apply Time-Shifted
	Girsanov Method to $\RX^{(n)}$ holds. 
	More concretely, it implies that for any $t > 0$,
	it holds almost surely that 
	$$ \Law(\RX_t^{(n)}\mid t < \tau_\infty,\mathcal{F}_t^{(n)}) 
	\ll \Law(\widetilde{Z}_t^{(n)}), $$
	where we recall 
	$\mathcal{F}_t^{(n)} = \sigma( W_s^{(j)}: j \leq n, s \leq t)$.

	In particular, as long as $\alpha_0 < 1$, 
	 $\beta_n$ (and $\alpha_n$) can be taken close enough to $\frac{1}{2}$ 
	 to satisfy the condition $\alpha_0 - \beta_n < \frac{1}{2}$.
\end{proposition} 
\begin{proof}
	In the setting of \cref{lem:regR} and \cref{lem:rhoeta}, 
	we see that the roughest drift term in \eqref{eq:R} is 
	in $C_T\Cc^{(-\frac{1}{2}-\alpha_0)\mm}$,
	so the Time-Shifted Girsanov condition from \cref{rem:TSG} and
        the associate \cref{thm:tsG} 
	is satisfied if $\frac{1}{2} - \alpha_0 + \beta_n > 0.$  
\end{proof}

\begin{remark}[Redundant constraint]
	As in Remark \ref{rem:redun}, 
	the constraint for the Time-Shifted Girsanov Method 
	$\alpha_0 - \beta_n < \frac{1}{2}$ in \cref{prop:GirR}
	also implies that $\alpha_0 - \beta_n < 1$ in \cref{prop:regR},
	i.e., $\RX^{(n)}$ has the canonical regularity. 
\end{remark}

\begin{remark}[Third Note on $\alpha < 1$]
  We reiterate where $\alpha = \alpha_0 < 1$ is needed for the analysis of 
  the remainder $S^{(n)}$:
  \begin{enumerate}
    \item Together with \eqref{eq:alpha}-\eqref{eq:condExtra}, i.e.,
    $\alpha + \alpha_n < \frac{3}{2}$ and $\alpha + \beta_n < \frac{3}{2}$,
    it makes sure that various $B(f, g)$ terms, 
    such as \eqref{eq:Rdrift} and \eqref{eq:regeta},
    are well-defined with their 
    canonical regularity, where $f$ and $g$ are Gaussian objects.

    \item It makes sure that $B(\X^{(0,n)}) - B(Z^{(0,n-1)})$
    has the same regularity as $B(Z^{(0,n)}) - B(Z^{(0,n-1)})$
    so that \eqref{eq:Rdrift} holds. This corresponds to 
    the regularization effect of J as discussed in \cref{rem:Jregs}.

    \item It makes sure that the equation \eqref{eq:rho} is
    well posed with \eqref{eq:regrho} holds.
  \end{enumerate}
\end{remark}

Finally, by collecting all results above, we prove our main absolute
continuity of the law of $u_t$ with respect to the law of $z_t$ defined in \eqref{eq:OU},
for $\alpha < 1$.

\begin{corollary}[The overall result on the $\X$ system]
	Fix an $n$ and an $\alpha$ so that $\frac{1}{2} \le \alpha < \frac{2n+1}{2n+2}$. Then there exists 
	a sequence real numbers $\beta_n < \alpha_n < \ldots < \alpha_0 = \alpha$ 
	such that the standing noise factorization assumptions on the
        $\{\alpha_j: j=0,\dots,n\}$ in
 \eqref{eq:alpha}-\eqref{eq:Qi} hold as well as the hypothesis of
 \cref{lem:regX}, \cref{prop:regCon}, \cref{prop:GirCon},
 \cref{lem:regR}, \cref{lem:rhoeta}, \cref{prop:regR}, and \cref{prop:GirR}.
 More concretely, it implies that for any $t > 0$, 
 it holds almost surely that
 \begin{align*}
 	\Law(u_t \mid  \tau_\infty>t) &= \Law\left(\RX_t^{(n)}
                                          +\sum_{k=0}^n
                                          \X_t^{(k)}\spc{\Big|}
                                          \tau_\infty>t\right) \\ &\ll 
 	\Law\left(\tZ_t^{(n)} +\sum_{k=0}^n Z_t^{(k)}\right) = \Law(z_t),
 \end{align*}
 where, as a reminder, $z_t$ is the linear part of $u_t$
 with initial condition $z_0$ as defined in \eqref{eq:OU}.
 \end{corollary}
\begin{proof}
	The constraint from $\RX^{(n)}$ can be satisfied 
	as long as $\alpha = \alpha_0 < 1$. 
	The result follows from \cref{cor:constntX} and \cref{cor:absXY}. 
\end{proof}

\section{The \texorpdfstring{$\Y$}{X} decomposition of Noise and Smoothness}\label{sec:Xdecomposition}
Now we consider the $\Y$ system \eqref{eq:system2}.
We note that the maximal existence time $\tau_\infty$ of solutions of \eqref{eq:system2}
is the same as $u$, as in Lemma \ref{lem:localExt}.
We want to show that the same noise factorization assumptions 
\eqref{eq:alpha}-\eqref{eq:Qi} on the system \eqref{eq:system2}
also give the desired absolute continuity result \eqref{eq:XYabsCont}.
Since all computations are based on the canonical regularity,
which is dictated by the same stochastic forcing terms,
we may follow the same arguments of the previous section with minimal modifications.
The main change is the need to make sense of products of 
more complicated Gaussian objects $\Y^{(0,i)}$. 
The idea is that when $\alpha = \alpha_0 < 1$, 
the singular terms have positive regularity after convolving
with the heat kernel once or twice. 
Hence, most products can be classically defined,
and the remaining ones are exactly those appeared before.   

\begin{lemma}[Canonical regularity of $\Y^{(i)}$ and drifts]
\label{lem:regY}
	Under the standing noise factorization assumptions
 	\eqref{eq:alpha}-\eqref{eq:Qi},
	if in addition $\alpha_0 < 1$ and 
	$\alpha_0 + \alpha_{i - 1} - \alpha_i < \frac{3}{2}$,
	for $1 \le i \le n$, then it holds a.s. that for $0 \le i \le n$,
	\begin{equation}\label{eq:regY}
		\Y^{(i)} \in C_T \Cc^{(\frac{1}{2} - \alpha_i)\mm}, 
	\quad J(\Y^{(0,i-1)}) - J(\Y^{(0,i-2)}) \in C_T\Cc^{(2-\alpha_0 - \alpha_{i-1})\mm}.
	\end{equation}
	In particular, the terms are well defined with their canonical
	regularity.
\end{lemma}
\begin{proof}
	With the same argument in \cref{prop:regCon}, 
	as long as each term in the equation for $\Y^{(i)}$ is well-defined,
	it holds a.s. that 
	$\Y^{(i)} \in C_T\Cc^{(\frac{1}{2} - \alpha_i)\mm}.$ 
	Hence, we can focus on the drifts in \eqref{eq:system2}. 
	We proceed by (finite) induction. 

	We start with base cases.
	For $i = 0$, clearly $\Y^{(0, 0)} = \Y^{(0)} = Z^{(0)}$, 
	and it holds a.s. that
	$\Y^{(0)} \in C_T \Cc^{(\frac{1}{2} - \alpha_0)\mm}$. 
	For $i = 1$, $J(\Y^{(0)}) = J(Z^{(0)})$ is well-defined by Appendix \ref{sec:Gaussian},
	and it holds a.s. that 
	\[
		J(\Y^{(0)}) \in C_T\Cc^{(2-2\alpha_0)\mm}.
	\]
	For the purpose of induction, we also note that
	\[
		B(J(\Y^{(0)}), Z^{(0,1)}) = B(J(Z^{(0)}), Z^{(0,1)}) 
		= B(J(Z^{(0)}), Z^{(0)}) +  B(J(Z^{(0)}), Z^{(1)})
	\]
	is well-defined by Remark \ref{rmk:OU} and Appendix \ref{sec:Gaussian},
	and it holds a.s. that 
	\[
		B(J(\Y^{(0,0)}), Z^{(0,1)}) \in C_T\Cc^{(-\frac{1}{2} - \alpha_0)\mm}.
	\]

	Next, we show our induction step. 
	Assume for $0 \le j \le i < n$, each term in the equation for $\Y^{(j)}$
	is well-defined, and it holds a.s. that 
	\[
		J(\Y^{(0,j-1)}) - J(\Y^{(0,j-2)}) \in C_T\Cc^{(2-\alpha_0 - \alpha_{j-1})\mm},
		\quad B(J(\Y^{(0,j)}), Z^{(0,j+1)}) \in C_T\Cc^{(-\frac{1}{2} - \alpha_0)\mm}.
	\]
	We want to show  
	\begin{equation} \label{eq:jxjx}
    \begin{split}
		J(\Y^{(0,i)}) - J(\Y^{(0,i-1)}) & \in C_T\Cc^{(2-\alpha_0 - \alpha_{i})\mm}, \\
		 B(J(\Y^{(0,i)}), Z^{(0,i+1)}) & \in C_T\Cc^{(-\frac{1}{2} - \alpha_0)\mm}.
    \end{split}
	\end{equation}
	
  We start with proving the first part of \eqref{eq:jxjx}. Note that
	\[
		\Y^{(0,i)} = J(\Y^{(0,i-1)}) + Z^{(0,i)}.
	\]
	We can rewrite
	\begin{align*}
		J(\Y^{(0,i)}) & = J(J(\Y^{(0,i-1)}) + Z^{(0,i)}) \\
		& = J(J(\Y^{(0,i-1)})) + 2J(J(\Y^{(0,i-1)}), Z^{(0,i)}) + J(Z^{(0,i)}),
	\end{align*}
	which gives
	\begin{equation}
	\begin{aligned}\label{eq:driftDecomp}
		 \lefteqn{J(\Y^{(0,i)}) - J(\Y^{(0,i-1)})}\hspace{3em}&\\
       &= J(J(\Y^{(0,i-1)}) + Z^{(0,i)}) - J(J(\Y^{(0,i-2)}) + Z^{(0,i-1)}) \\
		   & = J(J(\Y^{(0,i-1)})) - J(J(\Y^{(0,i-2)})) + J(Z^{(0,i)}) - J(Z^{(0,i-1)})\\
		   & \quad	\, + 2(J(J(\Y^{(0,i-1)}), Z^{(0,i)}) 
			- J(J(\Y^{(0,i-2)}), Z^{(0,i-1)})),
	\end{aligned}
	\end{equation}
  so it suffices to show the existence and regularity of each term in \eqref{eq:driftDecomp}.
	For $j < i$, using the induction hypothesis,
	we can define $J(\Y^{(0,j)})$ by the telescoping sum
	\[
		J(\Y^{(0,j)}) = \sum_{\ell = 0}^{j} J(\Y^{(0,\ell)}) - J(\Y^{(0,\ell - 1)}).
	\]
	Since $J(\Y^{(0,j)})$ has the regularity of $J(\Y^{(0)})$, 
	the roughest term in the sum, and $\alpha < 1$,
	$B(J(\Y^{(0,j)}))$ is well-defined classically and 
	\begin{equation} \label{eq:bjx}
		B(J(\Y^{(0,j)})) \in C_T\Cc^{(1-2\alpha_0)\mm}.
	\end{equation}
	Also, by the induction hypothesis, we know 
	\begin{equation} \label{eq:bjxz}
		B(J(\Y^{(0,i-1)}), Z^{(0,i)}) 
		- B(J(\Y^{(0,i-2)}), Z^{(0,i-1)}) 
		\in C_T\Cc^{(-\frac{1}{2} - \alpha_0)\mm}.
	\end{equation}
	By \cref{lem:regX}, we have
	\[
		J(Z^{(0,i)}) - J(Z^{(0,i-1)}) \in C_T\Cc^{(2-\alpha_0 - \alpha_i)\mm}.
	\]
	Since $\alpha_{i} > \frac{1}{2}$,
  based on \eqref{eq:bjx} and \eqref{eq:bjxz},
	$J(\Y^{(0,i)}) - J(\Y^{(0,i-1)})$ is well-defined with the same regularity
	of $J(Z^{(0,i)}) - J(Z^{(0,i-1)})$, which is the roughest term in \eqref{eq:driftDecomp}.

	To finish the induction step, we show the second part of \eqref{eq:jxjx}. 
  As before, we only need to work with each term in the following expansion
	\begin{align*}
		B(J(\Y^{(0,i)}), Z^{(0,i+1)})  
    = \, & B(J(J(\Y^{(0,i-1)})+ Z^{(0,i)} ), Z^{(0,i+1)}) \\ 
		= \, & B(J^2(\Y^{(0,i-1)})) + B(J(Z^{(0,i)}), Z^{(0,i+1)}) \\
		  & \quad \, + 2B(J(J(\Y^{(0,i-1)}), Z^{(0,i)}), Z^{(0,i+1)}).
	\end{align*}
	Similarly, since $\alpha_0 < 1$, 
	$B(J^2(\Y^{(0,i-1)}))$ is well-defined classically and 
	\[
		B(J^2(\Y^{(0,i-1)})) \in C_T\Cc^{(2-2\alpha_0)\mm}.
	\]
	Again by \eqref{eq:alpha}-\eqref{eq:Qi}, \eqref{eq:diffGauss}, 
  and Appendix \ref{sec:Gaussian},
	\[
		B(J(Z^{(0,i)}), Z^{(0,i+1)}) \in C_T\Cc^{(-\frac{1}{2} - \alpha_0)\mm}.
	\]
	By the induction hypothesis, we know
	\[
		J(J(\Y^{(0,i-1)}), Z^{(0,i)}) \in C_T\Cc^{(\frac{3}{2} - \alpha_0)\mm}.
	\]
	Again, since $\alpha_0 < 1$ and 
	$Z^{(0,i+1)} \in C_T \Cc^{(\frac{1}{2}-\alpha_0)\mm}$,
	the following term is classically well-defined: 
	\[
		B(J(J(\Y^{(0,i-1)}), Z^{(0,i)}), Z^{(0,i+1)}) 
		\in C_T\Cc^{(-\frac{1}{2}-\alpha_0)\mm},
	\]
	because the sum of the regularity of the two terms in the product is positive.
	Therefore, $B(J(\Y^{(0,i)}), Z^{(0,i+1)})$ is well-defined with the
	desired regularity.
\end{proof}

Consider proving \eqref{eq:XYabsCont} for each level $X^{(i)}$ of \eqref{eq:system2} 
before the remainder $R^{(n)}$, where
\[
	X^{(i)}_t = J(X^{(0,i-1)})_t - J(X^{(0,i-2)})_t + Z^{(i)}_t.
\]
In the decomposition \eqref{eq:driftDecomp}, the term
$J(Z^{(0,i-1)}) - J(Z^{(0,i-2)})$ is exactly the drift
in the $Y^{(i)}$ equation \eqref{eq:system}, 
which gives the same 
constraint, as in the assumption of \cref{prop:GirCon}, 
for the Cameron-Martin theorem 
in \cref{lem:CM} and \cref{rem:CMApplicablity}.
Since the remaining terms in \eqref{eq:driftDecomp} at time $t > 0$
are smoother than the term
$J(Z^{(0,i-1)})_t - J(Z^{(0,i-2)})_t$ and 
 adapted to $\mathcal{F}^{(i-1)}_t$, 
they also satisfy
the condition for
the Cameron-Martin \cref{lem:CM}. 
Thus, we arrive at the same constraint for parameters
as in the assumption of \cref{prop:GirCon}.

\begin{proposition}
\label{prop:absContY}
Under the standing noise factorization assumptions
 \eqref{eq:alpha}-\eqref{eq:Qi}, if 
  $\alpha_0 + \alpha_{i-1} - \alpha_i < 1$ for all $1 \le i \le n$, 
  then the regularity conditions needed to apply \cref{lem:CM}, the Cameron-Martin Theorem, hold for $\Y^{(i)}$. 
  More concretely, it implies that for $1 \le i \le n$, for any $t > 0$,
  it holds almost surely that
  $$ \Law(\Y_t^{(i)}\mid \mathcal{F}_t^{(i-1)}) \sim
  \Law(\X_t^{(i)}\mid \mathcal{F}_t^{(i-1)}) \sim \Law(Z_t^{(i)}), $$
  where we recall 
  $\mathcal{F}_t^{(i-1)} = \sigma( W_s^{(j)}: j \leq i-1, s \leq t)$.
\end{proposition}

For the remainder, recall from (\ref{eq:system2}) that
\begin{multline}\label{eq:R2}
	\Ll \RY^{(n)}_t = \big(B(\Y^{(0,n)}_t) - B(\Y^{(0,n-1)}_t) 
	\\	+ 2 B(\Y^{(0,n)}_t, \RY^{(n)}_t)
		+ B(\RY^{(n)}_t)\big)dt + \tQ_nd\tW ^{(n)}_t
\end{multline}
with initial condition $u_0$. It remains to make sense of the term 
$B(\Y^{(0,n)}) - B(\Y^{(0,n-1)})$
with its canonical regularity.
From the proof of Lemma \ref{eq:regY}, in particular
the decomposition \eqref{eq:driftDecomp} with $J$ replaced by $B$ and 
$i$ replace by $n$,
we can show the same regularity for each term in the decomposition,
except for the term $B(Z^{(0,i)}) - B(Z^{(0,i-1)})$.
Under Condition \eqref{eq:condExtra}, instead we have
\[
	B(Z^{(0,i)}) - B(Z^{(0,i-1)}) \in C_T\Cc^{(-\frac{1}{2}-\alpha_0)\mm}.
\]
by Appendix \ref{sec:Gaussian} and the similar argument in Lemma \ref{lem:regR}.
In this case, we have 
\begin{equation} \label{eq:regRYdrift}
	B(\Y^{(0,n)}) - B(\Y^{(0,n-1)}) \in C_T \Cc^{(-\frac{1}{2}-\alpha_0)\mm}.
\end{equation}
Now we use basically the same argument in \cref{lem:rhoeta} and
\cref{lem:regR} to obtain the same constraint for canonical regularity.

\begin{proposition}[Constraint from the canonical regularity for $\RY^{(n)}$]
  \label{prop:regRX}
  Under the standing noise factorization assumptions
  \eqref{eq:alpha}-\eqref{eq:Qi},
  if the $\Y^{(i)}$ equations are well
  posed with all of their terms possessing the canonical regularity,
  and if in addition $\alpha_0 < 1$ and $\alpha_0 - \beta_n < 1$, 
  then $\RX^{(n)}$ has the canonical regularity, 
  \[  
    \RY^{(n)} \in C_T\Cc^{(\frac{1}{2}-\beta_n)\mm},
  \]
  namely that of the stochastic convolution in the equation,
  and
  \[
    B(\Y^{(0, n)}, \RY^{(n)}) \in C_T\Cc^{(\frac{1}{2}-\alpha_0)\mm}
  \]
  for any $T < \tau_\infty$ almost surely. Setting $\RY_t^{(n)}=\death$
  for $t \geq \tau_\infty$, we have that $ \RY^{(n)} \in C_T\overline
  \Cc^{(\frac{1}{2}-\beta_n)\mm}$ for any $T>0$.
\end{proposition}

\begin{remark}[Fourth Note on $\alpha < 1$]
  We note again the dependencies on $\alpha < 1$:
  \begin{enumerate}
    \item Under this condition, the $J(Z)$ terms have positive
    H\"older regularity, which greatly reduces the complexity
    of making sense of the $\Y$ equations in \cref{lem:regY}.

    \item It makes sure that $B(\Y^{(0,k)}) - B(\Y^{(0,k-1)})$
    has the same regularity as $B(Z^{(0,k)}) - B(Z^{(0,k-1)})$ 
    for $1 \le k \le n$
    so that \eqref{eq:regY} and \eqref{eq:regRYdrift} holds. 
    This also corresponds to 
    the regularization effect of J as discussed in \cref{rem:Jregs}.

    \item It makes sure that the $\RY^{(n)}$ equation is well posed, 
    and the Time-Shifted Girsanov Method applies 
    in exactly the same way like the $\RX^{(n)}$ equation.
  \end{enumerate}
\end{remark}

We observe that the terms in $\RY^{(n)}$ have the same regularity
as the corresponding terms in $\RX^{(n)}$, so we obtain 
the same result of \cref{prop:GirR} for $\RY^{(n)}$.

\begin{proposition}[Constraint from Time-Shifted Girsanov for $\RY^{(n)}$]
  \label{prop:absContRX}
  Under the standing noise factorization assumptions
  \eqref{eq:alpha}-\eqref{eq:Qi},
  if the $\Y^{(i)}$ equations are well
  posed with all of their terms possessing the canonical regularity,
  and if in addition $\alpha_0 < 1$ and $\alpha_0 - \beta_n < \frac{1}{2}$, 
  then the regularity condition needed to apply \cref{thm:tsG}, the Time-Shifted
  Girsanov Method, to $\RY^{(n)}$ holds. 
  More concretely, it implies that for any $t > 0$,
  it holds almost surely that
  $$ \Law(\RY_t^{(n)}\mid t< \tau_\infty,\mathcal{F}_t^{(n)}) 
  \ll \Law(\tZ_t^{(n)}), $$
  where we recall 
  $\mathcal{F}_t^{(n)} = \sigma( W_s^{(j)}: j \leq n, s \leq t)$.

  In particular, as long as $\alpha_0 < 1$, 
   $\beta_n$ (and $\alpha_n$) can be taken close enough to $\frac{1}{2}$ 
   to satisfy the condition $\alpha_0 - \beta_n < \frac{1}{2}$.
\end{proposition} 

Again, since the main argument in the previous section follows from
computations of the same regularity, 
we can obtain the same overall result as the previous section,
by \cref{cor:absXY}.

\begin{corollary}[The overall result on the $\Y$ system]
  Fix an $n$ and an $\alpha$ so that $\frac{1}{2} \le \alpha < \frac{2n+1}{2n+2}$. Then there exists 
  a sequence real numbers $\beta_n < \alpha_n < \ldots < \alpha_0 = \alpha$ 
  such that the standing noise factorization assumptions on the
        $\{\alpha_j: j=0,\dots,n\}$ in
 \eqref{eq:alpha}-\eqref{eq:Qi} hold as well as the hypothesis of
 \cref{lem:regY}, \cref{prop:absContY}, \cref{prop:regRX} and \cref{prop:absContRX} .
 More concretely, it implies that for any $t > 0$, 
 it holds almost surely that 
 \begin{align*}
  \Law(u_t \mid  \tau_\infty>t) &= \Law\left(\RY_t^{(n)} +\sum_{k=0}^n \Y_t^{(k)}\spc{\Big|} \tau_\infty>t\right) \\&\ll 
  \Law\left(\tZ_t^{(n)} +\sum_{k=0}^n Z_t^{(k)}\right) = \Law(z_t), 
 \end{align*}
 where, as a reminder, $z_t$ is the linear part of $u_t$ 
 with initial condition $z_0$ as defined in \eqref{eq:OU}.
\end{corollary}

\section{Discussion}

The Time-Shifted Girsanov Method, described in \cref{sec:TSG}, was used
in \cite{Mattingly_Suidan_2005} to show that the hyper-viscous,
two-dimensional Navier-Stokes equation satisfied the translation of
\cref{thm:SimplyMain} to that setting when the forcing is smooth
enough to have classical solutions but not so smooth that is
infinitely differentiable in space. This is completely analogous to the
Theorem proven here when $\alpha<\frac12$. We conjecture that a
version of \cref{thm:SimplyMain}  does not hold for the Gaussian
measure generated by the  Ornstein-Unlenbeck process obtained by
removing the nonlinearity from \eqref{eq:burgers}. It would be interesting to compare and
contrast that setting to the current one when $\alpha=1$. In both
settings, it would be interesting to  understand the structure of the
transition measure when $Q \approx e^{-A}$ where we expect the system
to have more in common with a finite dimensional hypoelliptic system.

\vspace{1em}
\noindent \textbf{Acknowledgments: } MR and JCM thank MSRI for its
hospitality during the 2015 program ``New Challenges in PDE:
Deterministic Dynamics and Randomness in High and Infinite Dimensional
Systems'' where they began working on the  multilevel decomposition to
prove equivalence. This work builds on an unpublished manuscript of
JCM and Andrea Watkins Hairston which looks at the case analogous to
$\alpha < \frac12$ in related PDEs using the Time-Shifted Girsanov
Method directly on the main equation without the levels of
decomposition needed for the singular case. JCM and MR also thank a
grant from the \emph{Visiting professors} programme of {G.\,N.\,A.\,M.\,P.\,A.}
which allowed JCM to visit Pisa during the summer
of 2016, where this work evolved closer to its current direction.
JCM and LS thanks the
National Science foundation for its partial support through the
grant NSF-DMS-1613337. LS also thanks SAMSI for its partial support,
through the grant  NSF-DMS-163852,
during the 2020-2021 academic year when all of the pieces finally came
together in the singular setting and this note was written.

\appendix
\section{Besov spaces and Paraproducts}
\label{Besov}
Results of this section can be found at 
\cite{bahouri2011fourier,gubinelli2015paracontrolled,catellier2018paracontrolled}.
We recall the definition of Littlewood-Paley blocks.
Let $\chi, \varphi$ be smooth radial functions $\R \to \R$ such that
\begin{itemize}
	\item $0 \le \chi, \varphi \le 1$, 
		$\chi(\xi) + \sum_{j \ge 0} \varphi(2^{-j}\xi) = 1$ for any $\xi \in \R$,

	\item $\supp \chi \subseteq B(0, R)$, 
		$\supp \varphi \subseteq B(0, 2R) \setminus B(0, R)$,

	\item $\supp \varphi(2^{-j} \cdot) \cap \supp \varphi(2^{-i} \cdot) = \emptyset$
		if $|i - j| > 1$.
\end{itemize}
The pair $(\chi, \varphi)$ is called a dyadic partition of unity.
We use the notations
\begin{equation}\label{eq:LPpartition}
\varphi_{-1} = \chi, \quad \varphi_j = \varphi(2^{-j} \cdot),
\end{equation}
for $j \ge 0$. 
Then the family of Fourier multipliers $(\Delta_j)_{j \ge -1}$ 
denotes the associated Littlewood-Paley blocks, i.e., 
$$\Delta_{-1} = \chi(D), \quad \Delta_j = \varphi(2^{-j} D)$$ for $j \ge 0$.

\begin{definition}
	For $s \in \R, p, q \in [1, \infty]$, the Besov space $B^s_{p,q}$ is
	defined as 
	\[
		B^s_{p,q} = \left\{u \in \Ss': \|u\|_{B^s_{p,q}} := 
		\left\| \left(2^{js} \|\Delta_i u \|_{L^p}\right)_{j \ge -1} 
		\right\|_{\ell^q} < \infty \right\}.
	\] 
	As a convention, we denote by $\Cc^{s}$ the separable version
	of the Besov-H\"older space $B^s_{\infty, \infty}$ , i.\,e., 
	$\Cc^{s}$ is the closure of $C^\infty(\T)$ 
	with respect to $\|\cdot\|_{B^s_{\infty, \infty}}.$
	We also write $\|\cdot\|_{\Cc^s}$ to mean $\|\cdot\|_{B^s_{\infty, \infty}}$.
\end{definition}
\begin{remark}
	To show that $u \in \Ss'$ is in $\Cc^s$, 
	it suffices to show $\|u\|_{B^{s'}_{\infty, \infty}} < \infty$
	for some $s' > s$.
\end{remark}
\begin{remark}\label{rm:clasicalHolder}
  For $0 < s < 1$, $f$ is in the classical space of 
  $s$-H\"older continuous functions if
  and only if $f \in L^\infty$ and $\|f\|_{\Cc^s} < \infty$.
\end{remark}

\begin{proposition}[Besov embedding]\label{prop:emb}
	Let $1 \le p_1 \le p_2 \le \infty$ and $1 \le q_1 \le q_2 \le \infty$.
	For $s \in \R$, the space $B^s_{p_1, q_1}$ is continuously
	embedded in $B^{s - (\frac{1}{p_1} - \frac{1}{p_2})}_{p_2, q_2}$.
\end{proposition}

\begin{definition}
	A smooth function $\eta: \R \to \R$ is said to be an $S^m$-multiplier
	if for every multi-index $\alpha$,
	\[
		\left|\frac{\partial^\alpha \eta}{\partial x^\alpha}(\xi)\right| 
			\lesssim_\alpha (1 + |\xi|)^{m - |\alpha|}, \quad \forall \xi \in \R.
	\]
\end{definition}

\begin{proposition}\label{prop:multiplier}
	Let $m \in \R$ and $\eta$ be a $S^m$-multiplier. Then,
	for all $s \in \R$ and $1 \le p, q \le \infty$,
	the operator $\eta(D)$ is continuous from $B^s_{p,q}$ to $B^{s-m}_{p,q}$.
\end{proposition}

The following estimate can be found at 
\cite[Lemma 2.5]{catellier2018paracontrolled} and 
\cite[Lemma A.7]{gubinelli2015paracontrolled}.
\begin{proposition}\label{prop:heat}
	Let $A$ be the negative Laplacian, and $\gamma, \delta \in \R$ 
	with $\gamma \le \delta$. Then 
	$$\|e^{-tA} u\|_{\Cc^\delta} \lesssim 
		t^{\frac{1}{2}(\gamma - \delta)}\|u\|_{\Cc^{\gamma}}.$$
	for all $u \in \Cc^\gamma$.
\end{proposition}

For $u \in \Cc^{\gamma}$ and $v \in \Cc^{\delta}$, 
we can formally decompose the product $uv$ as
\begin{equation} \label{eq:prod}
	uv = u \lpara v + u \circ v + u \rpara v,
\end{equation}
where
\[
	u \lpara v = v \rpara u := \sum_{i< j - 1} \Delta_i u \Delta_j v,
	\quad u \circ v = \sum_{|i - j| < 1} \Delta_i u \Delta_j v.
\]

\begin{proposition}[Bony estimates]\label{prop:para}
	Let $\gamma, \delta \in \R$. Then we have the estimates
	\begin{itemize}
		\item $\|u \lpara v\|_{\Cc^\delta} 
			\lesssim \|u\|_{L^\infty} \|v\|_{\Cc^\delta}$
			for $u \in L^\infty$ and $v \in \Cc^\delta$.

		\item $\|u \lpara v\|_{\Cc^{\gamma + \delta}} 
			\lesssim \|u\|_{\Cc^{\gamma}} \|v\|_{\Cc^\delta}$
			for $\gamma < 0$, $u \in \Cc^\gamma$ and $v \in \Cc^\delta$.

		\item $\|u \res v\|_{\Cc^{\gamma + \delta}} 
			\lesssim \|u\|_{\Cc^{\gamma}} \|v\|_{\Cc^\delta}$
			for $\gamma + \delta > 0$, $u \in \Cc^\gamma$ and $v \in \Cc^\delta$.
	\end{itemize}
\end{proposition}
\begin{remark} \label{rmk:prod}
Note that the paraproduct $\lpara$ is always 
a well-defined continuous bilinear operator. 
The product 
$\Cc^\gamma \times \Cc^\delta 
	\to \Cc^{\gamma \wedge \delta \wedge (\gamma + \delta)}, 
	(u, v) \mapsto uv$ 
is a well-defined, continuous bilinear map provided $\gamma + \delta > 0$.
In this case, we say the product is classically well-defined.
On the other hand, if we can directly show the existence of
the resonant product $u \circ v$ of two terms $u$ and $v$, 
then the product $uv$ will be well-defined.
In this case, we usually have $u \circ v \in \Cc^{\gamma + \delta}$
given that $u \in \Cc^\gamma$ and $v \in \Cc^\delta$,
so we still obtain $uv \in \Cc^{\gamma \wedge \delta \wedge (\gamma + \delta)}$
without the condition $\gamma + \delta > 0$.
\end{remark}

\section{Construction of finite Gaussian chaos objects}\label{sec:Gaussian}
In this section, we construct various singular processes  
with their canonical regularity that are necessary for our analysis above. 
We refer to 
\cite{hairer2013solving,gubinelli2017kpz,mourrat2015construction,
catellier2018paracontrolled,gubinelli2015paracontrolled}
for relevant technical details.
In particular, \cite{gubinelli2017kpz} provides constructions of 
many of these Gaussian objects that correspond to the harder case 
of $\alpha = 1$ in our setting. 
We mainly apply the unified argument in \cite{mourrat2015construction} 
and point out the minimal modifications for our setting.
More specially, we only need to change relevant
Fourier multipliers of the nonlinearity and the driving noise 
and compute estimates with the same procedure.

We first set up some notations:
\begin{itemize}
	\item We write $\Z_0 = \Z \setminus \{0\}.$ 
	\item For $x \in \Ss'(\T)$, $k \in \Z$, 
		$\widehat{x}(k)$ denotes the $k$-th Fourier mode of $x$. 
	\item For $k \in \Z$, let $e_k$ be the $k$-th Fourier basis function 
		given by complex exponentials.
	\item For a process $x$, we use the notation $x_{s, t} := x_t - x_s$.
	\item We write $\eqOrder$ to mean both $\lesssim$ and $\gtrsim$.
	\item We write $k \sim k'$ if $k \in \supp\,\varphi_i$, $k' \in \supp\,\varphi_j$
		and $|i - j| \le 1$, and by abuse of this notations,
		we write $k \sim 2^j$ if $k \in \supp\,\varphi_j$,
    where $\varphi_j$ is defined in \eqref{eq:LPpartition}.
	\item Let $\psi$ be a smooth radial function with compact support
	and $\psi(0) = 1$. 
	We regularize a process $x$ by setting
	$$x^\epsilon_t = \sum_{k \in \Z} \psi(\epsilon k)\hat{x}_t(k) e_k,$$
	and for convenience, we also write
	$$ x^\epsilon_t = \sum_{|k| \lesssim \epsilon^{-1}} \hat{x}_t(k) e_k.$$
\end{itemize}

Fix $\gamma, \delta \in (0, 1)$.
Let $(z^{(\gamma)}_t\,:\, t \in [0, T])$ 
denote the Ornstein-Uhlenbeck process defined by
\[
	z^{(\gamma)}_t := \int_0^t e^{-(t-s)A}Q^{(\gamma)}dW_s,
\] 
for some operator $Q^{(\gamma)} \simQ A^{\gamma/2}$,
and $W$ is a cylindrical Brownian motion. 
We define $z^{(\delta)}$ in the same way, 
but we use different symbols to note that
$z^{(\gamma)}$ and $z^{(\delta)}$ are driven by different, independent 
cylindrical Brownian motions so that they 
are independent.
It is more convenient to write $z^{(\gamma)}$
and $z^{(\delta)}$ in Fourier space:
we have a family of independent, standard
complex-valued Brownian motions $(W(k)\,:\, k \in \Z)$ 
with the real-valued constraint $\overline{W(k)} = W(-k)$
such that for $k \in \Z_0$,
\[
	{\widehat{z^{(\gamma)}_t}}(k) = \int_0^t e^{-|k|^2(t-s)} q_k dW_s(k)
\]
where $q_k$ is the eigenvalue of $Q$ corresponding to $e_k$
and $|q_k| \eqOrder |k|^{\gamma}$. In particular, $q_0 = 0$,
so $\widehat{z^{(\gamma)}_t}(0) = 0$, which means
$z^{(\gamma)}$ has mean zero in space.

\begin{remark}
  Since we work with processes like $z^{(\gamma)}$ that have mean zero in space,
  we ignore the $0$-th Fourier mode by default, for example, 
  in various summations involving Fourier modes. 
\end{remark}

\subsection{Preliminary results}
We will use Proposition 3.6, Lemma 4.1 and Lemma 4.2 from
\cite{mourrat2015construction}.
\begin{proposition}[\cite{mourrat2015construction}] \label{prop:regW}
	Let $x: [0, T] \to \Ss'(\T)$ be a stochastic process 
	in some finite Wiener chaos such that
	\[
		k + k' \neq 0 \quad \implies \quad 
		\EE[\widehat{x_s}(k) \, \widehat{x_t}(k')] = 0.
	\]
	If for some $t \in [0, T]$, $\EE[|\widehat{x_t}(0)|^2] \lesssim 1$ and for all $k \in \Z_0$,
	\begin{equation} \label{eq:fest}
		\EE[|\widehat{x_t}(k)|^2] \lesssim \frac{1}{|k|^{1 + 2\kappa}},
	\end{equation}
	then for every $\beta < \kappa$, $p \ge 2$, we have 
	\[
		\EE[\|x_t\|_{\Cc^\beta}^p] < \infty.
	\]
	If, in addition to \eqref{eq:fest}, there exists $h \in (0, 1)$
	such that $\EE[|\widehat{x_{s,t}}(0)|^2] 
    \lesssim {|t - s|^h}$ and
	\begin{equation} \label{eq:fest2}
		\EE[|\widehat{x_{s,t}}(k)|^2] 
		\lesssim \frac{|t - s|^h}{|k|^{1+2\kappa-2h}},
	\end{equation}
	uniformly in $0 < |t - s| < 1$ and $k \in \Z_0$, then 
	$\tau \in C_T \Cc^{\beta}$, and
	\[
		\sup_{0< |t - s| < 1}\frac{\EE\|x_{s,t}\|_{\Cc^\beta}^p}{
		|t - s|^\frac{h p}{2}} < \infty.
	\]
\end{proposition}

\begin{remark}
  Later when we apply \cref{prop:regW}, 
  checking the condition on the $0$-th Fourier mode is straightforward, 
  so we will omit details of that part.
\end{remark}

\begin{lemma}[\cite{mourrat2015construction}] \label{lem:fsum1}
	Let $a, b \in \R$ satisfy
	$a + b > 1$ and $a, b < 1.$
	We have uniformly for all $k \in \Z_0$,
	\[
		\sum_{\substack{k_1, k_2 \in \Z_0 \\ k_1 + k_2 = k}} 
		\frac{1}{|k_1|^a} \frac{1}{|k_2|^b}
		\lesssim \frac{1}{|k|^{a + b - 1}}.
	\]
\end{lemma}

\begin{lemma}[\cite{mourrat2015construction}] \label{lem:fsum2}
	Let $a, b \in \R$ satisfy
	$a + b > 1.$
	We have uniformly for all $k \in \Z_0$,
	\[
		\sum_{\substack{k_1 + k_2 = k, \\ k_1 \sim k_2}} 
		\frac{1}{|k_1|^a} \frac{1}{|k_2|^b}
		\lesssim \frac{1}{|k|^{a + b - 1}}.
	\]
\end{lemma}

The following proposition provides a way to 
bound $p$-th moments of H\"older norms of
a process via estimates of its
Littlewood-Paley blocks.
\begin{proposition} \label{prop:bdd}
	Let $x : [0, T] \to \Ss'(\T)$ be a stochastic process in
	some finite Wiener chaos such that for $h \ge 0$
	small enough
	\[
		\EE\left[ \left|\Delta_j x_{s,t}(x) \right|^2 \right]
		\le C |t - s|^h 2^{-j(2\kappa - 2h)}.
	\]
	Then for any $\beta < \kappa$ and $p > 1$,
	\[
		\EE\left[ \left\|x\right\|^p_{C_T \Cc^\beta} \right] 
		\lesssim C^{p/2}.
	\]
\end{proposition}
\begin{proof}
	By Gaussian hypercontractivity 
	(e.g. \cite[Proposition 3.3]{mourrat2015construction}),
	for $p > 1$, $h \ge 0$ small enough
	\[
		\EE\left[ \left\|\Delta_j x_{s,t} \right\|^{2p}_{L^{2p}} \right]
		\lesssim \left\|\EE\left[ \left|\Delta_j x_{s,t}(x) \right|^2 \right] 
		\right\|_{L^p(dx)}^p \lesssim C |t - s|^{hp} 2^{-jp(2\kappa - 2h)}.
	\]
	For any $\beta < \kappa$, by taking $h$ small enough, $p$ large enough,  
	by \cite[Proposition 2.7]{mourrat2015construction}
	or Besov embedding, Proposition \ref{prop:emb}, we have
	\[
		\EE\left[ \left\|x_{s,t}\right\|^p_{\Cc^\beta} \right]
		\lesssim C^{p/2} |t - s|^{\frac{hp}{2}}. 
	\]
	By a variant of Kolmogorov continuity theorem or 
	the standard Garsia-Rodemich-Rumsey lemma (\cite{garsia1970real}),
	we obtain 
	\[
		\EE\left[ \left\|x\right\|^p_{C_T \Cc^\beta} \right] \lesssim
		 C^{p/2}.
	\]
\end{proof}

\subsection{Regularity and convergence of \texorpdfstring{$z$}{z}}
The following result as in \cite[Lemma 4.4]{gubinelli2015paracontrolled}
follows from a straightforward computation, which 
will be useful later.
\begin{lemma}
	The spatial Fourier transform $\widehat{z^{(\gamma)}}$ of $z^{(\gamma)}$ 
	is a complex-valued, centered Gaussian process with covariance
	\[
		\EE\left[\widehat{z^{(\gamma)}_t}(k)\, 
		\overline{\widehat{z^{(\gamma)}_{t'}}(k')}\right]
		\eqOrder \begin{cases}
			 \one_{k = k'} |k|^{-2+2\gamma}(e^{-|k|^2|t' - t|} - e^{-|k|^2(t' + t)}),
				\quad & k \neq 0 \\
						 0 & k = 0.
		\end{cases}
	\]
	where $k, k' \in \Z$, $t, t' \in [0, T]$. 
	In particular, we have
	\begin{equation} \label{eq:zest1}
		\EE\left[\left|\widehat{z^{(\gamma)}_{t}}(k)\right|^2 \right]
		\lesssim \frac{1}{|k|^{2 - 2\gamma}},
	\end{equation}
	\begin{equation} \label{eq:zest2}
		\EE\left[\widehat{z^{(\gamma)}_{t}}(k) \,  
			\overline{\widehat{z^{(\gamma)}_{t'}}(k)}\right]
		\lesssim \frac{1}{|t - t'|^\rho} \frac{1}{|k|^{2 - 2\gamma + 2\rho}},
	\end{equation}
	\begin{equation} \label{eq:zest3}
		\EE\left[\left|\widehat{z^{(\gamma)}_{s,t}}(k)\right|^2 \right]
		\lesssim \frac{|t - s|^h}{|k|^{2 - 2\gamma - 2h}}.
	\end{equation}
	for all $s,t,t' \in [0,T]$, $\rho, h \in [0, 1]$ and $k \in \Z_0$.
\end{lemma}
Using Proposition \ref{prop:regW} with the previous lemma, 
we have the following result.
\begin{proposition} \label{prop:regz}
	For any $\beta < \frac{1}{2} - \gamma$, we have
	$z^{(\gamma)} \in C_T \Cc^\beta.$
\end{proposition}

One may also adapt the argument by changing a few parameters
in \cite[Section 4.1]{catellier2018paracontrolled}
to prove the following approximation result.
\begin{proposition}
	For any $\beta < \frac{1}{2} - \gamma$ and any $p > 1$, we have
	\[
		\lim_{\epsilon \to 0} \EE\left[\left\|z^{(\gamma), \epsilon} 
			- z^{(\gamma)}\right\|^{p}_{C_T \Cc^{\beta}}\right] = 0.
	\]
\end{proposition}

\subsection{Construction of \texorpdfstring{$B(z^{(\gamma)})$}{B(z(gamma))}
  and \texorpdfstring{$J(z^{(\gamma)})$}{J(z(gamma))}}
If $\gamma < \frac{1}{2}$, $z^{(\gamma)}$ is a function-valued process,
so $(z^{(\gamma)})^2$ and $J(z^{(\gamma)})$ are well-defined classically.
Then $B(z^{(\gamma)}) \in C_T \Cc^{\beta}$ for
	any $\beta < -\frac{1}{2} - \gamma$.

If $\frac{1}{2} \le \gamma$, $z^{(\gamma)}$ is distribution-valued, 
and $(z^{(\gamma)})^2$ is not classically well-defined, so 
we introduce a renormalization procedure.
For the regularized process $z^{(\gamma), \epsilon}$,
define the renormalized product 
$$(z^{(\gamma), \epsilon})^{\diamond 2}_t := 
(z^{(\gamma), \epsilon})^{2}_t - \EE[(z^{(\gamma), \epsilon})^{2}_t].$$

\begin{lemma} For any $s, t, t' \in [0, T]$, we have
\begin{align} \label{eq:z21}
	\EE\left[\left|\widehat{(z^{\epsilon})^{\diamond 2}_t}(k)\right|^2 \right] 
	& \lesssim \sum_{\substack {k_1 + k_2 = k, \\ |k_1|, |k_2| \lesssim \epsilon^{-1}} } 
			\frac{1}{|k_1|^{2 - 2\gamma}|k_2|^{2 - 2\gamma}}, \\
	\label{eq:z22}
	\EE\left[\widehat{(z^{\epsilon})^{\diamond 2}_t}(k)\,
			\overline{\widehat{(z^{\epsilon})^{\diamond 2}_{t'}}(k)}
	\right] 
	& \lesssim \frac{1}{|t - t'|^\rho} 
	\sum_{\substack {k_1 + k_2 = k, \\ |k_1|, |k_2| \lesssim \epsilon^{-1}} } 
		\frac{1}{|k_1|^{2 - 2\gamma + \rho}
		|k_2|^{2 - 2\gamma + \rho}}, \\
	\label{eq:z23}
	\EE\left[\left|\widehat{(z^{\epsilon})^{\diamond 2}_{s,t}}(k)\right|^2 \right] 
	& \lesssim {|t - s|^h} 
	\sum_{\substack {k_1 + k_2 = k, \\ |k_1|, |k_2| \lesssim \epsilon^{-1}} } 
		\frac{1}{|k_1|^{2 - 2\gamma - h}
		|k_2|^{2 - 2\gamma -h}},
\end{align}
for any $k \in \Z_0$, and $\rho, h \in [0,1]$.
\end{lemma}
\begin{proof}
	For convenience, we write $z = z^{(\gamma)}$. 
	Because of the renormalization, $(z^{\epsilon})^{\diamond 2}_t$ belongs
	to the second homogeneous Wiener chaos. 
	For $k \in \Z_0$, $t > 0$, by Itô formula,
	\[
		\widehat{(z^{\epsilon})^{\diamond 2}_t}(k) 
		= 2 \sum_{\substack {k_1 + k_2 = k, \\ |k_1|, |k_2| \lesssim \epsilon^{-1}} }
			\int_0^t \int_0^s e^{-|k_1|^2(t-s) - |k_2|^2 (t-r)} q_{k_1} q_{k_2}
			dW_r(k_2)\,dW_s(k_1),
	\]
	so we have
	\[
		\EE\left[\left|\widehat{(z^{\epsilon})^{\diamond 2}_t}(k)\right|^2
		\right] \eqOrder 
			\sum_{\substack {k_1 + k_2 = k, \\ |k_1|, |k_2| \lesssim \epsilon^{-1}} } 
			\frac{1}{|k_1|^{-2\gamma}|k_2|^{-2\gamma}} 
			\int_0^t \int_0^s e^{-2|k_1|^2(t-s)-2|k_2|^2(t-r)}\,dr \,ds, 
	\]
	and \eqref{eq:z21}, \eqref{eq:z22} hold. On the other hand,
	\begin{align*}
		\EE\left[\left|\widehat{(z^{\epsilon})^{\diamond 2}_{s,t}}(k)\right|^2
		\right] & \lesssim 
			\sum_{\substack {k_1 + k_2 = k, \\ |k_1|, |k_2| \lesssim \epsilon^{-1}}} 
			\frac{1}{|k_1|^{-2\gamma}|k_2|^{-2\gamma}} 
			\left(\int_s^t \int_0^{s_1} e^{-2|k_1|^2(t-s_1)-2|k_2|^2(t-s_2)}\,ds_2 \,ds_1
			\right. \\
			& \quad +
			\left.\int_0^s \int_0^{s_1} 
			\left(e^{-|k_1|^2(t-s_1)-|k_2|^2(t-s_2)} - 
			e^{-|k_1|^2(s-s_1)-|k_2|^2(s-s_2)}\right)^2\,ds_2 \,ds_1 
			\right) \\
			& = \sum_{\substack {k_1 + k_2 = k, \\ |k_1|, |k_2| \lesssim \epsilon^{-1}}} 
			\frac{(\text{I} + \text{II})}{|k_1|^{-2\gamma}|k_2|^{-2\gamma}},
	\end{align*}
	where using both the bound $e^{-r} \le 1$ and  
	$1 - e^{-rt} \lesssim rt$ for $r \ge 0$ and 
	computing integrals of the form $\int e^{-r} dr$, 
	for any $h \in [0,1]$,
	\begin{align*}
		\text{I} & \lesssim \frac{1}{|k_2|^2} \int_s^t e^{-2|k_1|^2(t - s_1)}\,ds_1
			\lesssim \frac{1}{|k_2|^2 }\left(\frac{1}{|k_1|^2} \wedge |t - s|\right)
			\le \frac{|t - s|^h}{|k_1|^{2-2h} |k_2|^2}, \\
		\text{II} & = \int_0^s \int_0^{s_1} 
			\left(e^{-(|k_1|^2 + |k_2|^2)(t-s)} - 1\right)^2
			e^{-2|k_1|^2(s-s_1)-2|k_2|^2(s-s_2)}\,ds_2 \,ds_1 \\
		& \lesssim \frac{1}{|k_2|^2} 
		\left(\frac{1}{|k_1|^2} \wedge |t - s|\right) 
		\le \frac{|t - s|^h}{|k_1|^{2-2h} |k_2|^2},
	\end{align*}
	so \eqref{eq:z23} holds by exchanging the role of $k_1$ and $k_2$.
\end{proof}

For $\gamma < \frac{3}{4}$, we can show that 
the renormalized product $(z^{(\gamma), \epsilon})^{\diamond 2}$ 
converges to a limiting process
$(z^{(\gamma)})^{\diamond 2}$ with the desired regularity.
\begin{proposition} \label{prop:z2}
	If $\frac{1}{2} \le \gamma < \frac{3}{4}$, 
	then there exists a process $(z^{(\gamma)})^{\diamond 2}$
	such that 
	\begin{equation} \label{eq:z2est1}
		\EE\left[\left| \widehat{(z^{(\gamma)})^{\diamond 2}_t}(k) \right|^2 \right]
		\lesssim \frac{1}{|k|^{3 - 4\gamma - \eta}},
	\end{equation}
        for any $k \in \Z_0$, $0 \le t \le T$ and small enough $\eta > 0$,
	and
	\[
		\lim_{\epsilon \to 0} \EE\left[\left\|(z^{(\gamma), \epsilon})^{\diamond 2} 
			- (z^{(\gamma)})^{\diamond 2}\right\|^{p}_{C_T \Cc^{\beta}}\right] = 0,
	\]
	for any $\beta < 1 - 2\gamma$ and $p > 1$.
\end{proposition}
\begin{proof} When $\gamma > \frac{1}{2}$, by \eqref{eq:z21} and Lemma \ref{lem:fsum1}, 
	it holds uniformly in $\epsilon$ that
	\begin{equation}\label{eq:z2sum}
		\EE\left[\left|\widehat{(z^{\epsilon})^{\diamond 2}_t}(k)\right|^2 \right] 
		\lesssim \sum_{ {k_1 + k_2 = k} } 
			\frac{1}{|k_1|^{2 - 2\gamma}}\frac{1}{|k_2|^{2 - 2\gamma}} 
		\lesssim \frac{1}{|k|^{3 - 4\gamma}}.
	\end{equation}
  When $\gamma = \frac{1}{2}$, in order to apply Lemma \ref{lem:fsum1},
  we give up arbitrary small amount of decay in $k_1$ and $k_2$ of \eqref{eq:z2sum}
  to obtain \eqref{eq:z2est1}.
	
  Similarly, using \eqref{eq:z23} and Lemma \ref{lem:fsum1}, 
	for $h \ge 0$ small enough,
	\[
		\EE\left[\left|\Delta_j {(z^{\epsilon})^{\diamond 2}_{s,t}}(x)\right|^2 \right] 
		\lesssim \sum_{ k \sim 2^j } 
			 \EE\left[\left|\widehat{(z^{\epsilon})^{\diamond 2}_{s,t}}(k)\right|^2 \right] 
		 \lesssim \sum_{ k \sim 2^j } \frac{|t - s|^h}{|k|^{3 - 4\gamma -2h}} 
		 \lesssim |t - s|^h  2^{-j(2 - 4\gamma -2h)},
	\]
	which implies, by Proposition \ref{prop:bdd},
	\[
		\sup_\epsilon \EE\left[\left\|(z^{(\gamma), \epsilon})^{\diamond 2} 
			\right\|^{p}_{C_T \Cc^{\beta}}\right] < \infty.
	\]
	for any $\beta < 1 - 2\gamma$ and $p > 1$.
	We can obtain the same estimate for $(z^{\epsilon})^{\diamond 2} 
		- (z^{\epsilon'})^{\diamond 2}$, 
	since the terms involving $\psi(\epsilon k)$, $\psi(\epsilon' k')$
	are uniformly bounded, and other terms are the same.
	As $\epsilon, \epsilon' \to 0$, in the $k$-th Fourier mode, we have
	a factor like $$\psi(\epsilon k_1) \psi(\epsilon k_2) 
	- \psi(\epsilon' k_1)\psi(\epsilon' k_2) \longrightarrow 0.$$
	By dominated convergence, $((z^{(\gamma), \epsilon})^{\diamond 2})_\epsilon$
	is a Cauchy sequence in $L^p(\Omega, C_T \Cc^{\beta})$,
  where $\Omega$ denotes the underlying probability space.
	We denote the limit by $(z^{(\gamma)})^{\diamond 2}$, 
	and the result follows.
\end{proof}

However, if $\gamma \ge \frac{3}{4}$, $(z^{(\gamma)})^{\diamond 2}$
is no longer a well-defined process but a space-time distribution
so that $(z^{(\gamma)})^{\diamond 2}_t$ has no meaning at a fixed time $t > 0$.
This can be recognized from the fact that the summation in
\eqref{eq:z2sum} diverges.
In this case, we instead consider the regularized process 
\[
	J(z^{(\gamma), \epsilon})_t^{\diamond }
	:= \int_0^t e^{-(t-s)A} \partial_x (z^{(\gamma), \epsilon})^{\diamond 2} \,ds,
\]
for which we can show that, 
with the help of temporal regularity provided by the heat kernel, 
$J(z^{(\gamma), \epsilon})^{\diamond }$ converges to a well-defined process
$J(z^{(\gamma)})^{\diamond}$ with the desired regularity.

\begin{proposition} \label{prop:jz}
	If $\frac{1}{2} \le \gamma < 1$, there exists a process $J(z)^{\diamond}$ 
	such that 
	\begin{equation} \label{eq:jzest1}
		\EE\left[\left| \widehat{J(z^{(\gamma)})^{\diamond}_t}(k) \right|^2 \right]
		\lesssim \frac{1}{|k|^{5 - 4\gamma}},
	\end{equation}
	\begin{equation} \label{eq:jzest2}
		\EE\left[\left| \widehat{J(z^{(\gamma)})^{\diamond}_{s,t}}(k) \right|^2 \right]
		\lesssim \frac{|t - s|^h}{|k|^{5 - 4\gamma - 2h}},
	\end{equation}
	for any $k \in \Z_0$, $h \in [0,1]$ and $s, t \in [0,T]$,
	and 
	\[
		\lim_{\epsilon \to 0} \EE\left[\left\|
			J(z^{(\gamma), \epsilon})^{\diamond} 
			- J(z^{(\gamma)})^{\diamond}\right\|^{p}_{C_T \Cc^{\beta}}\right] = 0,
	\] 
	for any $\beta < 2 - 2\gamma$ and $p > 1$.
\end{proposition}
\begin{proof}
	We focus on the case $\frac{3}{4} \le \gamma < 1$, 
	since the other case is already done.
	Take $\rho \in (0,1)$ satisfy $2\gamma - \frac{3}{2} < \rho < \gamma - \frac{1}{2}$,
	and by \eqref{eq:z22} and Lemma \ref{lem:fsum1},
	\begin{align*}
		\EE\left[\widehat{(z^{\epsilon})^{\diamond 2}_s}(k)
				\overline{\widehat{(z^{\epsilon})^{\diamond 2}_{s'}}(k)}
		\right] 
		& \lesssim \frac{1}{|s - s'|^\rho} \sum_{k_1 + k_2 = k} 
			\frac{1}{|k_1|^{2 - 2\gamma + \rho}} \frac{1}{|k_2|^{2 - 2\gamma + \rho}} 
		\lesssim \frac{1}{|s - s'|^\rho}\frac{1}{|k|^{3  - 4\gamma + 2\rho}},
	\end{align*}
	which implies
	\begin{align*}
		\EE\left[\left| \widehat{J(z^{\epsilon})^{\diamond}_t}(k) \right|^2 \right] 
		& = \int_0^t \int_0^t |k|^2 e^{-|k|^2(t-s)-|k|^2(t - s')}
			\EE\left[\widehat{(z^{\epsilon})^{\diamond 2}_s}(k)
				\overline{\widehat{(z^{\epsilon})^{\diamond 2}_{s'}}(k)}
		\right] \,ds' \, ds \\
		& \lesssim \int_0^t \int_0^t |k|^2 e^{-|k|^2(t-s)-|k|^2(t - s')}
			 \frac{1}{|s - s'|^\rho} \frac{1}{|k|^{3  - 4\gamma + 2\rho}}
			\,ds' \, ds \\ 
		& \lesssim \frac{1}{|k|^{5 - 4\gamma}}.
	\end{align*}
	On the other hand, by taking $\rho \in (0,1)$ as above, using the bound
	\[
		\frac{e^{-|k|^2(t-s_1)-|k|^2(t - s_2)}}{|s_1 - s_2|^\rho |k|^{2\rho}} 
		\lesssim 1,
	\]
	with similar computations for \eqref{eq:z23}, we have
	\begin{align*}
		\EE\left[\left| \widehat{J(z^{(\gamma)})^{\diamond}_{s,t}}(k) \right|^2 \right]
		& \lesssim \frac{|k|^2}{|k|^{3  - 4\gamma}}
			\left(\int_s^t \int_s^{t} 
			\frac{e^{-|k|^2(t-s_1)-|k|^2(t - s_2)}}{|s_1 - s_2|^\rho |k|^{2\rho}}
		 \,ds_2 \, ds_1 \right. \\
		& \quad + \left.\int_0^s \int_0^s (e^{-|k|^2(t-s)} - 1)^2 \frac{
		 e^{-|k|^2(s-s_1)-|k|^2(s - s_2)}}{|s_1 - s_2|^\rho |k|^{2\rho}}
			\,ds_2 \, ds_1 \right) \\
		& \lesssim \frac{1}{|k|^{1 - 4\gamma}} 
		\left(\frac{1}{|k|^4} \wedge \frac{|t - s|}{|k|^2}\right) 
		\le \frac{|t - s|^{h}}{|k|^{5 - 4\gamma -2h}},
	\end{align*}
	for $h \in [0,1]$.
	With the above estimate, the last statement follows
	from the same argument as in Proposition \ref{prop:z2}.
\end{proof}

\begin{remark}[Renormalization is not needed for the nonlinearity of Burgers]
Since $\partial_x: \Cc^{\beta} \to \Cc^{\beta - 1}$ is continuous
and annihilates quantities that are constant in space,
we note that
\[
	\partial_x (z^{(\gamma)})^{\diamond 2} 
	= \lim_{\epsilon \to 0}  \partial_x (z^{(\gamma), \epsilon})^{\diamond 2}
	= \lim_{\epsilon \to 0}  \partial_x (z^{(\gamma), \epsilon})^{2}
	= B(z^{(\gamma)}).
\]
Also, by the Fourier expansion of $B(z^{(\gamma), \epsilon})$,
we see that the $0$-th mode is zero, 
so the renormalization procedure is not actually needed for $B(z^{(\gamma)})$.
By the same reasoning, 
we can interpret $J(z^{(\gamma)}) = J(z^{(\gamma)})^{\diamond}$.
\end{remark} 

\subsection{Construction of \texorpdfstring{$B(J(z^{(\gamma)}), z^{(\gamma)})$}{B(J(z(gamma)),z(gamma))}}
If $\gamma < \frac{1}{2}$, then $B(J(z^{(\gamma)}), z^{(\gamma)})$ is
classically well-defined. Then
$B(J(z^{(\gamma)}), z^{(\gamma)}) \in C_T \Cc^{\beta}$ for any
$\beta < ((\frac{3}{2} - \gamma) \wedge (\frac{1}{2} - \gamma) 
\wedge (2 - 2\gamma)) - 1 = -\frac{1}{2} - \gamma$, 
according to Remark \ref{rmk:prod}.

If $\frac{1}{2} \le \gamma < 1$,
we show the existence
of the resonant product $J(z^{(\gamma)}) \circ z^{(\gamma)}.$
\begin{proposition}
	Suppose $\frac{1}{2} \le \gamma < 1$. 
	Let $\vartheta = J(z^{(\gamma)})$, $z = z^{(\gamma)}$. 
	We have for any $k \in \Z_0$,
	\begin{equation*} 
  		\EE\left[\left| \widehat{\vartheta_t \circ z_t}(k) \right|^2 \right]
		\lesssim \frac{1}{|k|^{6 - 6\gamma}},
	\end{equation*}
	and $\vartheta \circ z \in C_T \Cc^{\beta}$
	for any $\beta < \frac{5}{2} - 3\gamma.$
\end{proposition}
\begin{proof}
Our approach is the same as \cite[Page 29 - 31]{mourrat2015construction}, 
and to see the argument more clearly, 
the reader is encouraged to write down the corresponding diagrams
of our case.

For convenience, write $P_{t}(k) = e^{-|k|^2t} \one_{t \ge 0}$.
By the Wiener chaos decomposition 
(see \cite{mourrat2015construction} for a simple strategy using diagrams),
\[
	\widehat{\vartheta_t \circ z_t}(k) = I^{(3)}_t(k) + 2 \times I^{(1)}_t(k),
\]
where $I^{(3)}$ belongs to the third Wiener chaos, 
$I^{(1)}$ belongs to the first Wiener chaos, and
they are given by
\begin{align*}
	I^{(3)}_t(k) =  6 \sum_{\substack{k_1 + k_2 = k_4, \\ k_3 + k_4 = k, \\ k_3 \sim k_4}}
	& \int_0^t \int_0^{s_3} \int_0^{s_2} \int_0^t
	P_{s_4 - s_1}(k_1) P_{s_4 - s_2}(k_2) P_{t - s_3}(k_3) P_{t - s_4}(k_4) \\
	& \times ik_4\, q_{k_1} q_{k_2} q_{k_3} \, ds_4 
	\, dW_{s_1}(k_1) \, dW_{s_2}(k_2) \, dW_{s_3}(k_3), 
\end{align*}
\begin{align*}
	I^{(1)}_t(k) = \sum_{\substack{k_1 + k_2 = k_4, \\ k_3 + k_4 = k, \\ k_3 \sim k_4,
		\\ k_1 + k_3 = 0}}
	&  \int_0^t  \int_0^{t} \int_0^{s_4}
	P_{s_4 - s_1}(k_1) P_{s_4 - s_2}(k_2) P_{t - s_1}(k_3) P_{t - s_4}(k_4) \\
	& \times ik_4 \, q_{k_1} q_{k_2} q_{k_3} \, ds_1 \, ds_4  \, dW_{s_2}(k_2).
\end{align*}
Consider $\EE[|{{I^{(3)}_t}}(k)|^2]$ 
and expand out the integrals. We see that the inner integral
has almost the same expression as $\EE[|\widehat{\vartheta_t}(k')|^2]$
which can be bounded by $\frac{1}{|k'|^{5 - 4\gamma}}$.
With this bound, the remaining 
outer integral has almost the same expression as
$\EE[|\widehat{z_t}(k'')|^2]$ which can be bounded
by $\frac{1}{|k''|^{2-2\gamma}}$. Hence, we have by Lemma \ref{lem:fsum2},
\begin{align*}
	\EE\left[\left|{{I^{(3)}_t}}(k)\right|^2\right]
	& \lesssim \sum_{\substack{k' + k'' = k, \\ k' \sim k''}} 
	\frac{1}{|k'|^{5 - 4\gamma}}\frac{1}{|k''|^{2-2\gamma}}
	\lesssim \frac{1}{|k|^{6 - 6\gamma}}.
\end{align*}
Now consider $\EE[|{{I^{(1)}_t}}(k)|^2]$
and expand out the integrals.
The inner integral becomes the left-hand side the following expression, 
almost the same as $\EE[|\widehat{z_t}(k)|^2]$, 
\[
	\int_0^\infty P_{s - r}(k) P_{s' - r}(-k)q_k^2\,dr 
	\lesssim \frac{1}{|k|^{2 - 2\gamma}}.
\]
With this bound, the next outer integral has
the following expression
\begin{equation} \label{eq:outer}
	\sum_{\substack{k' + k'' = k, \\ k' \sim k''}} \int_0^\infty ik' P_{t - s}(k')
	\int_0^\infty P_{t - r}(k'') P_{s - r}(-k'')q_{k''}^2\,dr,
\end{equation}
whose size is bounded by, using Lemma \ref{lem:fsum2} again, 
\[
	\sum_{\substack{k' + k'' = k, \\ k' \sim k''}}
	\frac{1}{|k'|}\frac{1}{|k''|^{2-2\gamma}} 
	\lesssim \frac{1}{|k|^{2 - 2\gamma}}.
\]
The remaining outer integral also has the form \eqref{eq:outer}. 
We arrive at the result
\begin{align*}
	\EE\left[\left|{I^{(1)}_t}(k)\right|^2\right] 
	\lesssim \left(\frac{1}{|k|^{2 - 2\gamma}}\right)^3 = \frac{1}{|k|^{6-6\gamma}}.
\end{align*}

For the remaining statement, we note that
\[
	\vartheta_t \circ z_t - \vartheta_s \circ z_s
	= \vartheta_t \circ z_{s,t} - \vartheta_{s,t} \circ z_{s}.
\]
We can show \eqref{eq:fest2} of Proposition \ref{prop:regW}, 
by replacing similar bounds used above \eqref{eq:zest1} and \eqref{eq:jzest1} 
with \eqref{eq:zest3} and \eqref{eq:jzest2}.
\end{proof}
Hence, if $\frac{1}{2} \le \gamma < 1$,
$B(J(z^{(\gamma)}), z^{(\gamma)}) \in C_T \Cc^{\beta}$ for any
$\beta < ((2 - 2\gamma) \wedge (\frac{1}{2} - \gamma) 
\wedge (\frac{5}{3} - 3\gamma)) - 1 = -\frac{1}{2} - \gamma$, 
according to Remark \ref{rmk:prod}.

\subsection{Construction of \texorpdfstring{$B(J(z^{(\gamma)}))$}{B(J(z(gamma)))}}
Since $\gamma < 1$, $J(z^{(\gamma)}) \in C_T \Cc^\beta$ for some 
$\beta > 0$. Hence, $B(J(z^{(\gamma)}))$ is classically well-defined.
By Remark \ref{rmk:prod}, 
if $\gamma < \frac{1}{2}$, then $B(J(z^{(\gamma)})) \in C_T \Cc^\beta$
for any $\beta < \frac{1}{2} - \gamma,$
and if $\frac{1}{2} \le \gamma < 1$, then $B(J(z^{(\gamma)})) \in C_T \Cc^\beta$
for any $\beta < 1 - 2\gamma.$

\subsection{Construction of \texorpdfstring{$B(z^{(\gamma)}, z^{(\delta)})$}{B(z(gamma),z(delta))}
  and \texorpdfstring{$J(z^{(\gamma)}, z^{(\delta)})$}{J(z(gamma),z(delta))}}
Recall $z^{(\gamma)}$ and $z^{(\delta)}$ are independent.
For $\gamma + \delta < \frac{3}{2}$, then it suffices to show the existence
of the resonant product $z^{(\gamma)} \circ z^{(\delta)}.$
\begin{proposition} \label{prop:zz}
	Suppose $\gamma + \delta < \frac{3}{2}$. 
	Then we have 
	for any $k \in \Z_0$,
	\[
		\EE\left[\left| \widehat{z^{(\gamma)}_t \circ z^{(\delta)}_t}(k) \right|^2 \right]
		\lesssim \frac{1}{|k|^{3 - 2\gamma - 2\delta}},
	\]
	and 
	$z^{(\gamma)} \circ z^{(\delta)} \in C_T \Cc^{\beta}$
	for any $\beta < 1 - \gamma - \delta.$
\end{proposition}
\begin{proof}
	By the definition of $z^{(\gamma)}_t \circ z^{(\delta)}_t$, 
	independence and \eqref{eq:zest1},
	we have for any $k \in \Z_0$
	\begin{equation} \label{eq:zzest}
		\begin{split}
		\EE\left[\left| \widehat{z^{(\gamma)}_t \circ z^{(\delta)}_t}(k) \right|^2 \right]
		& \lesssim \sum_{\substack{k_1 + k_2 = k, \\ k_1 \sim k_2}}
		\EE\left[\left| \widehat{z^{(\gamma)}_t}(k_1) \right|^2 \right]
		\EE\left[\left| \widehat{z^{(\delta)}_t}(k_2) \right|^2 \right] \\ 
		& \lesssim \sum_{\substack{k_1 + k_2 = k, \\ k_1 \sim k_2}}
			\frac{1}{|k_1|^{2 - 2\gamma}} \frac{1}{|k_2|^{2 - 2\delta}} 
		\, \lesssim \frac{1}{|k|^{3 - 2\gamma - 2\delta}},
		\end{split}
	\end{equation}
	where we used Lemma \ref{lem:fsum2}, 
	since $\gamma + \delta < \frac{3}{2}.$
	Note that
	\[
		z^{(\gamma)}_t \circ z^{(\delta)}_t - z^{(\gamma)}_s \circ z^{(\delta)}_s
		= z^{(\gamma)}_t \circ z^{(\delta)}_{s,t} + z^{(\gamma)}_{s,t} \circ z^{(\delta)}_s.
	\]
	We can show \eqref{eq:fest2} of Proposition \ref{prop:regW} in the same way
	as \eqref{eq:zzest} by using \eqref{eq:zest3} instead. 
	Then the result follows from Proposition \ref{prop:regW}.
\end{proof}

Thus, when $\gamma + \delta < \frac{3}{2}$, 
$B(z^{(\gamma)}, z^{(\delta)}) \in C_T \Cc^{\beta}$ for any
$\beta < ((\frac{1}{2} - \gamma) \wedge (\frac{1}{2} - \delta) 
\wedge (1 - \gamma - \delta)) - 1$, according to Remark \ref{rmk:prod}.

When $\frac{3}{2} \le \gamma + \delta < 2$, we encounter the same 
situation as $(z^{(\gamma)})^{\diamond 2}$ that $z^{(\gamma)} z^{(\delta)}$
is not a well-defined process but a space-time distribution,
so we work on defining the process $J(z^{(\gamma)}, z^{(\delta)})$
as we did for $J(z^{(\gamma)})$.
\begin{proposition}
	Suppose $\frac{3}{2} \le \gamma + \delta < 2$.
	Let $z = z^{(\gamma)}$, $\tilde{z} = z^{(\delta)}$. Then
	we have for any $k \in \Z_0$,
	\begin{equation*}
		\EE\left[\left| \widehat{J(z, \tilde{z})_t}(k) 
		\right|^2 \right] \lesssim \frac{1}{|k|^{5 - 2\gamma - 2\delta}},
	\end{equation*}
	and $J(z, \tilde{z}) \in C_T \Cc^{\beta}$ for any $\beta < 2 - \gamma - \delta.$
\end{proposition}
\begin{proof}
	Let $k \in \Z_0$. 
	By independence and \eqref{eq:zest2}, for any $\rho_1, \rho_2 \ge 0$,
	\begin{align*}
		\EE\left[\widehat{(z \tilde{z})_s}(k)
				\overline{\widehat{(z \tilde{z})_{s'}}(k)} \right]
		& = \sum_{k_1 + k_2 = k} 
			\EE\left[\widehat{z_s}(k_1)
				\overline{\widehat{z_{s'}}(k_1)} \right]
			\EE\left[\widehat{\tilde{z}_s}(k_2)
				\overline{\widehat{\tilde{z}_{s'}}(k_2)} \right]\\
		& \lesssim \frac{1}{|s - s'|^{\rho_1 + \rho_2}} \sum_{k_1 + k_2 = k}
			\frac{1}{|k_1|^{2 - 2\gamma + 2\rho_1}} 
			\frac{1}{|k_2|^{2 - 2\delta + 2\rho_2}}.	
	\end{align*}
	Recall that $\gamma, \delta < 1.$ 
	By taking $\rho_1, \rho_2 \ge 0$ such that 
	\[
		\rho_1 < \gamma - \frac{1}{2}, \quad \rho_2 < \delta - \frac{1}{2},
		\quad \rho_1 + \rho_2 > \gamma + \delta - \frac{3}{2},
	\]
	we can use Lemma \ref{lem:fsum1} to obtain
	\[
		\EE\left[\widehat{(z \tilde{z})_s}(k)
				\overline{\widehat{(z \tilde{z})_{s'}}(k)} \right]
		 \lesssim \frac{1}{|s - s'|^\rho} \frac{1}{|k|^{3 - 2\gamma - 2\delta + 2\rho}},
	\]
	where $\rho = \rho_1 + \rho_2 \in [0, 1].$ Then
	\begin{align*}
		\EE\left[\left| \widehat{J(z, \tilde{z})_t}(k) 
		\right|^2 \right] 
		& \lesssim \int_0^t \int_0^t |k|^2 e^{-|k|^2(t-s)-|k|^2(t - s')}
			\EE\left[\widehat{(z \tilde{z})_s}(k)
				\overline{\widehat{(z \tilde{z})_{s'}}(k)}
		\right] \,ds \, ds' \\
		& \lesssim \int_0^t \int_0^t |k|^2 e^{-|k|^2(t-s)-|k|^2(t - s')}
			 \frac{1}{|s - s'|^\rho} \frac{1}{|k|^{3 - 2\gamma - 2\delta + 2\rho}}
			 \,ds \, ds' \\ 
		& \lesssim \frac{1}{|k|^{5 - 2\gamma - 2\delta}}.
	\end{align*}
	By the same computation in Proposition \ref{prop:jz},
	for any $h \in [0,1]$,
	\begin{align*}
		\EE\left[\left| \widehat{J(z, \tilde{z})_{s,t}}(k) 
		\right|^2 \right] 
		\lesssim \frac{|t - s|^h}{|k|^{5 - 2\gamma - 2\delta - 2h}}.
	\end{align*}
	The result follows from Proposition \ref{prop:regW}.
\end{proof}

\subsection{Construction of \texorpdfstring{$B(J(z^{(\gamma)}), z^{(\delta)})$}{B(J(z(gamma)),z(delta))}}
If $\gamma < \frac{1}{2}$, then $B(J(z^{(\gamma)}), z^{(\delta)})$ is
classically well-defined. Then
$B(J(z^{(\gamma)}), z^{(\delta)}) \in C_T \Cc^{\beta}$ for any
$\beta < ((\frac{3}{2} - \gamma) \wedge (\frac{1}{2} - \delta) 
\wedge (2 - \gamma - \delta)) - 1$, 
according to Remark \ref{rmk:prod}.

If $\frac{1}{2} \le \gamma < 1$,
we need to show the existence
of the resonant product $J(z^{(\gamma)}) \circ z^{(\delta)}.$
\begin{proposition}
	Suppose $\frac{1}{2} \le \gamma < 1$. 
	Let $\vartheta = J(z^{(\gamma)})$, $z = z^{(\delta)}$. 
	We have for any $k \in \Z_0$,
	\[
		\EE\left[\left| \widehat{\vartheta_t \circ z_t}(k) \right|^2 \right]
		\lesssim \frac{1}{|k|^{6 - 4\gamma - 2\delta}},
	\]
	and $\vartheta \circ z \in C_T \Cc^{\beta}$
	for any $\beta < \frac{5}{2} - 2\gamma - \delta.$
\end{proposition}
\begin{proof}
	By the definition of $\vartheta_t \circ z_t$, independence, 
	\eqref{eq:zest1} and \eqref{eq:jzest1},
	we have for any $k \in \Z_0$,
	\begin{align*}
		\EE\left[\left| \widehat{\vartheta_t \circ z_t}(k) \right|^2 \right]
		& \lesssim \sum_{\substack{k_1 + k_2 = k, \\ k_1 \sim k_2}}
		\EE\left[\left| \widehat{\vartheta_t}(k_1) \right|^2 \right]
		\EE\left[\left| \widehat{z_t}(k_2) \right|^2 \right] \\ 
		& \lesssim \sum_{\substack{k_1 + k_2 = k, \\ k_1 \sim k_2}}
			\frac{1}{|k_1|^{5 - 4\gamma}} \frac{1}{|k_2|^{2 - 2\delta}} 
		\,\lesssim \frac{1}{|k|^{6 - 4\gamma - 2\delta}},
	\end{align*}
	where we used Lemma \ref{lem:fsum2}, 
	since $\gamma, \delta < 1$.
	Similar to Proposition \ref{prop:zz}, 
	we can show \eqref{eq:fest2} of Proposition \ref{prop:regW} 
	by using \eqref{eq:zest3} and \eqref{eq:jzest2}. 
	Then the result follows from Proposition \ref{prop:regW}.
\end{proof}
Thus, when $\frac{1}{2} \le \gamma < 1$, 
$B(J(z^{(\gamma)}), z^{(\delta)}) \in C_T \Cc^{\beta}$ for any
$\beta < ((2 - 2\gamma) \wedge (\frac{1}{2} - \delta) 
\wedge (\frac{5}{2} - 2\gamma - \delta)) - 1$, 
according to Remark \ref{rmk:prod}.

\section{Existence and Regularity of Solutions}\label{sec:ExistenceRegularity}

We now prove a local existence theorem by a fixed point argument for the
type of equations needed in this note. This result is quite standard,
and we sketch the argument both for completeness and to highlight
the structure of these equations. We will consider the following
integral equation 
\begin{align}\label{eq:abstractIntEq}
  v_t = e^{-t A} v_0 + c_1J(v)_t + c_2 J(g,v)_t + G_t \eqdef \Phi(v)_t,
\end{align}
where $c_i \in \R$, and for some $T>0$, we have $G \in C_T\Cc^\sigma$, $g\in  C_T\Cc^\gamma$ and $v_0 \in
\Cc^\sigma$ for some $\gamma$ and $\sigma$. 
We will assume that 
\begin{equation}\label{eq:fixptparam}
  \gamma+1>\sigma>0 \quad \text{and} \quad \sigma+\gamma>0 .
\end{equation}
Now for $v^{(1)}, v^{(2)} \in  C_t\Cc^\sigma$ for some $t \in (0,T]$,
we have for $s \in (0,t]$ that
\begin{equation}
 \Phi(v^{(1)})_s -  \Phi(v^{(2)})_s =  c_2 J(g, v^{(1)}-v^{(2)})_s + c_1 J(v^{(1)} +v^{(2)} ,v^{(1)}-v^{(2)})_s.
\end{equation}
Because of our assumptions on $\sigma$ and $\gamma$, we have that
\begin{align*}
  \|J( v^{(2)} ,v^{(1)}-v^{(2)})\|_{C_s\Cc^{\sigma+1}} &\lesssim_s \big(
  \|v^{(1)}\|_{C_s\Cc^\sigma} +  \|v^{(2)}\|_{C_s\Cc^\sigma}
                                                         \big)\|v^{(1)} -v^{(2)}\|_{C_s\Cc^\sigma},\\
  \| J(g, v^{(1)}-v^{(2)})\|_{C_s\Cc^{\gamma+1}}  &\lesssim_s  \|g\|_{C_s\Cc^\gamma} \|v^{(1)} -v^{(2)}\|_{C_s\Cc^\sigma},
\end{align*}
where the dependent constant $s$ in each inequality goes to zero as $s
\rightarrow 0$. Hence, there exists a $K_s$ so that $K_s
\rightarrow0$ as $s\rightarrow0$ and
\begin{align*}
 \| \Phi(v^{(1)})-  \Phi(v^{(2)}) \|_{C_s\Cc^{\sigma}}  \leq K_s
  (  \|v^{(1)}\|_{C_s\Cc^\sigma} +  \|v^{(2)}\|_{C_s\Cc^\sigma}+
  \|g\|_{C_s\Cc^\gamma}   ) \|v^{(1)} -v^{(2)}\|_{C_s\Cc^\sigma} \,.
\end{align*}
Hence, fixing any $R>0$ so that $\|v_0\|_{\Cc^\sigma} +
\|G\|_{C_T\Cc^\sigma} <R$ and 
$\|g\|_{C_T\Cc^\gamma}   < R$, there exists $s>0$ such that $\Phi$
is a contraction on $\{  v \in C_s\Cc^\sigma : \|v\|_{ C_s\Cc^\sigma}
\le R\}$. This implies that there exists a fixed point with $v_t=\Phi(v)_t$
for all $t\in [0,s]$. Since $c_1J(v) + c_2 J(g,v)  \in {
  C_s\Cc^\sigma}$ is well-defined classically for $v \in C_s\Cc^\sigma$ 
  given our assumptions \eqref{eq:fixptparam}, 
  we have proven the following result.
\begin{proposition}[Local Existence and Regularity]\label{prop:local_existence}
  In the above setting with $\gamma+1 > \sigma > 0$ and $\gamma+\sigma>0$, the integral equation \eqref{eq:abstractIntEq}
  has a unique local solution $v$ with $v \in C_s\Cc^\sigma$ for
  some $s>0$.  In particular, if the regularity of the additive forcing
  $G_t$ is set by the stochastic convolution in the equation, then \eqref{eq:abstractIntEq}
 has the canonical regularity in the sense of \cref{def:CanonicalRegularity}.
\end{proposition}
\begin{remark}
  By repeatedly applying the above result we can extend the existence
  to a maximal time $\tau$ such that $\|v_t\|_{\Cc^\sigma} \rightarrow
  \infty$ as $t \rightarrow \tau$ when $\tau <\infty$. By setting $v_t
  = \death$ for all $t \geq \tau$ when $\tau< \infty$, we see that $v
  \in C_T \overline \Cc^\sigma$ for all $T >0$. (See 
  Section~\ref{sec:functionSpaces} for the definition of $C_T
  \overline \Cc^\sigma$ and related discussions.)
\end{remark}
The previous Proposition~\ref{prop:local_existence} can be seamlessly
extended to less regular initial conditions. Assume
\begin{equation}\label{eq:fixptparam2}
  \gamma
    > - \tfrac12,
      \qquad
  \sigma
    > 0
      \qquad \text{and} \qquad
  \sigma+\gamma
    > 0,
\end{equation}
set $\rho=\sigma\wedge(\gamma+1)$, and consider $v_0\in\Cc^{\sigma_0}$,
with
\begin{equation}\label{eq:fixptparam3}
    \sigma_0
      >-1,
        \qquad
    \rho-\sigma_0\in(0,2).
\end{equation}
Set $\theta=\frac12(\rho-\sigma_0)$, and
for $T>0$ define the space
\[
  \mathcal{X}_T^{\sigma_0,\rho}
    = \{u\in C_T\Cc^{\sigma_0}: \sup_{t\in[0,T]}t^\theta\|u(t)\|_{\Cc^\rho}<\infty\},
\]
with norm
\[
  \|u\|_{\mathcal{X}_T^{\sigma_0,\rho}}
    \eqdef \|u\|_{C_T\Cc^{\sigma_0}} + \sup_{t\in[0,T]}t^\theta\|u(t)\|_{\Cc^\rho}.
\]
With the above choice of $\sigma_0$ and under conditions
\eqref{eq:fixptparam2} and \eqref{eq:fixptparam3},
we have that
\[
  \begin{aligned}
    \|t\mapsto e^{-tA}v_0\|_{\mathcal{X}_T^{\sigma_0,\rho}}
      &\lesssim \|v_0\|_{\Cc^{\sigma_0}},\\
    \|J(v^{(1)},v^{(2)})\|_{\mathcal{X}_T^{\sigma_0,\rho}}
      &\lesssim K_T\|v^{(1)}\|_{\mathcal{X}_T^{\sigma_0,\rho}}
        \|v^{(2)}\|_{\mathcal{X}_T^{\sigma_0,\rho}},\\
    \|J(g,v)\|_{\mathcal{X}_T^{\sigma_0,\rho}}
      &\lesssim K_T\|g\|_{C_T\Cc^\gamma}
        \|v\|_{\mathcal{X}_T^{\sigma_0,\rho}},\\
    \|G\|_{\mathcal{X}_T^{\sigma_0,\rho}}
      &\lesssim \|G\|_{C_T\Cc^\sigma},
  \end{aligned}
\]
with $K_T\downarrow0$ as $T\downarrow0$. With these
inequalities at hand, the same proof outlined above
for \cref{prop:local_existence}
can be adapted to this setting, yielding the following
result.
\begin{proposition}
  Consider $\gamma,\sigma$ as in \eqref{eq:fixptparam2},
  $\rho=\sigma\wedge(\gamma+1)$, and $\sigma_0$
  as in \eqref{eq:fixptparam3}, and let $v_0\in\Cc^{\sigma_0}$.
  Then the integral equation~\eqref{eq:abstractIntEq} has
  a unique local solution $v$ in $\mathcal{X}_s^{\sigma_0,\rho}$
  for some $s>0$.
\end{proposition}

The restriction on $\rho$ can be dropped by parabolic regularization,
yielding the following result.
\begin{corollary}
  Consider $\gamma,\sigma$ as in \eqref{eq:fixptparam2},
  $\rho=\sigma\wedge(\gamma+1)$ and $\sigma_0$ such
  that $-1<\sigma_0<\rho$. If $v_0\in\Cc^{\sigma_0}$,
  then the integral equation~\eqref{eq:abstractIntEq} has
  a unique local solution $v$ in $C_s\Cc^{\sigma_0}$
  for some $s>0$ such that
  $v\in C([\epsilon,s];\Cc^\rho)$ for all $\epsilon>0$.
  
  In particular, if $\sigma\leq\gamma+1$, and
   the regularity of the additive forcing
  $G$ is set by a stochastic convolution in the equation,
  then \eqref{eq:abstractIntEq} has the canonical regularity
  in the sense of \cref{def:CanonicalRegularity}
  on every closed interval included in $(0,s]$.
\end{corollary}

\bibliographystyle{amsalpha}
\bibliography{levelsOfNoise.bib}
\end{document}